\documentclass[oneside,reqno,11pt]{amsart}

\usepackage[top=1in, bottom=1.25in, left=1.25in, right=1.25in]{geometry}
\usepackage{amsmath}
\usepackage{amsfonts}
\usepackage{amssymb}
\usepackage{amsthm}
\usepackage{hyperref}
\usepackage{mathrsfs}
\usepackage{caption}
\usepackage{subcaption}
\usepackage{tikz-cd}
\usepackage{multicol}
\usepackage{xypic}

\newtheorem{thm}{Theorem}[section]
\newtheorem{lemma}[thm]{Lemma}

\newtheorem{cor}[thm]{Corollary}
\newtheorem{prop}[thm]{Proposition}
\newtheorem{defn}[thm]{Definition}
\newtheorem{example}[thm]{Example}
\newtheorem{remark}[thm]{Remark}

\newtheorem*{claim*}{Claim}

\newcommand{\Z}{\mathbb{Z}} 
\newcommand{\C}{\mathbb{C}} 
\newcommand{\wt}[1]{\widetilde{#1}} 

\newcommand{\ainf}{A_\infty} 
\newcommand{\F}[1]{\mathcal{F}(#1)} 
\newcommand{\EW}[1]{\mathcal{W}_\Diamond(#1)} 
\newcommand{\cHH}[1]{HH^*_c(#1)} 
\newcommand{\dd}{\mathfrak{d}}
\newcommand{\KK}{\mathfrak{K}}
\newcommand{\qq}{\mathbf{q}}

\newcommand{\Qv}{{\mathcal{Q}^\vee}} 
\newcommand{\Qvv}{{\mathcal{Q}^\vee_0}} 
\newcommand{\Qve}{{\mathcal{Q}^\vee_1}} 
\newcommand{\psurf}{\Sigma^\vee \setminus\Qvv} 
\newcommand{\psurfv}{\Sigma^\vee\setminus\Qvv}
\newcommand{\SH}{SH^*(X_\Qv)} 
\newcommand{\ksmap}{\mathsf{ks}} 
\newcommand{\ks}[1]{\mathsf{ks}(#1)} 

\newcommand{\lL}{\mathbb{L}} 
\newcommand{\Xe}{X_e} 
\newcommand{\xe}{x_e} 
\newcommand{\bXe}{\bar{X}_e} 
\newcommand{\unit}{\mathbf{1}_\lL} 
\newcommand{\pt}{\mathbf{pt}}

\newcommand{\WL}{W_\lL} 
\newcommand{\CF}{CF^*(\lL,\lL)}  
\newcommand{\CFloop}{CF_{cyc}^*(\lL,\lL)} 
\newcommand{\mk}{\left\{m_k\right\}_{k\geq1}} 

\newcommand{\CFb}{CF^*((\lL,b),(\lL,b))} 
\newcommand{\mbb}[1]{m_1^{b,b}\left(#1\right)} 
\newcommand{\mkb}{m_k^{b,\ldots,b}} 
\newcommand{\deltap}{\delta_+} 
\newcommand{\deltam}{\delta_-} 
\newcommand{\MCA}{A_\lL} 
\newcommand{\CFbloop}{CF_{cyc}^*((\lL,b),(\lL,b))} 
\newcommand{\HFbloop}{HF_{cyc}^*((\lL,b),(\lL,b))} 

\newcommand{\Q}{\mathcal{Q}} 
\newcommand{\PM}{\mathcal{P}} 
\newcommand{\Jac}{J} 
\newcommand{\mf}{\mathrm{mf}(W)} 
\newcommand{\Cent}[1]{\mathcal{Z}#1} 
\newcommand{\Loop}[1]{#1_{cyc}} 
\newcommand{\antipos}[1]{\mathcal{O}^+(#1)} 
\newcommand{\antineg}[1]{\mathcal{O}^-(#1)} 
\newcommand{\antipm}[1]{\mathcal{O}^\pm(#1)} 
\newcommand{\cP}{\mathcal{P}} 
\newcommand{\cJ}{\mathcal{J}} 
\newcommand{\MP}{MP\left(\Q\right)} 
\newcommand{\zig}[1]{\mathrm{zig}(#1)} 
\newcommand{\zag}[1]{\mathrm{zag}(#1)} 

\numberwithin{equation}{section}

\allowdisplaybreaks[1]

\begin{document}
\author[Dahye Cho]{Dahye Cho}
\address{G-LAMP AI-Bio Convergence Research Institute \\ Soongsil University \\ 369 Sangdo-ro \\ Dongjak-Gu \\ Seoul 06978 \\ Korea} 
\email{dahye.cho@gmail.com}

\author[Hansol Hong]{Hansol Hong}
\address{Department of Mathematics \\ Yonsei University \\ 50 Yonsei-Ro \\ Seodaemun-Gu \\ Seoul 03722 \\ Korea} 
\email{hansolhong@yonsei.ac.kr}

\author[Hyeongjun Jin]{Hyeongjun Jin}
\address{Department of Mathematics \\ Yonsei University \\ 50 Yonsei-Ro \\ Seodaemun-Gu \\ Seoul 03722 \\ Korea} 
\email{hj\_jin@yonsei.ac.kr}

\author[Sangwook Lee]{Sangwook Lee}
\address{Department of Mathematics and Integrative Institute for Basic Science \\ Soongsil University \\ 369 Sangdo-ro \\ Dongjak-Gu \\ Seoul 06978 \\ Korea} 
\email{sangwook@ssu.ac.kr}

\title[Closed-String Mirror Symmetry for Dimer Models]{Closed-String Mirror Symmetry for Dimer Models}
\begin{abstract}
For all punctured Riemann surfaces arising as mirror curves of toric Calabi–Yau threefolds, we show that their symplectic cohomology is isomorphic to the compactly supported Hochschild cohomology of the noncommutative Landau–Ginzburg model defined on the NCCR of the associated toric Gorenstein singularities.
This mirror correspondence is established by analyzing the closed–open map with boundaries on certain combinatorially defined immersed Lagrangians in the Riemann surface, yielding a ring isomorphism.
We give a detailed examination of the properties of this isomorphism, emphasizing its relationship to the singularity structure. 
\end{abstract}

\maketitle
\tableofcontents

\section{Introduction}

Dimer models are quivers $\mathcal{Q}$ embedded in a Riemann surface $\Sigma$ such that the embedding partitions $\Sigma$ into faces with opposite boundary orientations. 
This structure gives rise to nontrivial relations in the path algebra of $\Q$, resulting in an interesting noncommutative algebra called the  Jacobi algebra  $J=\mathrm{Jac}\,(\Q)$. 
The algebra $\Jac$ comes equipped with a distinguished central element $W$, so that $(\Jac, W)$ defines a \emph{noncommutative Landau--Ginzburg (LG) model}, which we take as our mirror $B$-model throughout.

In particular, when $\mathcal{Q}$ satisfies a suitable consistency condition and $\Sigma$ is a torus, it is known that $\Jac$ provides a  noncommutative crepant resolution (NCCR) of a $3$-dimensional toric Gorenstein singularity $Y_\Q$. 
Hence, $(\Jac, W)$ is equivalent to a Landau--Ginzburg model on a toric crepant resolution $\widetilde{Y_\Q}$ of the same singularity. 
In fact, the potential $\tilde{W}$ on the commutative resolution $\widetilde{Y_\Q}$ is induced by the height function for which all vectors in the toric fan have height~$1$. 
\[
\xymatrix{
(\Jac, W) \ar[rd]_{\mathrm{NCCR}} && (\widetilde{Y_\Q}, \tilde{W}) \ar[ld]^{\mathrm{CCR}} \\
& (Y_\Q, \underline{W}) &
}
\]
Locally, on each toric $\mathbb{C}^3$-chart, the potential is given by the product of three affine coordinates. 

The mirror of the Landau--Ginzburg (LG) model $(\widetilde{Y_\Q}, \tilde{W})$ 
is known to be a hypersurface in $(\mathbb{C}^\times)^2$, namely a punctured 
Riemann surface $X$ whose tropical degeneration can be described combinatorially 
by the fan data of $\widetilde{Y_\Q}$.  
In this paper, we study the mirror symmetry of $X$, incorporating 
the noncommutative LG model $(\Jac, W)$ on the mirror side.  
Specifically, we begin with a dimer model $\mathcal{Q}$ on a torus and 
investigate the Lagrangian Floer theory of the punctured Riemann surface $X$, 
which serves as the mirror $A$-model to the noncommutative LG model $(\Jac, W)$.

Mirror symmetry in this direction---taking the commutative LG model 
$(\widetilde{Y_\Q}, \tilde{W})$ as the $B$-model---has been extensively 
studied in the literature; see, for example, \cite{HLee, CHL_gl, PS}.  
The opposite direction of mirror symmetry, where 
$(\widetilde{Y_\Q}, \tilde{W})$ plays the role of the $A$-model, 
was investigated in~\cite{AA}, where the authors construct and compute 
the (partially wrapped) Fukaya category associated with 
$(\widetilde{Y_\Q}, \tilde{W})$.

On the other hand, the punctured Riemann surface $X$ can, in fact, be obtained combinatorially via \emph{dimer duality}. 
Given a dimer $\mathcal{Q}$, one performs a simple operation---consisting of cutting, flipping, and pasting---to obtain another dimer $\mathcal{Q}^\vee$, called the \emph{dual dimer}, which is embedded in another surface $\Sigma^\vee$ (see \ref{subsec:zigzag Lagrangians}). 
This process is involutive: applying the same operation to $\mathcal{Q}^\vee$ recovers the original dimer $\mathcal{Q}$. 
The punctured surface $X=X_\Qv$ is then obtained by removing the vertices of $\mathcal{Q}^\vee$ from $\Sigma^\vee$. 

Bocklandt~\cite{Bock16} explained the homological mirror symmetry for $X_\Qv$ using this dimer duality, by comparing the wrapped Floer theory of the edges of $\mathcal{Q}^\vee$---which define noncompact Lagrangians in $X_\Qv$ by construction---with certain matrix factorizations of~$W$.

\begin{thm}\cite[Corollary 9.4.]{Bock16} Let $\Q$ be a consistent dimer which admits a perfect matching (see \ref{subsec:pmpmpm}), and  $X_{\Q^\vee}:=\Sigma^\vee \setminus \Q^\vee_0$.Then there is a fully faithful embedding
\[
\mathrm{WFuk}(X_\Qv)\hookrightarrow \mathrm{mf}(\Jac,W).
\]
where the right hand side is the matrix factorization category of the noncommutative Landau-Ginzburg model $(J,W)=(\mathrm{Jac}\,(\Q),W)$.
\end{thm}

Alternatively, the dimer duality can be interpreted through the mirror construction based on Maurer--Cartan deformations of immersed Lagrangians via weak bounding cochains. 
In this approach, we start with the dual dimer $\mathcal{Q}^\vee$. 
Certain cyclic paths in the path algebra of $\mathcal{Q}^\vee$ canonically determine an immersed Lagrangian 
\[
\mathbb{L} = \bigoplus_i L_i
\]
in $X_\Qv$, where each $L_i$ is an immersed circle in $\Sigma^\vee$. 

In~\cite{CHL21}, the second-named author established a local mirror construction by allowing weak bounding cochains (in the sense of~\cite{FOOO09a}) to vary with noncommutative coefficients. 
This yields a noncommutative LG model as the weak Maurer--Cartan moduli space parameterizing Floer-theoretic deformations of~$\mathbb{L}$. 
Applying this construction to $\mathbb{L}$ recovers precisely the same noncommutative LG model $(\Jac, W)$; see Section~\ref{sec:floerpre} for further details on the corresponding formal deformation theory.

One of the advantages of this new interpretation is that it enables us to derive $B$-model invariants of $(\Jac, W)$ directly from the Floer-theoretic data of $\mathbb{L}$, thereby providing a geometric framework for comparing them with $A$-model invariants purely through the Floer theory of $\mathbb{L} \subset X$.  
For instance, the compactly supported Hochschild cohomology $HH_c^\ast(\mf)$---which serves as the closed-string $B$-model invariant of $(\Jac, W)$---can be obtained from a suitable variant of the Floer complex of $\mathbb{L}$, as described below.

In the same spirit as \cite{CT13}, one may alternatively define $HH_c^\ast(\mf)$ as the Hochschild cohomology $HH^\ast(\Jac, W)$ of the curved algebra $(\Jac, W)$, which is computed explicitly in \cite{Wong21} via a spectral sequence argument.  Our first result shows that $HH^\ast(\Jac, W)$ can be realized as a boundary deformation of the Floer complex $CF(\mathbb{L}, \mathbb{L})$, where the (weak) bounding cochain is allowed to vary over the entire (noncommutative) weak Maurer--Cartan space.

\begin{thm} There is a chain-level identification which induces
\[
\HFbloop  \;\cong\; HH^\ast(\Jac,W) \;=\; \cHH{\mf}.
\]
Here, ``cyc'' indicates that $\HFbloop$ carries a natural quiver structure arising from the composability of morphisms between $L_i$ and $L_j$, and we collect cyclic elements with respect to this structure. 
(The precise definition of $\HFbloop$ will be given in Section~\ref{sec:floerpre}.)
\end{thm}

With this in mind, one can naturally construct a geometric map relating $HH^\ast(\Jac,W)$ to the corresponding A-model invariants of $X$, namely, the symplectic cohomology $SH^\ast(X)$.
Indeed, one has a chain-level map
\begin{equation}\label{eqn:ksintro}
SC^\ast(X) \;\longrightarrow\; \CFbloop,
\end{equation}
which can be viewed as an adaptation of the Kodaira--Spencer map of~\cite{FOOO16}, counting pseudo-holomorphic disks with interior insertions from ambient chains in $SC^\ast (X)$. 
The map~\eqref{eqn:ksintro} is closer in spirit to that of~\cite{HJL25}, in the sense that the mirror critical locus is non-isolated and thus involves higher-degree information as well. 
The construction presented here provides a noncommutative generalization of this. 

On the cohomology level, the induced map
\[
\mathsf{KS}: SH^\ast(X) \;\longrightarrow\; HH^\ast(\Jac,W)
\]
is automatically a ring homomorphism, whose proof essentially goes back to~\cite{FOOO16} and relies on degeneration arguments in the TQFT formalism. 
In this paper, we establish the closed-string mirror symmetry between $X$ and $(\Jac,W)$ by showing that $\mathsf{KS}$ is an isomorphism. 
Our main result is as follows.

\begin{thm}\label{thm:closed-string mirror symmetry}
The Kodaira--Spencer map
\[
\mathsf{KS}: SH^\ast(X) \;\xrightarrow{\;\cong\;}\; \cHH{\mf}
\]
is a ring isomorphism.
\end{thm}

The Kodaira--Spencer map $\mathsf{KS}$ considered here differs substantially from the one introduced in~\cite{FOOO16}. 
Even when restricted to the degree-zero component (with respect to the usual $\mathbb{Z}/2$-grading), its image is no longer a scalar multiple of the unit class
\[
\mathbf{1}_{\mathbb{L}} \in \HFbloop = \cHH{\mf}.
\]
In fact, the subring generated by the unit class $\mathbf{1}_{\mathbb{L}} = \sum_i \mathbf{1}_{L_i}$ is strictly smaller than $SH^{\mathrm{even}}(X)$ whenever the associated toric variety $Y_{\mathcal{Q}}$ has singularities along codimension-two strata.  Intuitively, this phenomenon can be understood from the perspective of 
the commutative crepant resolution $\widetilde{Y_\Q} \to Y_\Q$, as illustrated below.

\medskip

\noindent\textbf{The smooth case.} 
Consider first the nonsingular situation, where $X=X_\Qv$ is a pair of pants and $Y_\Q = \mathbb{C}^3$. In this case, $W:Y\to \mathbb{C}$ is given by $W=xyz$.
There is a natural correspondence between three components of conical ends at punctures of $X$ and the three coordinate axes of the mirror $\mathbb{C}^3$, which together form the critical locus of~$W$. 
The subring of $SH^*(X)$ spanned by even and odd $\mathbb{Z}_{>0}$-familes of Reeb orbits encircling each puncture of~$X$ is isomorphic to $\mathbb{C}[x,dx]$, which can be interpreted as the holomorphic de Rham complex of the corresponding coordinate axis in $\mathbb{C}^3$.

\medskip

\noindent\textbf{The singular case.} 
When $Y_\Q$ carries nontrivial singularities, the situation becomes more subtle. In Figure~\ref{fig:exintro}, the surface $X=X_\Qv$ is an elliptic curve with seven punctures, and its mirror is a Landau--Ginzburg model on the NCCR $\Jac$ of the toric Gorenstein singularity $Y_\Q$. 
The toric fan of $Y_\Q$ is the cone spanned by lattice vectors lying in a polygon $P_\Q$ (the pentagon in Figure~\ref{fig:exintro}) contained in the height-one hyperplane (of $N$). 

\begin{figure}[h]
\includegraphics[scale=0.5]{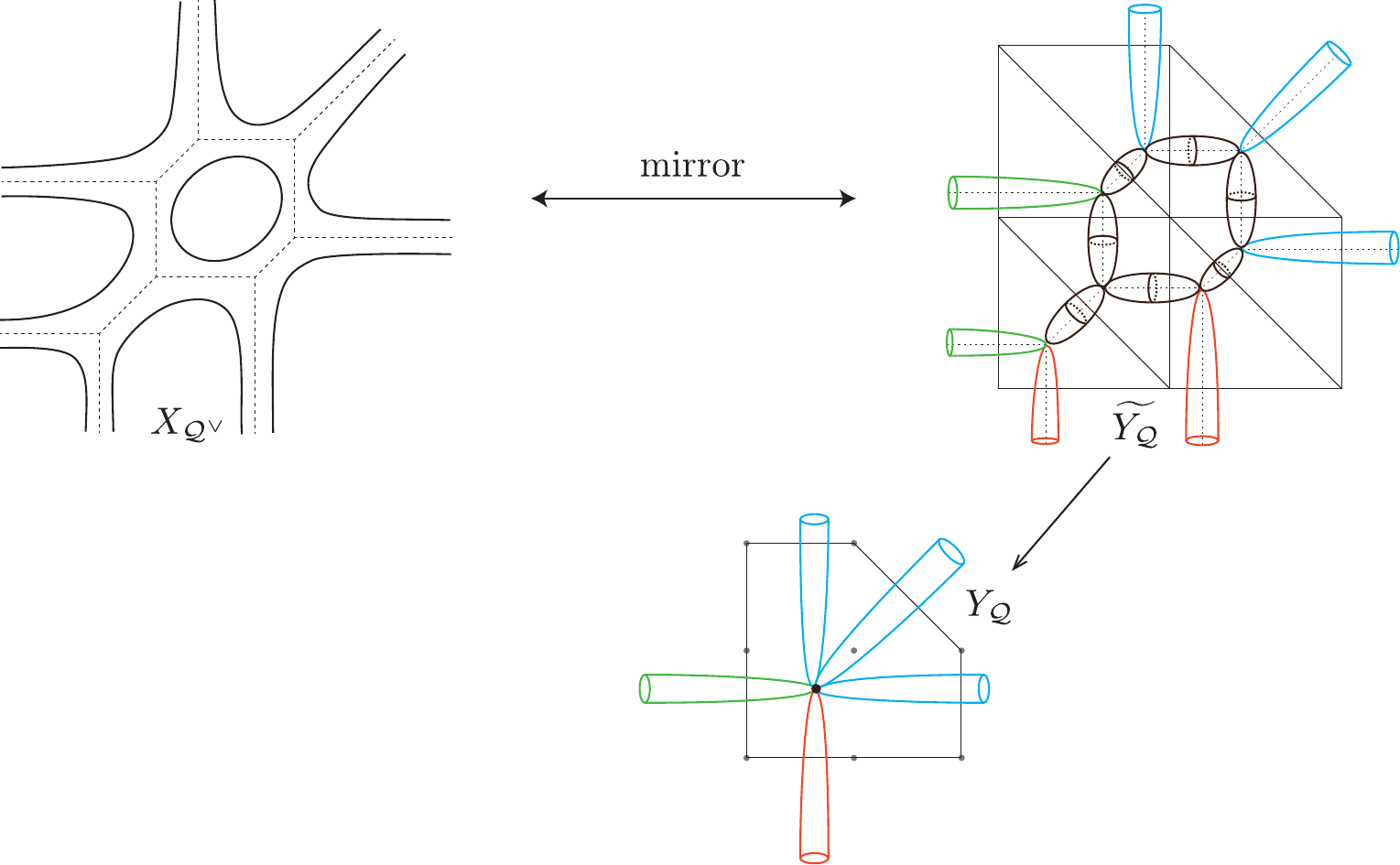}
\caption{A punctured Riemann surface, the mirror toric Gorenstein singularity, and its toric crepant resolution}\label{fig:exintro}
\end{figure}

Unlike the smooth case, $Y_\Q$ possesses only five irreducible codimension-two strata, so two expected components are ``missing.'' The subring generated by the unit class in $ \HFbloop$ can captures these five irreducible components, only.
Examining a toric crepant resolution $\widetilde{Y_\Q} $ obtained by any simplicial decomposition of $P_\Q$,
we find that $\widetilde{Y_\Q}$ indeed contains the expected number of codimension-two strata (each isomorphic to~$\mathbb{C}$); however, under the resolution map, 
\[
\widetilde{Y_\Q} \longrightarrow Y_\Q,
\]
certain components are identified, creating the singularities of $Y_\Q$.

Therefore, by comparing the orbits around the punctures of $X$ with their $\mathsf{KS}$-images, one can identify precisely the additional elements in $HF_{\mathrm{cyc}}^{\mathrm{even}} ((\mathbb{L},b),(\mathbb{L},b))$ that correspond to the singularities along the codimension-two strata of~$Y_\Q$. 
Roughly, these elements are the special classes $\Psi_{i,j} \in HH^{\mathrm{even}}_c(\mathrm{mf}(W))$ appearing in Wong’s computation \cite{Wong21}. 
The odd-degree component of $\mathsf{KS}$ exhibits a similar feature, reflecting the singularity of $Y_\Q$ at the torus fixed point and giving rise to distinguished classes $\Theta_v \in HH^{\mathrm{odd}}_c(\mathrm{mf}(W))$.
This behavior in the odd-degree part is closely related to the fact that the immersed Lagrangian $\mathbb{L}$ has multiple irreducible components, as will be clarified in \ref{subsec:KSmapandsing}.

\begin{remark}
Our previous work \cite{HJL25} deals with the special case where $Y_\Q$ is an orbifold given as an abelian quotient of $\mathbb{C}^3$. 
The classes $\Psi_{i,j}$ and $\Theta_v$ are analogous to the \emph{twisted sectors} of the (closed-string) $B$-model invariants introduced in \cite{HJL25}, which account for the orbifold singularities of the toric orbifold $Y_\Q$ in this case.
\end{remark}

We conclude by noting that the methods developed in this paper may provide a general framework for geometrically relating the closed-string $A$-model invariants of a symplectic manifold $X$ to those of its (local) noncommutative mirror $B$-model $(Y_{\mathbb{L}}, W_{\mathbb{L}})$, constructed from the Maurer--Cartan deformation theory of any given Lagrangian$\mathbb{L}$. 
We conjecture that $\HFbloop$ computes the Hochschild-type invariant of $(Y_{\mathbb{L}}, W_{\mathbb{L}})$. This can indeed be verified in the uncurved case (i.e., when $W_{\mathbb{L}} = 0$), up to certain completion issues (see \cite{HLT}). 
Furthermore, we expect this invariant to coincide with the ``local piece'' of the symplectic cohomology $SH^*(X)$ that can be probed by $\mathbb{L}$ via the open--closed string maps.

\begin{center}
{\bf Notations}
\end{center}

\begin{itemize}
\item $\Q$ : a zigzag consistent dimer embedded in a torus ($B$-model quiver)
\item $\C\Q$ : the path algebra of $\Q$ over $\C$
\item $\C\Q_{cyc} = \C\Q/[\C\Q,\C\Q]$ : the set of cyclic words in $\C\Q$
\item $\Q_2 = \Q_2^+ \cup \Q_2^-$ : the set of faces, positive faces, and negative faces 
\item $\Phi_\Q = \sum_{F\in\Q_2^+} \partial F - \sum_{F\in\Q_2^-} \partial F$ : the spacetime superpotential in $\C\Q_{cyc}$
\item $\Jac = \C\Q/(\partial_e\Phi_\Q:e \in \Q_1)$ : the Jacobi algebra for the dimer $\Q$
\item $\Cent{\Jac}$ : the center of $\Jac$
\item $\Jac_{cyc}$ : cyclic subalgebra of $\Jac$
\item $\Qv$ : the dual dimer obtained from $\Q$ ($A$-model quiver)
\item $\Sigma^\vee$ : the compact Riemann surface where the dimer $\Qv$ is embedded
\item $X=X_{Q^\vee}=\psurf$ : the punctured Riemann surface for the pair $(\Sigma^\vee,\Q_0^\vee)$
\end{itemize}



\section*{Acknowledgements}
The authors would like to thank Siu-Cheong Lau, Ju Tan, and Myungjin Jeong for valuable discussions and insightful comments. The first and second authors were supported by the National Research Foundation of Korea (NRF) grant funded by the Korean government (MSIT) (RS-2025-00517727, 2020R1A5A1016126).
The first and fourth authors were supported by the NRF through the Basic Science Research Program (RS-2021-NR060140) and the G-LAMP Program (RS-2025-25441317), both funded by the Ministry of Education.

\section{Floer theory preliminaries: noncommutative mirror construction}\label{sec:floerpre}
 
In this section, we briefly review the general construction of a (local) noncommutative mirror arising from the algebraic deformation of the $A_\infty$-algebra $\CF$ associated to a compact exact Lagrangian $\mathbb{L}$. 
In our main application, we will employ the pearl trajectory model for $\CF$, whose technical details are provided in~\cite{Sei11}. 
When $\mathbb{L}$ is an immersed Lagrangian with transverse self-intersections, $\CF$ acquires additional generators coming from each self-intersection point, which record the branch jumps between the local sheets of $\lL$ meeting at that point. 
We will give a concrete description of their roles as inputs of the $A_\infty$-operations in Section~\ref{subsec:zigzag Lagrangians} for the case $\dim \mathbb{L} = 1$. 
For more general aspects of immersed Lagrangian Floer theory, see~\cite{AJ10}.

\subsection{Noncommutative deformation of $\CF$ by (weak) bounding cochains}\label{subsec:nc deformation}
Let $(E,\omega=d\lambda)$ be a Liouville manifold, and let $\F{E}$ denote the compact Fukaya category. We fix some possibly immersed, mutually transversal, closed, spin, oriented, unobstructed Lagrangians $L_1,\ldots,L_k \in \mathrm{Ob}(\F{E})$ and let
\begin{equation*}
\lL=\bigoplus_{i=1}^k L_i.
\end{equation*}
Following \cite{AJ10}, its (immersed) Lagrangian Floer complex takes the form
\begin{equation}\label{eqn:immersed Lagrangian Floer complex}
\begin{split}
\CF 
& = \bigoplus_{i=1}^k \left[H^*(L_i;\mathbb{C}) \oplus \bigoplus_{p \in L_i} \mathrm{Span}_{\mathbb{C}}\{(p_-,p_+),(p_+,p_-)\} \right] \oplus \bigoplus_{\substack{i,j=1,\ldots,k\\i \neq j}} CF^*(L_i, L_j)
\end{split}
\end{equation}
where $p\in L_i$ runs over all immersed points of $L_i$, and $p_\pm$ are points in the normalization so that $(p_-,p_+),(p_+,p_-)$ indicate corresponding branch-jumps.
For $H^*(L_i;\mathbb{C})$ in \eqref{eqn:immersed Lagrangian Floer complex}, we take a Morse complex of a perfect Morse function on $L_i$. For later use, we denote by $\mathbf{1}_{L_i} \in C^*(L_i)$ the unit in $H^*(L_i)$. In our setting, $\mathbf{1}_{L_i}$ is simply the maximum of the chosen perfect Morse function on $L_i$.

We have a structure of unital $\ainf$-algebra 
\begin{equation}\label{eqn:ainf algebra for L}
(\CF,\mk)
\end{equation}
over $\mathbb{C}$. Below, we apply the noncommutative base change as in \cite{CHL21} to \eqref{eqn:ainf algebra for L}. Let us first fix a few notations and conventions.

\begin{defn}[{\cite[Definition 6.1]{CHL21}}]\label{def:endomorphism quiver}
We define the quiver $\Q^{\mathbb{L}}$ as follows.  
It has a vertex $v_i$ for each irreducible component $L_i$ of $\mathbb{L}$, and an edge $x$ from $v_i$ to $v_j$ for each odd-degree immersed Floer generator 
$X \in CF^*(L_i, L_j)$.
\end{defn}

The edge $x$ in $\Q^{\mathbb{L}}$ corresponding to $X$ will later be regarded as a formal coordinate function (dual to $X$) on the noncommutative mirror.  
For each such variable $x$ from $v_i$ to $v_j$, we define its \textbf{head} and \textbf{tail} by
\[
h(x) = v_j, \qquad t(x) = v_i.
\]
For each generator $X \in CF^*(L_i, L_j) \subset \CF$, we take the \textbf{head} and \textbf{tail} as
\[
h(X) = v_i, \qquad t(X) = v_j.
\]
For a Morse--Bott generator in $H^*(L_i)$, both the head and the tail are simply $v_i$.

To be more precise, let $\{\Xe\}_{e\in I}$ be the set of all odd degree immersed generators indexed by some set $I$. Consider a set of noncommutative formal variables $\xe$. We label an edge in $\Q^\lL$ by $x_e$ when it corresponds to $X_e$, and hence we have
\begin{equation*}
h(\xe)=t(\Xe), \;\; t(\xe)=h(\Xe).
\end{equation*}
\footnote{This means that we assign an arrow from $i$ (tail) to $j$ (head) for a morphism from $L_j$ to $L_i$. This seems a little unnatural, but is unavoidable since our convention of composing morphisms in the  $A_\infty$-structure follows \cite{FOOO09a} which is opposite to the usual convention of composing functions.}
We assign degrees for them as $|\xe| = 1-|\Xe|$. 
Juxtaposition of these formal variables $x_e$ defines the product structure, for which we use the \emph{(backward) concatenation}
\begin{equation*}
x_{e'} \cdot x_{e} = \left\{\begin{array}{ll} x_{e'} x_{e} & \text{if } t(x_{e'})=h(x_{e}), \\ 0 & \text{otherwise} \end{array}\right.
\end{equation*}
This results in the path algebra of the quiver $\Q^\lL$
\begin{equation*}
K :=\mathbb{C} \Q^\lL =\mathbb{C} \langle x_{e_l}\cdots x_{e_1} :  e_i \in I \rangle
\end{equation*}
generated by words or paths forms by $x_e$'s. The maps $h$ and $t$ have natural extensions to $K$, 
$$h(x_{e_l}\cdots x_{e_1})=h(x_{e_l}),\;\; t(x_{e_l}\cdots x_{e_1})=t(x_{e_1}).$$

Note that $K$ is a bimodule over the semisimple ring $\Bbbk=\mathbb{C} \pi_1 \oplus \cdots \oplus \mathbb{C} \pi_k$
spanned by idempotents $\pi_i$ (length zero paths) with $h(\pi_i) = t(\pi_i) = v_i$ such that
\begin{equation*}
\pi_j \cdot K \cdot \pi_i = R\langle x_{e_l}\cdots x_{e_1}:  h(x_{e_l}\cdots x_{e_1})= v_j,\; t(x_{e_l}\cdots x_{e_1})=v_i\rangle.
\end{equation*}
Similarly, $\CF$ also has a $\Bbbk$-bimodule structure, but with the opposite convention:
\begin{equation*}
\begin{array}{c}
\pi_i \cdot \CF \cdot \pi_j =  CF^*(L_i,L_j).
\end{array}
\end{equation*}
We then make the base change of $\ainf$-algebra $\CF$ to obtain
\begin{equation}\label{eqn:nc tensor}
K \otimes_\Bbbk \CF.
\end{equation}
Therefore, for a word $f \in K$ and a generator $X\in \CF$ is a nontrivial element $fX (=f \otimes X) \in K\otimes_\Bbbk \CF$ only if $h(X)=t(f)$.

The $\ainf$-structure on \eqref{eqn:nc tensor} is a linear extension on $\mk$ in \eqref{eqn:ainf algebra for L}, but it is sensitive to the order in which coefficients are pulled out, due to their noncommutativity.

\begin{defn}[{\cite[Definition 2.6]{CHL21}}]\label{def:nc extension}
Let $f_1 X_1,\ldots,f_k X_k \in K \otimes_\Bbbk \CF$. The $\ainf$-structure $\mk$ on \eqref{eqn:ainf algebra for L} extends to an $\ainf$-structure on \eqref{eqn:nc tensor} by the formula
\begin{equation}\label{eqn:nc ainf structure}
m_k(f_1 X_1,\ldots,f_k X_k) :=  f_k \cdots f_2\cdot f_1 m_k(X_1,\ldots,X_k).
\end{equation}
\end{defn}
It is not difficult to see that $\mk$ in Definition \ref{def:nc extension} satisfies $\ainf$-equations (see \cite[Lemma 2.9]{CHL21}).

Nontrivial generators $f X\in K \otimes_\Bbbk \CF$ such that $h(f) = t(X)$ will be referred to as \emph{cyclic paths} or simply \emph{cycles}. The span of all cyclic paths in \eqref{eqn:nc tensor} is denoted by $\CFloop$. Visually, for a cycle $fX \in K \otimes_\Bbbk CF^*(L_i,L_j)$, the assigned arrows are depicted as
\begin{center}
\begin{tikzcd}
L_i & [-3.1em] \bullet \ar[r, bend left = 24, "f"] & \bullet \ar[l, bend left = 24, "X"] & [-3.1em] L_j.
\end{tikzcd}
\end{center}

\begin{prop}\label{prop:loops form a subalgebra}
$\CFloop$ is an $\ainf$-subalgebra of $K\otimes_\Bbbk \CF$.
\end{prop}
\begin{proof}
Suppose that there is a nontrivial $m_k$-operation taking generators $X_1,\ldots,X_k$ as inputs, and its output involves a generator $Y$. Then we have  $h(X_1)=h(Y)$, $t(X_k)=t(Y)$, and $t(X_i)=h(X_{i+1})$ for $1\leq i < k$. If $f_1,\ldots,f_k\in K$ are words for which $f_i X_i$ are all cycles($1 \le i \le k$), then by Definition \ref{def:nc extension}
\begin{equation*}
m_{k}(f_1 X_1,\ldots,f_k X_k) = \cdots + (f_k  \cdots f_2 \cdot f_1) Y +\cdots =: \cdots+ fY + \cdots.
\end{equation*}
Since $h(f_i)=t(X_i)=h(X_{i+1})=t(f_{i+1})$, $f=f_k \cdots f_2 \cdot f_1$ is not zero, and we see that
\[h(f) = h(f_k) = t(X_k) = t(Y) \;\; \text{and} \;\; t(f) = t(f_1) = h(X_1) = h(Y).\]
Thus the output $fY$ is also a cycle, and the assertion follows by applying the same argument to each generator $Y$ appearing in the output.
\end{proof}

In particular, the following element of $\CFloop$ will play an central role:
\begin{equation*}
b=\sum_{e\in I} \xe\Xe \in  \CFloop.
\end{equation*}
We solve the weak Maurer–Cartan equation for $b$ in order to determine the relations among $\{ x_e \}$ that make $b$ a weak bounding cochain. Specifically, we compute
\begin{equation}\label{eqn:ncmceqn}
m(e^b) = m_1(b) + m_2(b, b) + \cdots,
\end{equation}
and seek for conditions under which the above expression becomes a $K$-linear combination of units. In practice, $m_k(b,\cdots,b)$ in \eqref{eqn:ncmceqn} is by definition given as the sum of terms of the form $m_k (x_{e_1} X_{e_1},\cdots,x_{e_k} X_{e_k})$ which can be computed using the rule \eqref{eqn:nc ainf structure}.
 Since $b$ is in odd degree, we have
\begin{equation*}
m_1(b)+m_2(b,b)+\cdots  = \sum_{i=1}^k W_{L_i} \mathbf{1}_{L_i} + \sum_{f} P_f X_f \in K \otimes_\Bbbk CF^{\mathrm{even}}(\lL,\lL)
\end{equation*}
for some series $P_f$ and $W_{L_i}$ in $x_e$'s, where $X_f$ runs over all even degree generators. 
\begin{remark}\label{rmk:convergenceissue}
In general, the element $P_f$ appearing in the equation is an infinite series, so one typically needs to complete $K$ in order to include $P_f$. The standard approach is to use Novikov coefficients and consider $T$-adic convergence of the series. However, in our setting, this procedure is unnecessary, as all series that arise in our main applications are actually finite. 
See Lemma~\ref{lem:Floer theory coefficient for zigzag Lagrangians}.
\end{remark}

We denote the sum of componentwise units $\mathbf{1}_{L_i}$ by
\begin{equation*}
\unit := \sum_{i=1}^k \mathbf{1}_{L_i} \in \CF,
\end{equation*}
and it serves as the unit of a $\ainf$-algebra $\CF$. We then define
\begin{equation*}
\WL := \sum_{i=1}^k W_{L_i} \in K,
\end{equation*}
and call it the \emph{worldsheet superpotential} or simply the \emph{potential} of $\lL$.
Since $W_{L_i} \cdot \mathbf{1}_{L_j} = \delta_{ij}$, we have
\begin{equation}\label{eqn:weakly unobstructedness}
m_1(b) + m_2(b,b) + \cdots + m_k(b,\ldots,b) + \cdots = \WL \unit
\end{equation}
modulo the relations $P_f \equiv 0$ for each $f$.
More precisely, if we denote by $\MCA$ the quotient of $K$ by the two-sided ideal $(P_f:f\in F)$, and make the further base change
\begin{equation}\label{eqn:Maurer-Cartan base change}
\MCA\otimes_K (K \otimes_\Bbbk \CF) \cong \MCA \otimes_\Bbbk \CF,
\end{equation}
then $b = \sum x_e X_e$ viewed as an element of \eqref{eqn:Maurer-Cartan base change} now solves the weak Maurer-Cartan equation \eqref{eqn:weakly unobstructedness}. The noncommutative algebra $\MCA$ will be called the (weak) Maurer-Cartan algebra of $\mathbb{L}$.

\begin{remark}
When the $\ainf$-structure is cyclic, one can identify $A_\lL$ with the \emph{Jacobi algebra} $J(\Phi_\lL) $ of the quiver $\Q^\lL$ with the  cyclic (spacetime) potential
\begin{equation*}\label{eqn:spacetime superpotential}
\Phi_\lL = \sum_k \frac{1}{k+1} \langle m_k(b,\ldots,b),b\rangle,
\end{equation*}
defined by
\begin{equation*}
J(\Phi_\lL) := K/(\partial_{\xe}\Phi_\lL:e\in I).
\end{equation*}
\end{remark}

We abuse the notation $W$ to also denote its induced element in $A_\lL$, which consists of cyclic paths in $\Q^\lL$.\footnote{The quiver structure on $K$ descends to $\MCA$ since each $P_f$ is homogenous, i.e., all of its summands share the same head and tail.} Then we have

\begin{prop}\label{prop:nclgmirror}\cite[Theorem 3.10]{CHL21}$W$ lies in the center of $A_\lL$. (In this case, $(A_\lL,W)$ is often called a noncommutative Landau-Ginzburg model.)
\end{prop}


We next deform the $A_\infty$-structure on $ \MCA \otimes_\Bbbk \CF$ by weak bounding cochain $b$. It gives rise to a new $A_\infty$ structure $\{m_k^{b,\cdots,b}\}$ defined by
\begin{equation}\label{eqn:defainfty}
\begin{split}
m_k^{b,\ldots,b}(f_1 X_1,\ldots, f_k X_k) &= m(e^b, f_1 X_1, e^b, \ldots, e^b, f_k X_k, e^b)\\
&= \sum_{l_0,\cdots,l_k \ge 0} m_{k+l_0+\cdots+l_k}(b^{\otimes l_0}, f_1 X_1, b^{\otimes l_1}, \ldots, f_k X_k, b^{\otimes l_k}).
\end{split}
\end{equation}
A similar convergence issue potentially arises in the expansion of $m_k^{b, \ldots, b}$, as discussed in Remark~\ref{rmk:convergenceissue}. However, this can again be resolved in our main application due to the finiteness property of $m_k$ operations (see Lemma~\ref{lem:Floer theory coefficient for zigzag Lagrangians}).

\begin{remark}\label{rmk:linearmkovercenter}
Readers are warned that the $A_\infty$-structure $m_k^{b, \ldots, b}$ is not linear over $A_\lL$ since variables are no longer commutative. However, it is still linear over the center of $A_\lL$, which is obvious from \eqref{eqn:nc ainf structure} and \eqref{eqn:defainfty}.
\end{remark}

Note that we now have a nontrivial curvature $m_0^b(1) = m(e^b) = \WL \unit$. Following the standard argument as in \cite[§3.6]{FOOO09a}, this gives us a curved $\ainf$-algebra
\begin{equation}\label{eqn:boundary deformed Maurer-Cartan base change}
\CFb := \left(\MCA \otimes_\Bbbk \CF, \{\mkb\}_{k\geq0}\right)
\end{equation}
over $\MCA$. Obviously, $\CFb$ is unital, and $(m_1^{b,b})^2=0$ since $b$ is a weak bounding cochain.

Analogously to $\CFloop$, one can collect cyclic elements in $\CFb$. 

\begin{defn}\label{def:cflbcyc}
We denote this by $\CFbloop$ a subspace of of $\CFb$ generated by elements $f X ( =f \otimes X) \in \MCA \otimes_\Bbbk \CF$ such that $h(f) = t(X)$.
\end{defn}

Since $b$ itself is cyclic, the same proof as in Proposition \ref{prop:loops form a subalgebra} tells us that $\CFbloop$ is an $A_\infty$ subalgebra of $\CFb$. Moreover, it contains the potential $W_\lL$ and the unit $\unit$ (hence it is unital). We will relate this with some closed-string invariant of the noncommutative LG model $(A_\lL,W)$.

\section{Dimer models and nc mirror symmetry for punctured Riemanns surfaces} 

We briefly recall the theory of dimer models, which are quivers $\Q$ embedded in a Riemann surface $\Sigma$ in a particularly structured way. In particular, Bocklandt \cite{Bock16} provides a combinatorial construction of a noncommutative mirror to the punctured Riemann surface $\Sigma \setminus \Q_0$ using dimer duality. This construction agrees with the noncommutative Landau–Ginzburg model given in Proposition~\ref{prop:nclgmirror}, as will be reviewed in Section~\ref{subsec:zigzag Lagrangians} below. We will also specify the closed-string $B$-model invariant of this noncommutative mirror that we will work with in our main application.

Our exposition will largely follow \cite{Wong21}, including its notation.

\subsection{Dimer models and noncommutative LG mirror}\label{subsec:dimer models}
Dimer models are a special class of quivers that naturally fit into the geometry of punctured Riemann surfaces, and are defined as follows.
\begin{defn}\label{def:dimer models}
A {\bf dimer model} $\Q$ on a compact Riemann surface $\Sigma$ is an embedded quiver
\begin{equation*}
\Q=(\Q_0,\Q_1) \subset \Sigma
\end{equation*}
such that every boundary of each component $F$ of $\Sigma\setminus\Q$ consists of at least 3 quiver edges that are altogether arranged either in the direction of the orientation (positive) of $\partial F$ or in the opposite direction (negative).
\end{defn}


We call a component of $\Sigma\setminus\Q$ as a \emph{face}. We denote the collection of faces by $\Q_2$. Faces of a dimer on a surface $\Sigma$ can be divided into two classes, \emph{positive faces} $\Q_2^+$ and \emph{negative faces} $\Q_2^-$, depending on whether the alignment of their boundaries and quiver arrows is positive or negative. Note that, given a dimer model $\Q$ on a surface $\Sigma$, every edge $e\in\Q_1$ is adjacent to exactly one face in each of $\Q_2^\pm$.

%
%

We define the cyclic (spacetime) potential of $\Q$ by
\begin{equation*}
\Phi_\Q = \sum_{F\in\Q_2^+}\partial F-\sum_{F\in\Q_2^-}\partial F
\end{equation*}
regarded as an element in $\C\Q_{cyc} = \C\Q/[\C\Q,\C\Q]$, where 
$$\C\Q:=\mathbb{C} \langle x_{e_k} \cdots x_{e_1} : h(x_{e_i}) = t(x_{e_{i+1}}), e_i \in \Q_1 \rangle$$ 
is the path algebra of $\Q$. Then 
the Jacobi algebra $J=\mathrm{Jac}\,(\Q)$ for $(\Q, \Phi_\Q)$ is defined by
\begin{equation*}
J=\mathrm{Jac}\,(\Q,\Phi_\Q) = \C\Q/\left(\partial_{x_e} \Phi : e \in \Q_1\right)
\end{equation*}
where $\partial_{x_{e}}$ is  the cyclic derivative at $e\in\Q_1$
\begin{equation*}
\partial_{x_e}=\frac{\partial}{\partial x_{e}} : \C\Q_{cyc} \to \C\Q, \qquad x_{a_k} \cdots x_{a_1} \mapsto \sum_{a_i=e} x_{a_{i-1}}\cdots x_{a_1} x_{a_k} \cdots x_{a_{i+1}}.
\end{equation*}
The relation $\partial_{x_{e}} \Phi =0$ in $\Jac$ is ofter referred to as the Jacobi relation.

%
%

There is a special central element $W \in \Cent{\Jac}$ which is defined as follows. To each $v\in\Q_0$, we fix an arbitrary face $F_v\in\Q_2$ that has $v$ in its corner. Let $W_v\in\Jac$ be the image of a boundary word of $F_v$ that is starting and ending at $v$. It is easy to see that $W_v$ is independent from the choice $F_v$ due to Jacobi relations. We set
\begin{equation*}
W = \sum_{v\in\Q_0}W_v
\end{equation*}
and call it the potential of $\Jac$. One can show that $W$ lies in the center of $\Jac$.

Cyclic paths in $\Jac$ form a subalgebra, defined by
\begin{equation*}
\Jac_{cyc} = \bigoplus_{v\in\Q_0} v \cdot \Jac \cdot v \subset \Jac,
\end{equation*}
where $v$ denote the idempotent (length 0 path) at $v$. It is known that $\Jac_{cyc}$ is a semisimple algebra over the center of $\Jac$, since every cyclic path in $\Jac$ can be uniquely extended to a central element whose component at each vertex $v$ is the original cyclic path. (In particular, its rank as a semisimple algebra equals the number of vertices in $\Q$.)

%

\subsection{Zigzag paths and consistency}
We now recall the zigzag cycles in a dimer $\Q$. If two consecutive edges $a'a$(i.e., $t(a') = h(a)$) lies on the boundary of a positive face, they are called a \emph{zigzag}. ($a$ is a \emph{zig} and $a'$ is a \emph{zag}.) Likewise, if $a'a$ lies on the boundary of a negative face, they are called a \emph{zagzig}, where in this case $a$ is a \emph{zag} and $a'$ is a \emph{zig}.
Fix a universal cover $\wt{\Sigma} \to \Sigma$ and let $\wt{\Q}$ be a lifting of the quiver $\Q$ in $\wt{\Sigma}$, which is obviously a dimer model $\wt{\Sigma}$. We write $\tilde{e} \in \wt{\Q}_1$ for  any lift of an edge $e \in \Q_1$.

\begin{figure}[h]
\includegraphics[scale=0.2]{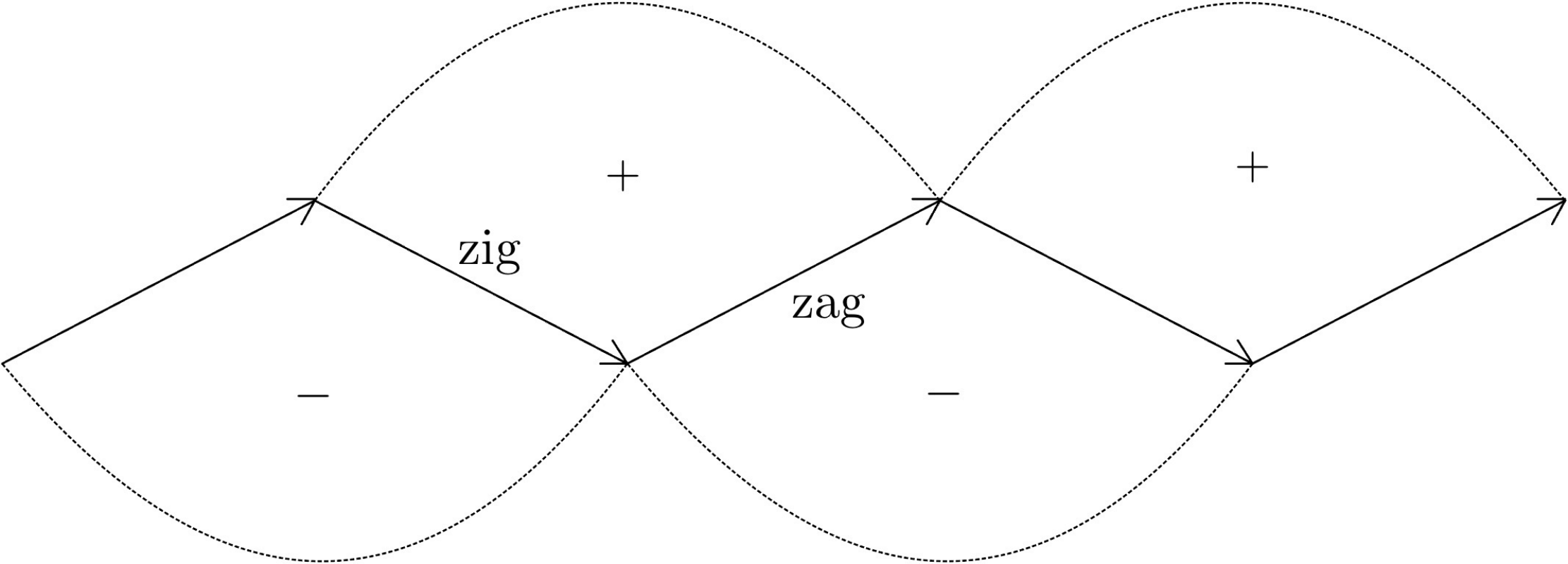}
\caption{Zigs and zags}
\label{fig:zigs and zags}
\end{figure}

\begin{defn}\label{def:zig and zag rays}
A {\bf zig ray} (resp. a {\bf zag ray}) at $\tilde{e}$ is a path
\begin{equation}\label{eqn:zigzag rays}
\cdots \tilde{b}_l \tilde{a}_l \cdots \tilde{b}_1 \tilde{a}_1
\end{equation}
in $\wt{\Q}$ such that $\tilde{a}_1=\tilde{e}$ and $\tilde{b}_l \tilde{a}_l$ are zigzags (resp. zagzigs), $\tilde{a}_{l+1}\tilde{b}_l$ are zagzigs (resp. zigzags) for all $l \geq 1$.
\end{defn}
We will focus on dimers whose zigzag rays satisfy the following condition, which guarantees several desirable algebraic properties of $\Jac$.

\begin{defn}\label{def:zigzag consistency}
A dimer $\Q$ is called to be {\bf zigzag consistent} if for every edge $e\in\Q_1$ the zig and zag rays in $\wt{\Q}_1$ that are emanating from $\tilde{e}$ does not intersect in edges other than the initial edge $\tilde{e}$.
\end{defn}

%
%
%
%

As $\Q$ is finite, any zig or zag ray \eqref{eqn:zigzag rays} projects down to a periodic path
\begin{equation}\label{eqn:projection of zigzag rays}
\cdots b_l a_l \cdots b_1 a_1
\end{equation}
in $\Q$. The shortest repeat in \eqref{eqn:projection of zigzag rays} gives rise to a closed path of even length, say $2L$,
\begin{equation}\label{eqn:zigzag cycles}
Z=b_L a_L \ldots b_1 a_1.
\end{equation}
It will be referred to as a \emph{zigzag cycle} in $\Q$. Indeed, zigzag cycles give rise to cycles in $\Sigma$ from the embedding $\Q\subset\Sigma$. It is often convenient to regard zigzag cycles as cyclic words in $\C\Q_{cyc}$ so that
\begin{equation*}
Z = a_1 b_L a_L \cdots a_2 b_1 = b_1 a_1 b_L \cdots b_2 a_2 = \cdots = a_L b_{L-1} \cdots a_1 b_L \in \C\Q_{cyc}
\end{equation*}
We will always use $\{a_l\}$ to denote the set of zigs whenever a zigzag cycle is represented in the form \eqref{eqn:zigzag cycles}.

A zigzag cycle $Z$ can be assigned with a homology class $[Z] \in H_1(\Sigma;\Z)$. If $\Q$ is zigzag consistent, $[Z]$ is always a nonzero generator in $H_1(\Sigma;\Z)$ known to be primitive, meaning that it is a part of some $\Z$-basis. 
We will focus particularly on consistent dimer models embedded in a torus.  
In this setting, the following proposition characterizes the intersection behavior of zigzag cycles in terms of their homology classes.

\begin{prop}[{\cite{Bro12}, \cite{Bock12}}]\label{prop:intersection of zigzag cycles}
Let $\Q$ be a zigzag consistent dimer on a torus $T^2$ and $Z_1$ and $Z_2$ be two distinct zigzag cycles in $\Q$.
\begin{enumerate}
\item[(i)] If $[Z_1]$ and $[Z_2]$ are linearly independent homology classes in $H_1(T^2;\Z)$, then there exists at least one edge $e\in \Q$ that appear in both $Z_1$ and $Z_2$.
\item[(ii)] If $[Z_1]=[Z_2]$, then they do not intersect.
\end{enumerate}
\end{prop}

Nonintersecting zigzag cycles will play an important role later.

\begin{defn}\label{defn:parallels}
For a zigzag consistent dimer $\Q$ on a torus, we say that two zigzag cycles are {\bf parallel} if they represent the same homology class.
\end{defn}

This terminology is justified by property (ii) of Proposition \ref{prop:intersection of zigzag cycles}.


Next, we introduce the notion of anti-zigzags. Let $Z$ be a zigzag cycle in a dimer $\Q$, and consider all the faces in $\Q_2^+$ adjacent to $Z$. Each zigzag in $Z$ forms part of the boundary of a face in $\Q_2^+$. Since these faces share the same sign, they cannot share common edges. The boundary paths of these faces complementary to $Z$ form a cycle.
We denote this by $\antipos{Z}$ and call it a \emph{positive anti-zigzag cycle} of $Z$. The cycle analogously defined from negative adjacent disks of $Z$ is denoted by $\antineg{Z}$ and is called \emph{negative anti-zigzag cycle}.

\begin{figure}[h]
\includegraphics[scale=0.2]{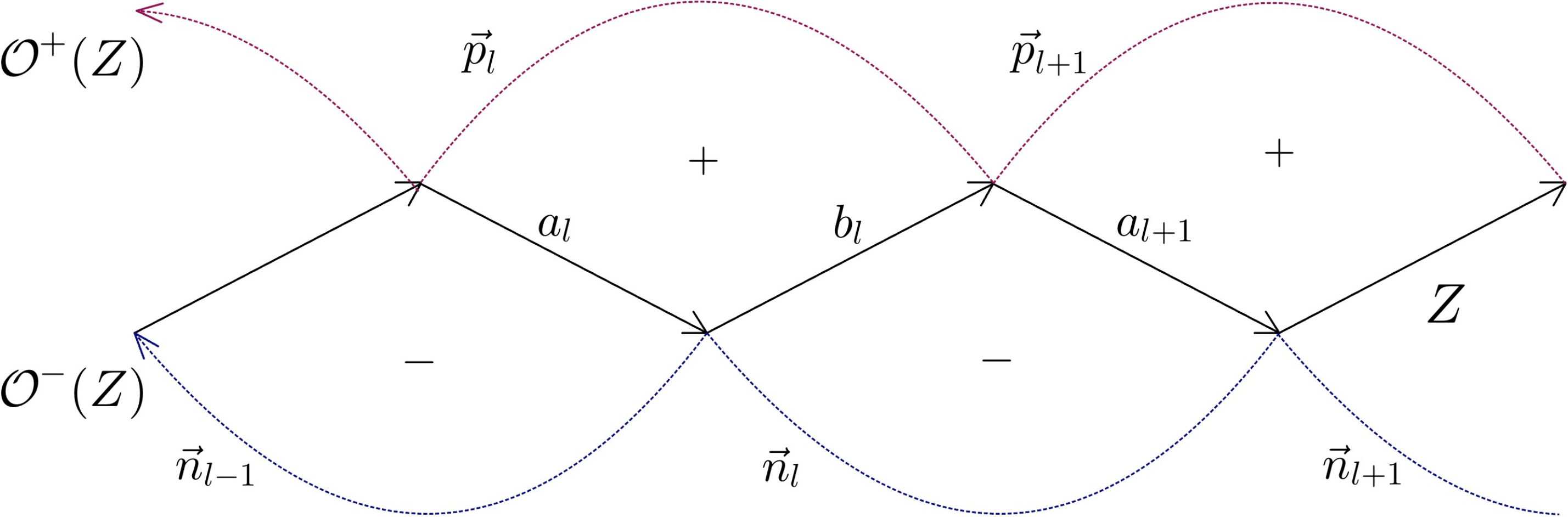}
\caption{A zigzag cycle $Z$ and its two anti-zigzags $\antipm{Z}$}
\label{fig:anti-zigzags}
\end{figure}

When zigzag cycle $Z$ is given as in \eqref{eqn:zigzag cycles}, we write
\begin{equation}\label{eqn:anti-zigzags}
\antipos{Z} = \vec{p}_1 \cdots \vec{p}_L \;\, \text{and} \;\; \antineg{Z} = \vec{n}_1 \cdots \vec{n}_L
\end{equation}
for anti-zigzags $\antipos{Z}$ and $\antineg{Z}$ (as elements in $\C\Q_{cyc}$). The indices are arranged in such a way that $\vec{p}_l b_l a_l$ are boundary words from positive disks adjacent to $Z$ and $\vec{n}_l a_{l+1} b_l$ are boundary words from negative disks adjacent to $Z$. 
Since concatenating $\antipm{Z}$ with $Z$ results in contractible loops in $\Sigma$, we see that $[\antipm{Z}]=-[Z]$ in $H_1(\Sigma;\Z)$.

\subsection{Perfect matchings}\label{subsec:pmpmpm}
Perfect matchings on a dimer is not only a tool to grade the Jacobi algebra $\Jac$, but also provides a strong link between theory of dimers and NCCR (noncommutative crepant resolution) of a 3 dimensional toric Gorenstein singularities.

\begin{defn}
A set $\PM$ of edges in $\Q$ is called a {\bf perfect matching} if for any face $F \in \Q_2$ there is exactly one edge $e \in \PM$ such that $e \in \partial F$. 
\end{defn}

Any perfect matching $\PM$ defines an obvious map 
\begin{equation*}
\deg_\PM: \Q_1 \to \mathbb{Z},  \qquad e \mapsto 
\left\{\begin{array}{l}
1 \quad e \in \PM \\
0 \quad  e \notin \PM,
\end{array}\right.
\end{equation*}
which can be easily shown to induce 
\begin{equation}\label{eqn:degpm}
\deg_\PM:\Jac \to \mathbb{C}.
\end{equation}
Observe that $\Q$ fixes a cellular decomposition of $\Sigma$, which is just $(\Q_0,\Q_1,\Q_2)$. Let us consider the associated cellular cochain complex
\begin{equation*}
\begin{tikzcd}[column sep = small]
0 \rar & \hom(Q_0, \Z) \rar{\partial} &  \hom(Q_1, \Z)  \rar{\partial} &  \hom(Q_2, \Z)  \rar & 0.
\end{tikzcd}
\end{equation*}
The coboundary map can be explicitly written as defined by $v^* \mapsto \sum_{h(e)=v} e^* - \sum_{t(e)=v} e^*$ and $e^* \mapsto \sum_{e\in\partial F} F^*$ in terms of dual basis elements. 
%
%
If we identify $\PM \subset \Q_1$ with a cellular 1-cochain $\sum_{e\in\PM} e^*$, then $\PM$ is a perfect matching if and only if we have
\begin{equation*}
\partial \PM = \sum_{F\in\Q_2} F^*.
\end{equation*}
Therefore the difference of two perfect matching defines an element in $H^1(\Sigma,\mathbb{Z}))$, and the set
%
\begin{equation}\label{eqn:matching lattice}
\{[\PM - \PM_o]:\PM \in PM(\Q)\} \subset H^1(\Sigma;\Z)
\end{equation}
for some fixed perfect matching $\PM_o$ represents a finite subset in $H^1(\Sigma;\Z)$. 

When $\Sigma\simeq T^2$, we have $H^1(T^2; \mathbb{Z}) \cong \mathbb{Z}^2$, and \eqref{eqn:matching lattice} defines a finite subset of this rank-2 lattice. The convex hull of \eqref{eqn:matching lattice} in the plane $H^1(\Sigma; \mathbb{Z}) \otimes \mathbb{R} \cong \mathbb{R}^2$ is called the \emph{matching polytope} of $\Q$, denoted by $\MP$. We classify perfect matchings as \emph{corner}, \emph{boundary}, or \emph{internal}, depending on the location of $[\PM - \PM_o]$ within $\MP$. These types reflect the behavior of perfect matchings with respect to zigzag cycles.
More specifically, let $\Q$ be a zigzag consistent dimer on a torus, and let
\begin{equation}\label{eqn:zigzag homology classes}
-\eta_1, \ldots, -\eta_N \in H_1(\Sigma; \mathbb{Z})
\end{equation}
denote the homology classes associated to the zigzag cycles. We assume these are ordered according to their counterclockwise cyclic order in the lattice $H_1(\Sigma; \mathbb{Z}) \cong \mathbb{Z}^2$.\footnote{The negative signs are chosen so that $\eta_1, \ldots, \eta_N$ correspond to the homology classes of the associated anti-zigzags.}
%

\begin{thm}[{\cite[§3.3, §3.4]{Gul08},\cite[Theorem 1.47]{BockABC}}]\label{thm:combinatorics of matching polytope for dimers in tori}
Suppose $\Q$ is a zigzag consistent dimer on a torus and $-\eta$ be one among \eqref{eqn:zigzag homology classes}.
\begin{enumerate}

\item[(i)] $\MP$ is a lattice polytope with $N$ edges, and the vectors $-\eta_1, \ldots, -\eta_N$ are the outward-pointing normals to these edges. 
\item[(ii)] Each corner of $\MP$ is represented by a unique perfect matching.
\item[(iii)] For $-\eta_i$, let $\PM_i,\PM_{i+1} \in MP(\Q)$ be two consecutive corner matchings joined by the edge normal to $-\eta_i$. Then we have
\begin{equation*}
\PM_i = \cJ_{\eta_i} \cup \bigcup_{j=1}^{m_i} \zig{Z_{i,j}}, \quad \PM_{i+1} = \cJ_{\eta_i} \cup \bigcup_{j=1}^{m_i} \zag{Z_{i,j}}.
\end{equation*}
where $Z_{i,1},\ldots,Z_{i,m_i}$ are the parallel zigzag cycles in the class $-\eta_i$, and $\cJ_{\eta_i} := \PM_i \cap \PM_{i+1}$ does not intersect $Z_{i,j}$'s.
\item[(iv)] The set 
\begin{equation}\label{eqn:boundary matchings}
\left\{\PM_I=\cJ_{\eta_i}\cup \bigcup_{j \in I} \zig{Z_{i,j}} \cup \bigcup_{j \notin I} \zag{Z_{i,j}} \in PM(\Q): I \subseteq \{1,\ldots,m_i\}\right\}
\end{equation}
exhausts the boundary matchings lying between $\PM_i$ and $\PM_{i+1}$.
\end{enumerate}
\end{thm}

\begin{figure}[h]
\includegraphics[scale=0.25]{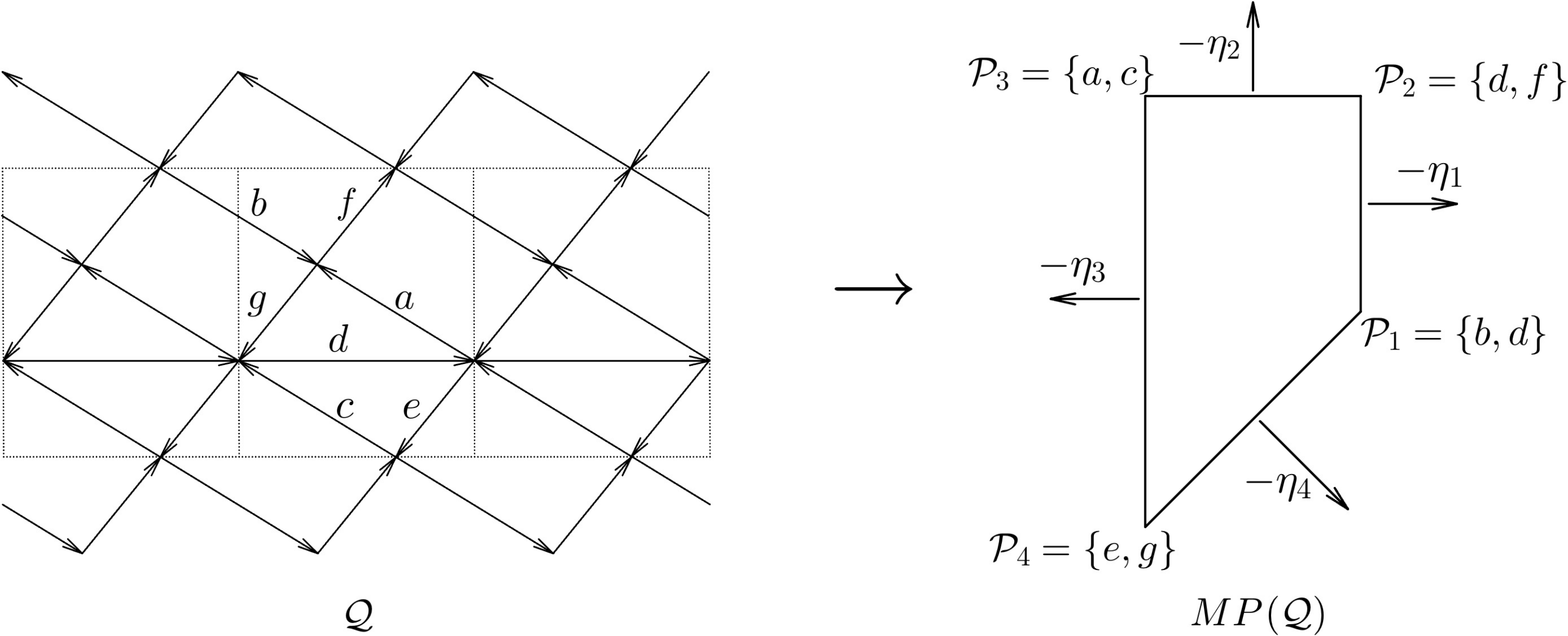}
\caption{A suspended pinchpoint $\Q$ and its matching polytope $\MP$}\label{fig:matchingpolytope}
\end{figure}

On the other hand, one can associate a 3-dimensional affine toric Gorenstein singularity $Y_\Q$ to the lattice polygon $\MP$. The fan of $Y_\Q$ consists of a single cone, constructed by placing $\MP$ in the hyperplane $\{z = 1\} \subset \mathbb{R}^3_{(x,y,z)}$ and taking the cone over $\MP$ with vertex at the origin $(0,0,0)$. The singular locus of $Y_\Q$ lies along codimension-2 strata corresponding to edges of $\MP$ that contain interior lattice points.


\begin{thm}[{\cite[Theorem 3.20]{BockABC}}]\label{thm:nccrmpmp}
Let $\Q$ be a zigzag consistent dimer on a torus. In the sense of \cite{vdB04}, $\Jac$ is the NCCR of its center $\Cent{\Jac}$ which is the coordinate ring of the affine toric Gorenstein threefold $Y_\Q$ defined by the matching polytope $\MP$.
\end{thm}

Moreover Broomhead proved that every affine toric Gorenstein singularity admits an NCCR from a consistent dimer model on a torus \cite[Theorem 8.6]{Bro12}.
%

\begin{remark}\label{rmk:pmdegandvan}
By Theorem~\ref{thm:nccrmpmp}, each corner matching $\mathcal{P}_i$ corresponds to a torus-invariant divisor in $Y_{\Q}$. For a central element $f$, its degree $\deg_{\mathcal{P}_i}(f)$ (as defined in~\eqref{eqn:degpm}) with respect to $\mathcal{P}_i$ measures the vanishing order of $f$, viewed as a regular function on $Y_{\Q}$, along the toric divisor associated to $\mathcal{P}_i$. This is essentially a reinterpretation of the proof of \cite[Theorem 3.20]{BockABC} in the language of toric geometry.
\end{remark}

\begin{example}\label{ex:dimer}
Figure~\ref{fig:matchingpolytope} depicts a dimer model in a torus (that resolves the toric Gorenstein singularity called a suspended pinchpoint), which has three vertices and six edges. The fundamental domain of the torus is given by dotted lines. There are five zigzag cycles, namely,
$$fb,\;adcf,\;ec,\;ga,\;dgbe$$
in four different homology classes:
$$-\eta_1 = [fb],\; -\eta_2 = [adcf],\; -\eta_3 = [ec]=[ga] (=-[fb]),\; -\eta_4 = [dgbe].$$
(Hence $N=4$.) From the picture and Theorem \ref{thm:combinatorics of matching polytope for dimers in tori}, one can find the matching polytope and corner matchings $\cP_i$. 
There are two boundary matchings $\{a,e\}$ and $\{c,g\}$, both sitting at the unique lattice point between $\cP_3$ and $\cP_4$.
\end{example}

\subsection{Zigzag Lagrangians in the dual dimer}\label{subsec:zigzag Lagrangians}
We now turn to the symplectic geometry associated with the dual dimer $\Qv$ of $\Q$. Following \cite[§10.1]{CHL21}, we recall how to construct immersed circles, known as \emph{zigzag Lagrangians}, from zigzag cycles in $\Qv$. We begin with the definition of dual dimers.

\begin{defn}\label{def:the dual dimer}
Suppose $\Q = (\Q_0,\Q_1) \subset \Sigma$ is a dimer model. We define a quiver
\begin{equation*}
\Qv = (\Qvv,\Qve)
\end{equation*}
as follows. First, we set $\Qvv$ to be the set of all zigzag cycles in $\Q$ and $\Qve=\Q_1$. Each $e\in\Qve$ has the head and tail $t(e) = v_{Z_1}$ and $h(e) = v_{Z_2}$ where $Z_1$ and $Z_2$ are zigzag cycles obtained from the projections of zag and zig rays at some lift $\tilde{e} \in \wt{\Q}$, respectively. It is called the {\bf dual dimer} of $\Q$.
\end{defn}

We orient the arrows $\Qve$ in the dual dimer $\Qv$ according to the orientations of the arrows in $\Q_1$.  
It is not difficult to see that the boundary cycle of a positive face in $\Q$ corresponds to a cycle in $\Qv$, along which we attach a positive face. The boundary cycle of a negative face in $\Q$ appears in $\Qv$ with reversed orientation, and we attach a negative face along this reversed cycle. Let $\Sigma^\vee$ denote the resulting Riemann surface, into which $\Qv$ embedds.
Thus, there is a one-to-one correspondence between the faces of the two dimers that preserves the set of boundary labels (though not the orientation in the case of negative faces). Obviously,

\begin{lemma}\label{lem:perfect matchings are the same}
Given a dimer model $\Q$, a set of edges $\PM \subset \Q_1=\Qve$ is a perfect matching of $\Q$ iff it is a perfect matching of $\Qv$.
\end{lemma}
%
Figure \ref{fig:dualtozigzag} shows the arrangement of faces of $\Q^\vee$ near a vertex $v_Z \in \Q^\vee_0$ which is dual to zigzag cycle $Z$ in $\Q$. It is worthwhile to compare this with their corresponding faces in $\Q$ drawn in Figure \ref{fig:anti-zigzags}.
\begin{figure}[h]
\includegraphics[scale=0.3]{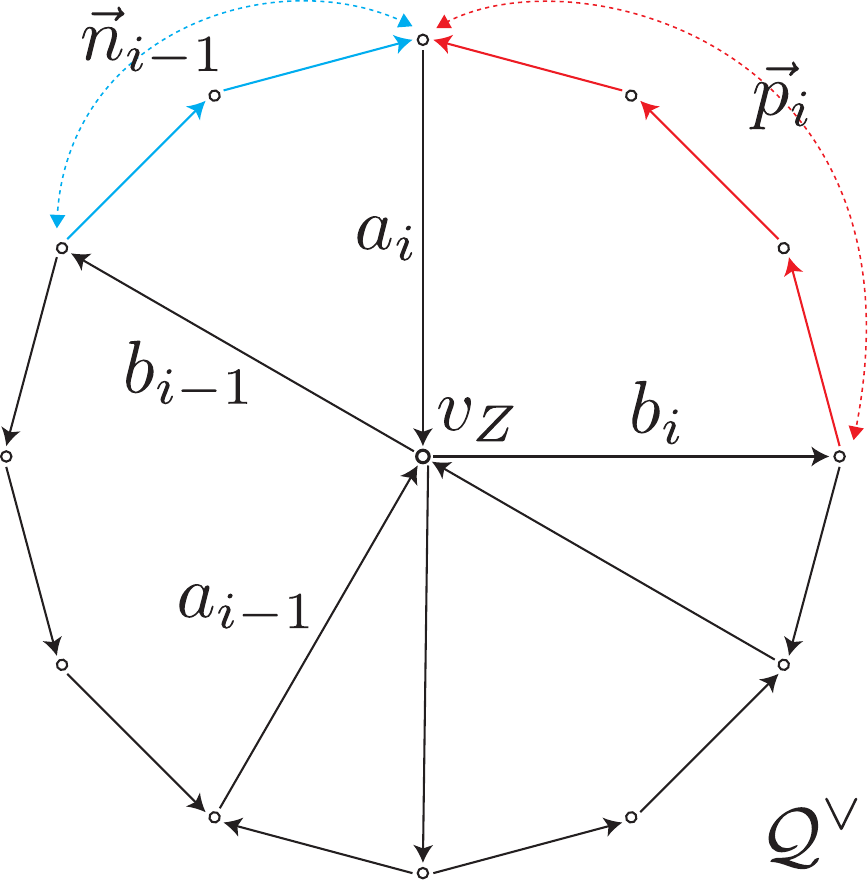}
\caption{Neighboring faces near the vertex $v_Z \in \Q^\vee_0$ dual to the zigzag path $Z$ in $\Q$ }
\label{fig:dualtozigzag}
\end{figure}

Removing $\Qvv$ from $\Sigma^\vee$ gives a punctured Riemann surface, which we denote by $X_\Qv := \psurfv$. We equip $X_\Qv$ with a Liouville structure such that near each puncture, the Liouville vector field points outward toward the puncture. As a result, $X_\Qv$ has conical ends modeled on $S^1 \times \mathbb{R}$ at each puncture.

%
%

We now introduce the \emph{zigzag Lagrangians} $\lL = \bigoplus_{i=1}^k L_i$. Let $Z_i^\vee$ be a zigzag cycle in the dual dimer $\Qv$. For each such cycle, we select the midpoints of all edges contained in $Z_i^\vee$ and connect consecutive midpoints by segments lying inside the face enclosed by the corresponding pair of edges (see Figure~\ref{fig:a zigzag Lag}). The union of these segments, after a slight smoothing at the corners, forms a closed curve. Since this curve avoids the punctures, it defines a Lagrangian circle $L_i$ in $X_\Qv$.

Repeating this construction for each zigzag cycle yields an immersed Lagrangian $\lL = \bigoplus_{i=1}^k L_i$ in $X_\Qv$. By construction, the components of $\lL$ correspond bijectively to the zigzag cycles in the dimer $\Qv$. Moreover, by applying a suitable smooth isotopy, these circles can be arranged to be exact Lagrangians \cite[Lemma 10.6]{CHL21}.

\begin{figure}[h]
\includegraphics[scale=0.25]{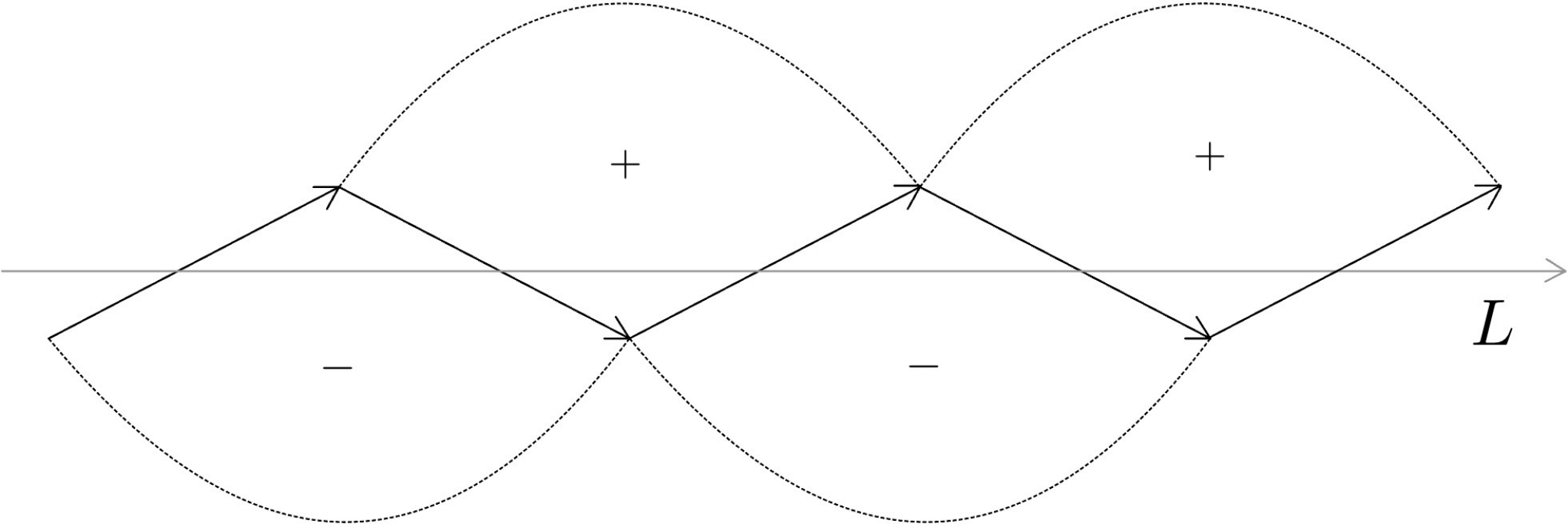}
\caption{A zigzag Lagrangian component $L$}
\label{fig:a zigzag Lag}
\end{figure}



We now describe the quiver $\Q^\lL$ introduced in Definition~\ref{def:endomorphism quiver}, in the context of the zigzag Lagrangians $\lL$. There is a one-to-one correspondence between the vertices of $\Q^\lL$ (that is, the components $L_i$) and the zigzag cycles in $\Qv$, which in turn correspond to the vertices of the original quiver $Q = (\Qv)^\vee$.
It is easy to see that there is an arrow $x$ from vertex $L_i$ to $L_j$ in $\Q^\lL$ whenever the zigzag cycles $Z_i$ and $Z_j$ intersect (it is possible that $i=j$). 

Let \( X_e \) denotes the intersection point of \( L_i \) and \( L_j \) corresponding to an edge \( e \in \Q^\vee_1 \). Then we obtain two immersed generators associated with this intersection: \( X_e \), of odd degree, and \( \bar{X}_e \), of even degree. These represent two distinct types of branch jumps between \( L_i \) and \( L_j \), as illustrated in Figure~\ref{fig:zigzagdisks}. When they appear as inputs of the \( m_k \)-operations, the boundary of a contributing holomorphic polygon turns at the corners formed by the pair \( (L_i, L_j) \) (at $X_e$).


Moreover, the correspondence between \( \Q^{\mathbb{L}} \) and \( \Q \) identifies the two algebras \( A_{\mathbb{L}} \) (constructed on \( \Q^{\mathbb{L}} \)) and \( \Jac \).


\begin{prop}[{\cite[Lemma 10.13 and 10.21, Corollary 10.16 and 10.20]{CHL21}}]\label{lem:spacetime superpotential}
The quiver $\Q^{\mathbb{L}}$ is canonically identified with $\Q$. 
Furthermore, with an appropriate choice of spin structures on $\mathbb{L}$, 
this identification induces an isomorphism between the (weak) 
Maurer--Cartan algebra $\MCA$ and the Jacobi algebra of $\Q$:
\[
\MCA \;\cong\; \Jac,
\]
where each variable dual to a generator $X_e$ corresponds naturally 
to the element $x_e \in \Jac$. 
Under this correspondence, the superpotential $\WL \in \MCA$ 
matches the potential $W \in \Jac$.
\end{prop}

Therefore, the noncommutative mirror $(\MCA, \WL)$ associated to the zigzag Lagrangians $\lL \subset X_\Qv$ coincides with the noncommutative Landau–Ginzburg model $(\Jac, W)$. From this point on, we will not strictly distinguish between the two, and use the notation $(\Jac, W)$ to refer to either model.

\begin{example}\label{ex:dualdimer}
The dual dimer $\Qv$ (dotted arows inin Figure~\ref{fig:five punctured sphere}) of $\Q$ in the Example~\ref{ex:dimer} is embedded in a sphere. The zigzag Lagrangian consists of three irreducible components, which are dual to three vertices in $\Q$ of the Example~\ref{ex:dimer}.
\begin{figure}[h]
\includegraphics[scale=0.4]{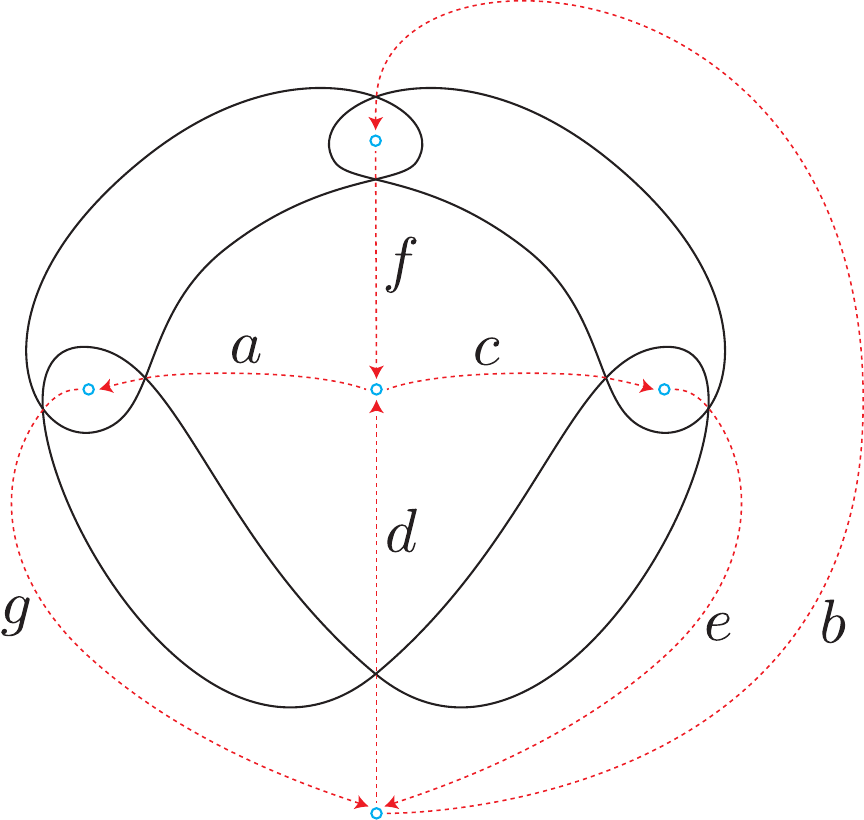}
\caption{The dual dimer $\Qv$ for the dimer $\Q$ in Figure~\ref{fig:matchingpolytope}}
\label{fig:five punctured sphere}
\end{figure}
\end{example}

Finally, we remark that for the zigzag Lagrangian $\lL$, the boundary-deformed complex
\[
\CFb = \left(\MCA \otimes_\Bbbk \CF, \{\mkb\}_{k \geq 0}\right)
\]
is well-defined without any convergence issues in the terms $m_k^{b, \dots, b}$. In fact, the output of $m_k^{b, \dots, b}$ is a finite sum, meaning that each coefficient therein is computed from finitely many contributions. While this may be intuitively clear from the geometric construction, it can be rigorously justified using the grading structure.

\begin{lemma}\label{lem:Floer theory coefficient for zigzag Lagrangians}
For zigzag Lagrangians $\lL$, the operations $m_k^{b, \dots, b}$ ($k \geq 0$) are well-defined on $\CFb$ without the need for completion.
\end{lemma}

\begin{proof}
Recall from \cite[10.5]{CHL21} that each perfect matching $\mathcal{P}$ determines a quadratic volume $\eta_P$ such that all odd-degree immersed generators have degree 1, except those lying on the perfect matching, which have degree $-1$.
Choose enough many perfect matchings, say $\mathcal{P}_1,\cdots,\mathcal{P}_l$, so that each odd-degree immersed generator $X$ has degree 1 for at least one $\eta_{\mathcal{P}_i}$. 
Now consider the fractional grading induced by the tensor product
$$\eta:= \eta_{\mathcal{P}_1} \otimes \cdots \otimes \eta_{\mathcal{P}_l} \left(\in \Gamma (T^\ast X_\Qv^{\otimes l})\right).$$
It is not difficult to see that the degree of $X$ with respect to $\eta$ averages out those with respect to $\eta_i$'s, i.e,.
$$ \deg_{\eta} (X) = \frac{1}{l} (\deg_{\eta_1} (X) + \cdots + \deg_{\eta_l} (X) ),$$
and by construction, we have $\deg_{\eta} (X) <1$. Therefore its shifted degree is negative, and hence $m_{k}^{b,\cdots,b}$ should be a finite sum (since the shifted degree of the operation is always $1$).
\end{proof}

\section{Closed-string B-model invariants for $(\Jac, W)$}
\label{subsec:compactly supported Hochschild cohomology}

The discussion naturally leads to the mirror pair arising from the dual pair of dimers $\Qv \subset \Sigma^\vee$ and $\Q \subset \Sigma$:
\begin{equation*}
\xymatrix{
\textnormal{\boxed{A} \, punctured Riemann surface} \,\, X_\Qv := \psurfv \ar@{<->}[d] \\
\textnormal{\boxed{B} \,  (nc) Landau–Ginzburg model} \,\, (\Jac, W)=(A_{\mathbb{L}},W_{\mathbb{L}})  }
\end{equation*}
Here, $\lL$ denotes the zigzag Lagrangian in $X_\Qv$. The homological mirror symmetry for this pair has been studied in \cite{Bock16, CHL21}, where it is shown that the wrapped Fukaya category of $X_\Qv$ is equivalent to a full subcategory of the matrix factorization category associated to $(\Jac, W)$.

Our goal is to compare the closed-string $A$-model invariants of $X_\Qv$ with the $B$-model invariants of $(\Jac, W)$. On the $A$-model side, we consider the symplectic cohomology of $X_\Qv$; on the $B$-model side, it is natural to consider the (compactly supported) Hochschild cohomology of the matrix factorization category of $(\Jac, W)$. The latter has been explicitly computed in \cite{Wong21}, which we now review.

\subsection{The compactly supported Hochschild cohomology $\cHH{\mf}$}\label{subsec:csupphh}

Let \((J, W)\) be a noncommutative Landau–Ginzburg (LG) model, where \(J\) is a \(\Bbbk\)-algebra and \(W \) is a central element. A general idea, going back to \cite{CT13}, is that one can instead study the Hochschild invariants of \((J, W)\) viewed as a \emph{curved algebra}. 

In this setting, the Hochschild cochain complex of \((J, W)\) is obtained by modifying the usual Hochschild differential \(d_J\) of the algebra \(J\) to include an additional component \(d_W\), which encodes the twist by \(W\). Although we will primarily work with a Koszul-type bimodule resolution of \(J\) later, we first present the explicit formula for the differential using the bar resolution.

For each \(k \geq 0\), the compactly supported Hochschild cochain complex \(CC_c^\ast(J,W)\), defined via the bar resolution, is given by:
\begin{equation*}
CC_c^k(J,W) = \bigoplus_{m+n = k} \mathrm{Hom}_\Bbbk^m(J^{\otimes n}, J).
\end{equation*}
Note that we take the \emph{direct sum} rather than the \emph{direct product}, which motivates the term ``compactly supported.'' As shown in \cite{CT13}, taking the product instead would lead to trivial cohomology.
It has the differential $d=d_J+d_W$ given by
\begin{equation*}
\begin{split}
d_J f(a_1\otimes\cdots\otimes a_{n+1}) = &(-1)^{|f||a_1|}a_1 f(a_2\otimes\cdots\otimes a_{n+1})\\
& + \sum_{i=1}^n (-1)^i f(a_1\otimes\cdots\otimes a_i a_{i+1} \otimes\cdots\otimes a_{n+1})\\
&+(-1)^{n+1}f(a_1\otimes\cdots\otimes a_n)a_{n+1}
\end{split}
\end{equation*}
and
\begin{equation}\label{eqn:dfndl}
d_W f(a_1\otimes\cdots\otimes a_{n-1}) = \sum_{i=1}^{n-1} (-1)^{i+1}f(a_1\otimes\cdots\otimes a_i \otimes W \otimes a_{i+1} \otimes\cdots\otimes a_{n-1}).
\end{equation}
\begin{defn}[{\cite[Definition 3.10]{CT13}, \cite[Section 2.4]{PP12}}]
We call
\begin{equation*}
\cHH{J,W} = H^*(CC_c^*(J,W);d)
\end{equation*}
the {\bf compactly supported Hochschild cohomology} of a curved algebra $(J,W)$. We identify this with $\cHH{\mf}$ throughout.
\end{defn}

Note that \((CC_c^*(J), d)\) forms a double complex, whose first page is precisely the Hochschild cohomology \(HH^*(J, J)\) of the algebra \(J\). The differential \(d_W\) then descends to the bracket \(\{W, -\}\) on \(HH^*(J, J)\), where \(\{-, -\}\) denotes the Gerstenhaber bracket on Hochschild cohomology. We have

\begin{prop}\label{prop:speccal}\cite[Proposition 5.2]{Wong21}
The spectral sequence  of the double complex \((CC_c^*(J), d = d_J + d_W)\) induces an additive isomorphism
$$ \cHH{J,W} (= \cHH{\mf} ) \cong H^*( HH^*(J,J);d_W = \{ W,-\}).$$
\end{prop}

In our main application, we will work with an alternative projective resolution \(J^\bullet\) of \(J=\Jac\) (see \ref{subsec:koszulresolJ}), which is better suited for this spectral sequence calculation due to its finite length.\footnote{While it is possible to transfer the differential \(d\) to \(J^\bullet\) via the standard comparison theorem (Theorem \ref{eqn:compthm}), we will not need an explicit formula for the transferred differential.} This resolution was originally introduced by Ginzburg \cite{Gin06}, and, interestingly, has a direct connection to the Floer theory of the zigzag Lagrangian \(\lL\). In fact, Wong \cite{Wong21} computed \(HH^*(J)\) primarily using this same resolution \(J^\bullet\), largely relying on Calabi–Yau duality.

%

\subsection{Calabi-Yau algebras and BV-structure on $\cHH{\Jac,W}$}\label{subsec:Calabi-Yau algebras}
Recall from \cite{Gin06} that a homologically smooth (DG-)algebra $J$ is called a \emph{Calabi-Yau algebra of dimension $d(\geq1)$}, or simply a CY$d$-algebra if 
\begin{equation*}
J  \cong  J^\vee[d]:=\mathrm{RHom}_{J\textsf{-Bimod}}(J,J \otimes_\mathbb{C} J)[d]
\end{equation*}
as $J$-bimodules. In general, we define $M^\vee := \hom_{J\textsf{-Bimod}} (M, J \otimes_\mathbb{C} J)$ to be the bimodule dual to a $J$-bimodule $M$. The target $J \otimes_\mathbb{C} J$ is viewed as the outer $J\textsf{-Bimod}$ bimodule when taking the bimodule hom, and $M^\vee$ is again a $J\textsf{-Bimod}$ bimodule via the inner structure on $J \otimes_\mathbb{C}  J$.

%

\begin{remark}
For any algebra $A$, the tensor product $A \otimes_\mathbb{C} A$ can be equipped with two $A$-bimodule structures, called the outer and the inner bimodule structures, defined as
$$ b( a \otimes a')c = ba \otimes a'c, \qquad b( a \otimes a')c = ac \otimes b a',$$
respectively. 
\end{remark}

If $J$ is Calabi-Yau of dimentions $d$, one can define a Van den Bergh duality map \cite{vdB98}
\begin{equation*}
\mathbb{D} :HH_*(J) \xrightarrow{\sim} HH^{d-*}(J)
\end{equation*}
which, coupled with the Connes $B$-operator  $B:HH_*(J) \to HH_{*+1}(J)$ (see \eqref{eqn:Connes} for the precise formula), 
induces a {\it Batalin-Vilkovisky structure}
\begin{equation}\label{eqn:BVop}
\Delta :HH^*(A) \to HH^{*-1}(A)
\end{equation}
on the Gerstenhaber algebra $(HH^*(A), \cup, \{-,-\})$, making it into a BV-algebra.
See \cite[3.4]{Gin06} for more details. 

One of great advantages from consistency condition on $\Q$ is the following.

\begin{thm}[{\cite[Theorem 4.4]{Bock16}}]
The Jacobi algebra $\Jac$ of a zigzag consistent dimer $\Q$ is a Calabi-Yau algebra of dimension 3.
\end{thm}

In particular, we can define a BV-structure on $HH^\ast (\Jac,\Jac)$. In what follows, we introduce a new projective resolution of $\Jac$ that manifests this Calabi-Yau structure.

\subsection{The Koszul resolution of $\Jac$}\label{subsec:koszulresolJ}
It is known that $J$ is a Calabi-Yau algebra (of dimension 3) if and only if the following complex of $J$-bimodules is exact (hence it is a free resolution of $J$ as the diagonal bimodule).
\begin{equation}\label{eqn:koszulcy}
J^\bullet:  0 \longrightarrow (J \otimes_{\mathbb{C}} J)^\Bbbk \stackrel{j}{\longrightarrow}  J \otimes_\Bbbk E^\ast \otimes_\Bbbk J \stackrel{c}{\longrightarrow}  J \otimes_\Bbbk E \otimes_\Bbbk J \stackrel{j^\vee}{\longrightarrow} J \otimes_\Bbbk J \longrightarrow J \longrightarrow 0
 \end{equation}
where $E$ is the $\mathbb{C}$-vector space spanned by $\Q_1$, and $E^\ast = \hom_\mathbb{C} (E,\mathbb{C})$. Notice that the resolution has a finite length. Each map in the sequence is given as follows.

\begin{itemize}
\item[(i)] By definition, $(J \otimes_{\mathbb{C}} J)^\Bbbk$ consists of $a \otimes_{\mathbb{C}} b$ such that $h(a) = t(b)$. This is opposite to $J \otimes_\Bbbk J$ whose element is $a \otimes_\Bbbk b$ satisfying $t(a) = h(b)$. Indeed, we have $(J \otimes_{\mathbb{C}} J)^\Bbbk = (J \otimes_\Bbbk J)^\vee$.
As before, the inner bimodule structure on $J \otimes_\mathbb{C} J$ makes $(J \otimes_{\mathbb{C}} J)^\Bbbk$ a $J$-bimodule, and it is generated $\sum_i \pi_i \otimes \pi_i$ over $J$. The first map $j$ is defined as
$$j: (J \otimes_{\mathbb{C}} J)^\Bbbk \to  J \otimes_\Bbbk E^\ast \otimes_\Bbbk J, \qquad a \otimes_\mathbb{C} b \mapsto \sum_{x\in \Q_1} \left(bx \otimes_\Bbbk \bar{x} \otimes_\Bbbk a - b \otimes_\Bbbk \bar{x} \otimes_\Bbbk xa\right).$$
where $\{ \bar{x} : x \in \Q_1 \}$ is the basis for $E^\ast$ dual to $\Q_1$. 

\item[(ii)]The second map 
$c: J \otimes_\Bbbk E^\ast \otimes_\Bbbk J \to  J \otimes_\Bbbk E \otimes_\Bbbk J$ in the sequence is given by
$$c(a\otimes_\Bbbk \bar{x} \otimes_\Bbbk b) :=\sum_{y \in \Q_1} a \left(\frac{\partial^2 \Phi}{\partial x\partial y} \right)' \otimes_\Bbbk  y \otimes_\Bbbk \left(\frac{\partial^2 \Phi}{\partial x\partial y}\right)''b.$$
for the cyclic potential $\Phi=\Phi_\lL$. Here, the second derivative (Hessian) is the composion of the usual cyclic derivatives
\begin{equation*}\label{eqn:ppy}
\frac{\partial}{\partial y} : \mathbb{C}\Q_{cyc}:=\mathbb{C}\Q/[\mathbb{C}\Q,\mathbb{C}\Q] \to \mathbb{C}\Q
\end{equation*}
with
\begin{equation*}\label{eqn:ppx}
 \frac{\partial}{\partial x} : \mathbb{C}\Q \to \mathbb{C}\Q \otimes \mathbb{C}\Q, \qquad \frac{\partial f}{\partial x} := \sum \left(\frac{\partial f}{\partial x}  \right)' \otimes \left(\frac{\partial f}{\partial x}  \right)'',
 \end{equation*}
where we write
$$df = \sum_x \left(\frac{\partial f}{\partial x}  \right)' \otimes dx_j \otimes \left(\frac{\partial f}{\partial x}  \right)''.$$
(Here, $df:=1 \otimes f - f \otimes 1$ for $f \in J$.) 
In practice, it can be easily calculated using the formula
$$ \frac{\partial}{\partial x} ( x_{\vec{v}}\,  x\,  x_{\vec{w}}) = x_{\vec{v}} \otimes x_{\vec{w}}.$$
when $x_{\vec{v}}$ and $x_{\vec{w}}$ are words not containing $x$.

\item[(iii)] Finally, $j^{\vee} : J \otimes_\Bbbk E \otimes_\Bbbk J \to J \otimes_\Bbbk J $ takes the form of
$$j^{\vee} (a \otimes_\Bbbk x \otimes_\Bbbk b) := ax \otimes_\Bbbk b - a \otimes_\Bbbk xb.$$
This agrees with the differential in the bar resolution in the sense that $j^{\vee}$ factors through the obvious inclusion
$$ J \otimes_\Bbbk E \otimes_\Bbbk J \hookrightarrow J \otimes_\Bbbk J \otimes_\Bbbk J \to J \otimes_\Bbbk J$$
with the second map being the differential in the bar resolution. 
\end{itemize}
The self-duality can be clearly seen from \eqref{eqn:koszulcy}. Thus it is straightforward that $J$ is Calabi-Yau if this gives a resolution.

\subsection{Hochschild invariants from the Koszul resolution}\label{subsec:hhfromk}
Let us now compute the Hochschild invariants of the algebra $J$ using the resolution $J^\bullet$. Applying $\hom_{J\textsf{-Bimod}} ( J^\bullet, -)$ to a $J$-bimodule $M$, we get the Hochschild cochain complex in coefficient $M$
\begin{equation}\label{eqn:hhcohomorig} 
\begin{array}{l}
  \hom_{J\textsf{-Bimod}} (J\otimes_\Bbbk J, M ) \longrightarrow  \hom_{J\textsf{-Bimod}} (J\otimes_\Bbbk E \otimes_\Bbbk J, M) \\
   \longrightarrow \hom_{J\textsf{-Bimod}} ( J \otimes_\Bbbk E^\ast \otimes_\Bbbk J,M) \longrightarrow  \hom_{J\textsf{-Bimod}} ((J\otimes_\mathbb{C} J)^{\Bbbk}, M ) 
  \end{array} 
\end{equation}
which reduces to
\begin{equation*}
 \hom_\Bbbk (\Bbbk, M ) \longrightarrow  \hom_\Bbbk ( E,M)  \longrightarrow  \hom_\Bbbk (E^\ast, M)  \longrightarrow \hom_\Bbbk (\Bbbk, M ).
\end{equation*}
The above complex can be further simplified to
\begin{equation}\label{eqn:koszulcyhh^*}
CC^\ast(J,M): \,\,M_{cyc} \stackrel{d_0}{\longrightarrow}  (M \otimes_\Bbbk E^\ast)_{cyc}  \stackrel{d_1}{\longrightarrow}  (M \otimes_\Bbbk E)_{cyc} \stackrel{d_2}{\longrightarrow} M_{cyc}.
\end{equation}
In particular, the Hochschild cochain complex in coefficient $J$ is given by
\begin{equation}\label{eqn:CCJJ}
CC^\ast(J,J): \,\, J_{cyc}  \longrightarrow  (J \otimes_\Bbbk E^\ast)_{cyc}   \longrightarrow  (J \otimes_\Bbbk E)_{cyc} \longrightarrow J_{cyc}.
\end{equation}

To clarify the connection with the Floer complex later, we will write \( X := \bar{x} \) and \( \bar{X} := x \) for the generators of \( E^* \) and \( E \) appearing in \eqref{eqn:koszulcyhh^*}. Then the identification between \eqref{eqn:hhcohomorig} and \eqref{eqn:koszulcyhh^*} is given as
$$ \hom_{J\textsf{-Bimod}} (J\otimes_\Bbbk J, M ) \to M_{cyc} \qquad f \mapsto f( 1 \otimes_\Bbbk 1),$$
where $1$ denotes the unit $ \sum_i \pi_i$ of $J$, and
\begin{equation*}
\begin{array}{rclrcl}
\hom_{J\textsf{-Bimod}} (J\otimes_\Bbbk  E \otimes_\Bbbk J, M ) &\!\!\cong\!\!& (M\otimes_\Bbbk E^\ast)_{cyc} &\quad f &\!\mapsto\!& \sum_{x \in Q_1} f( 1 \otimes_\Bbbk x \otimes_\Bbbk 1) \otimes_\Bbbk X  \\
 \hom_{J\textsf{-Bimod}} (J\otimes_\Bbbk  E^\ast \otimes_\Bbbk J, M ) &\!\!\cong\!\!& (M\otimes_\Bbbk E)_{cyc} & f &\!\mapsto\!& \sum_{x \in Q_1} f( 1 \otimes_\Bbbk \bar{x} \otimes_\Bbbk 1) \otimes_\Bbbk \bar{X} \\
 \hom_{J\textsf{-Bimod}} ((J\otimes_\mathbb{C} J)^\Bbbk, M ) &\!\!\cong\!\!& M_{cyc} & f &\!\mapsto\!& f \left(  \sum_i e_i \otimes_\mathbb{C} e_i \right),
 \end{array}
 \end{equation*}
and the differential on \eqref{eqn:koszulcyhh^*} can be written as
\begin{equation}\label{eqn:disforhhk}
\begin{array}{rcl}
d_0 (m) &=& \sum_{x \in Q_1} (xm -mx) \otimes_\Bbbk X, \\
d_1 (m \otimes_\Bbbk Y) &=&  \sum_{x \in Q_1} \left( \frac{\partial^2 \Phi}{\partial x \partial y} \right)' m \left( \frac{\partial^2 \Phi}{\partial x \partial y} \right)'' \otimes_\Bbbk \bar{X} , \\
d_2 (m \otimes_\Bbbk \bar{Y}) &=& ym - my = [y,m].
\end{array}
\end{equation}
It is easy to check that the above coincides with the complex $J^\bullet \otimes_{J\textsf{-Bimod}} M$, or more precisely,
$$\hom_{J\textsf{-Bimod}} ( J^\bullet, M) \cong  J^{3-\bullet} \otimes_{J\textsf{-Bimod}} M$$
which establishes the duality 
\begin{equation}\label{eqn:hochdualcy}
 HH^\ast (J,M) \cong HH_{3-\ast} (J, M).
\end{equation}
(This duality holds for any Calabi-Yau algebra.)

%

\subsection{BV operators in the Hochschild cohomology in view of different resolutions}\label{subsec:bvformula}

Our upcoming argument requires an explicit identification between the generators of \( HH^*(J, J) \) used in \cite{Wong21} and those of a certain Floer cohomology. Although the identification of their underlying complexes will turn out to be straightforward, some of the generators in \cite{Wong21} involve the BV operator $\Delta$ (see \eqref{eqn:BVop}) in their definitions. To address this, we need to compute the BV operator on a specific generator in the Hochschild cohomology of $J$ (obtained from the Koszul resolution).

\begin{lemma}\label{lem:bvopex} For an element $a \in HH^3 (J,J)$, choose its representative $x_1 \cdots x_k \in \mathbb{C} \Q$ (i.e., a cyclic path which lifts $a$). Then we have
$$\Delta ( a ) =    \displaystyle\sum_i (x_{i+1} \cdots x_k )(x_1 \cdots x_{i-1} )\otimes_\Bbbk \bar{X_i},$$
which defines a class in $HH^2 (J,J)$. 
\end{lemma}

The proof of Lemma~\ref{lem:bvopex}, together with further details on the BV operators, will be provided in Appendix~\ref{sec:comparison}.

Recall that there is an additional component $d_W$ of the Hochschild differential (for $HH^\ast(J,W)$), which produces a differential acting on $HH^\ast(J, J)$. 
We remark that $d_W$ (see~\eqref{eqn:dfndl}) can be expressed entirely in terms of the BV operator $\Delta$ as
\begin{equation}\label{eqn:BVbracket}
d_W(f) = \{W, f\} = \Delta(Wf) - W\,\Delta(f),
\end{equation}
using the algebraic relations among the operators $\Delta$, $\cup$, and $\{-,-\}$. 
This formula will be useful later for explicitly computing $d_W$ on cocycles in $HH^3(J,J)$.

\subsection{Explicit calculation for $HH^\ast(J,J)$ and $d_W$}\label{subsec:hhjacw}
We present several explicit computational results on $HH^\ast(J,J)$ and $HH^\ast_c (mf(W)) = HH^\ast (J,W)$ that will be used in the later part of the paper. For details about the calculations, we refer readers to \cite[Section]{Wong21}.
We first list generators of $HH^\ast(J,J)$ according to their homological degree (i.e., the resolution degree from $J^\bullet$). 
Recall that the dimer $\Q$ is embedded in the torus $T^2$.
\begin{itemize}
\item[($\deg 0$)] Each nontrivial $\alpha \in H_1 (T^2;\mathbb{Z})$ defines a Hochschild cycle $x_\alpha \in HH^0 (J,J) = \mathcal{Z} (\Jac)$ whose underlying path in $\mathbb{C} \Q$ belongs to the class $\alpha$. By requiring it to be a non-multiple of $W$, it is uniquely determined (up to Jacobi relations). In particular, each antizigzag class $\eta_i$ defines a Hochschild cocycle $x_{\eta_i}$.
\item[($\deg 1$)]
Each corner matching $\mathcal{P}_i$ defines a Hochschild $1$-cocycle $\partial_{\cP_i}$ defined (on the chain level) by
\begin{equation*}
\partial_{\cP_i} \in (J \otimes E^\ast)_{cyc} : e \mapsto 
\left\{
\begin{array}{lcl}
e \quad \textnormal{if} \,\, e \in \mathcal{P}_i\\
0 \quad \textnormal{otherwise}
\end{array}\right.
\end{equation*}
where we regard this as a map $E \to J$ (hence an element of $(J \otimes_\Bbbk E^\ast)_{cyc} \subset CC^\ast (J,J)$ \eqref{eqn:CCJJ}), which can be extended to a derivation $J \to J$. Let $\sigma_i$ be the cone in $H_1 (T^2;\mathbb{Z})$ spanned by $\eta_i$ and $\eta_{i+1}$. For $\alpha \in \mathrm{int}\, \sigma_i$, we define
$$ \partial_\alpha := x_\alpha W^{-1} \partial_{\cP_i} \in (J \otimes E^\ast)_{cyc}.$$
The expression involves $W^{-1}$, and hence it a priori only gives a map $J\to J[W^{-1}]$, but one can prove that it factors through $J$ by some grading argument using perfect matchings.
\item[($\deg 3$)]
Any vertex $v \in \Q_0$ (or the corresponding idempotent) defines an element of $[v] \in  HH_0 (J,J)$. Denote by $\theta_v \in HH^3 (J,J)$ the element dual to $[v]$ under the duality $HH^3 (J,J) \cong HH_0 (J,J)$ \eqref{eqn:hochdualcy}. On the chain level, it is simply the idempotent at $v$ in $J_{cyc}$ .
\item[($\deg 2$)] Suppose $Z_{i,1} \cdots Z_{i,m_i}$ are parallel zigzag cycles in the class $\eta_i$ arranged in such a way that they divide $T$ into closed strips $\mathcal{E}_{i,j}$ bounded by $\mathcal{O}^+ (Z_{i,j} ) \cup \mathcal{O}^- (Z_{i,j+1})$. For $v \in \mathcal{E}_{i,j}$, we set $\psi_{i,j}:=\Delta ( x_{\eta_i} \theta_v) \in HH^2(J,J)$ where $\Delta$ is the BV operator on $HH^\ast (J,J)$.
\end{itemize}

One can specify a representative of $\psi_{i,j}$ as follows. Note that any vertex $v$ contained in the antizigzag $\mathcal{O}^+ (Z_{i,j})$  lies in $\mathcal{E}_{i,j}$ by definition. Moreover, $x_{\eta_i} \theta_v$ regarded as as element of $J_{cyc}$ in this case is nothing but the labelling of the antizigzag $\mathcal{O}^+ (Z_{i,j})$ itself that ends at $v$. Therefore, by Lemma \ref{lem:bvopex}, we have
\begin{equation}\label{eqn:psiij}
\psi_{i,j} = \Delta(x_{\eta_i} \theta_v) =\Delta ( x_1 \cdots x_k ) =    \displaystyle\sum_i (x_{i+1} \cdots x_k )(x_1 \cdots x_{i-1} )\otimes_\Bbbk \bar{X_i},
\end{equation}
where $x_1 \cdots x_k$ represents the antizigzag $\mathcal{O}^+ (Z_{i,j})$ such that $t(x_k) = v$.

In terms of the above cocycles, $HH^\ast (J,J)$ is given as follows.

\begin{prop}\cite[Theorem 1.2]{Wong21}\label{prop:hhjacw} For a zigzag consistent dimer model $\Q$, the Hochschild cohomology of $\Jac$ is given by
\begin{equation*} \label{eqn:HH2Jadd}
\begin{array}{rcl}
HH^0 (J,J) &\cong& \mathcal{Z}    \\
HH^1(J,J) &\cong&  \displaystyle\bigoplus_{1 \leq i \leq N} \mathcal{Z} \cup  \partial_{\cP_i}  \oplus \displaystyle\bigoplus_{\alpha \in \mathrm{int}\, \sigma_i} \mathbb{C}  \partial_\alpha \\
HH^2 (J,J) &\cong& \displaystyle\bigoplus_{1 \leq i,j \leq N} \mathcal{Z} \cup  \partial_{\cP_i}  \cup   \partial_{\cP_j} \oplus \displaystyle\bigoplus_{\substack{1 \leq j \leq N \\ \alpha \in \mathrm{int}\,\sigma_i}} \mathbb{C}    \partial_\alpha \cup  \partial_{\cP_j}  \oplus \displaystyle\bigoplus_{\substack{1 \leq i \leq N \\ 1\leq j \leq m_i, n\in \mathbb{Z}_{>0} }}  \mathbb{C}  x_{\eta_i}^n \psi_{i, j} \\
HH^3 (J,J) &\cong& \displaystyle\bigoplus_{1 \leq i,j,k \leq N} \mathcal{Z} \cup   \partial_{\cP_i}  \cup   \partial_{\cP_j} \cup \partial_{\cP_k}   \oplus  \displaystyle\bigoplus_{\substack{1 \leq j,k \leq N \\ \alpha \in \mathrm{int}\,\sigma_i}}  \mathbb{C}   \partial_\alpha \cup  \partial_{\cP_j} \cup \partial_{\cP_k}  \oplus \displaystyle\bigoplus_{\substack{1 \leq i \leq N \\ v \in \mathcal{E}_{i,j}, n \in \mathbb{Z}_{>0} } } \mathbb{C}  x_{\eta_i}^n \theta_v 
\end{array}
\end{equation*}
where $\mathcal{Z} = \mathcal{Z} (\Jac)$ denotes the center of the Jacobi algebra. 
\end{prop}

The following formulas from \cite{Wong21} will be used later.

\begin{lemma}\cite[Proposition 4.21, 4.23]{Wong21}\label{lem:calcinwong}
The following holds for the Gerstenharber bracket $\{-,-\}$ and the cup product $\cup$ on $HH^\ast(J,J)$.
\begin{itemize}
\item[(i)]
The Gerstenharber bracket between $f \in HH^0(J,J)$ and $\partial_{\cP_i} \in HH^1(J,J)$ is given by
$$ \{\partial_{\cP_i}, f\} = \deg_{P_i} (f) f$$
\item[(ii)] The cup product between $\partial_{\cP_k} \in HH^1(J,J)$ and $\psi_{i,j} \in HH^2(J,J)$ is given by
$$ \partial_{\cP_k} \cup \psi_{i,j}
= \deg_{P_k} (x_{\eta_i})\, x_{\eta_i}\,\theta_v$$
for $v \in \mathcal{E}_{i,j}$.
\end{itemize}
\end{lemma}

\section{Hochschild invariants of $(\Jac,W)$ from Floer theory of zigzag Lagrangians}\label{sec:compareresol}

In this section, we show that the Hochschild invariant $\cHH{\mathrm{mf}(W)}$ of the noncommutative LG-mirror in Section \ref{subsec:compactly supported Hochschild cohomology}  can be obtained from the deformed Floer complex $\CFbloop$ defined in \ref{def:cflbcyc} for the zigzag Lagrangian $\lL$. Namely, we prove

\begin{prop}\label{prop:flequalshh}
For a zigzag consistent dimer $\Q$, there is an (additive) isomorphism
$$\HFbloop \cong \cHH{\mathrm{mf}(W)}$$
where $\lL$ is the zigzag Lagrangian $\lL$ obtained from the dual dimer $\Qv$.
\end{prop}

We will use a certain nongeometric $\mathbb{Z}$-grading on this Floer complex to make it a double complex, and use associated spectral sequence to show that the cohomology of $\CFbloop$  computes $\cHH{\mf}$. This indeed coincides with the spectral sequence for $\cHH{\mf}$ itself, appearing in Proposition \ref{prop:speccal}.

\subsection{$\CFb$ for the zigzag Lagrangians $\lL\subset X_\Qv$}\label{subsec:CFb for zigzag Lagrangians}

Let $\Q$ be a dimer, and $\lL$ the zigzag Lagrangians in $X_\Qv$. Recall from \ref{subsec:zigzag Lagrangians} that
the self-intersections in $\lL$ occur along an edge of $\Qv$, which, by dimer duality, leads to an one-to-one correspondence between the edges $e \in\Q_1$ and the self-intersections in $\lL$. Such a self-intersection point gives rise to two generators of $CF(\lL,\lL)$ indicating branch jumps between local Lagrangian components at the intersection point. For $e\in\Q_1$, we denote the corresponding generators in the odd and even degrees by $\Xe$ and $\bXe$, respectively. The immersed Floer complex \eqref{eqn:immersed Lagrangian Floer complex} can be written as
\begin{equation*}
\CF = \bigoplus_{i=1}^k C^*(L_i) \oplus \bigoplus_{e\in\Q_1}\mathrm{Span}\{\Xe,\bXe\},
\end{equation*}
where $C^*(L_i)$ is a Morse complex of a Lagrangian circle $L_i$ (more precisely, the domain circle of a possibly immersed circle $L_i$) with respect to a chosen perfect Morse function.
Pictorially, $X_e$ corresponds to the corner formed by two branches of $\lL$ whose orientations are aligned (either positively or negatively), and $\bXe$ is complementary to $X_e$. See Figure \ref{fig:zigzagdisks}.

Solving (weak) Maurer-Cartan equation for $b=\sum_{e \in \Q_1} x_e X_e$ produces a noncommutative LG model $(\MCA,W)$. We know that it agrees with the Jacobi algebra $\Jac$ together with the central element $W$ (Proposition \ref{lem:spacetime superpotential}).

We define a grading on $\CFb = \MCA \otimes_\Bbbk CF^\ast (\lL,\lL)$ so that generators listed below have degree $0,1,2$ and $3$: 
\begin{equation}\label{eqn:extdeggen}
	\begin{tikzcd}
	 	\oplus_i \MCA\cdot\mathbf{1}_i,
	 	& \oplus_{e \in \Q_1} \MCA\cdot\Xe,
	 	& \oplus_{e\in \Q_1} \MCA\cdot\bXe,
	 	& \oplus_i \MCA\cdot \pt_i,
	 \end{tikzcd}
\end{equation}
where $\mathbf{1}_i$ is the maximum of the perfect Morse function on $L_i$ (hence it gives the unit class), and $\pt_i$ is the minimum. We will often call this the ``exterior grading" due to its resemblance with the grading on the exterior algebra (especially in the case of pair-of=pants), and write $|\mathbf{1}_i| =0$, $|\Xe|=1$, $|\bXe|=2$ and $|\pt_i|=3$ to indicate their exterior degrees. 

This agrees with the Floer grading only in modulo $2$, and 
the Floer differential $m_1^{b,b}$ on $\CFb$ decomposes as $m_1^{b,b} = m_{1,+1}^{b,b} + m_{1,-1}^{b,b}$ according to the exterior grading. Writing $\delta = m_1^{b,b}$, $\delta_+ = m_{1,+1}^{b,b}$ and $\delta_- = m_{1,-1}^{b,b}$ for simplicity, we have
\begin{equation*}\CFb: 
	\begin{tikzcd}
	 	\oplus_i \MCA\cdot\mathbf{1}_i \ar[r, bend left=10, "\deltap" description] 	 	& \oplus_{e \in \Q_1} \MCA\cdot\Xe \ar[l, bend left=10, "\deltam" description] \ar[r, bend left=10, "\deltap" description]
	 	& \oplus_{e\in \Q_1} \MCA\cdot\bXe \ar[l, bend left=10, "\deltam" description] \ar[r, bend left=10, "\deltap" description]
	 	& \oplus_i \MCA\cdot \pt_i,  \ar[l, bend left=10, "\deltam" description]
	 \end{tikzcd}
\end{equation*}
i.e., $|\delta_+|=1$ and $|\delta_-| =-1$.
Here, the component $ \bigoplus \MCA\cdot \mathbf{1}_i \to \bigoplus \MCA\cdot \pt_i$ of $\delta$ vanish due to unitality. By choosing Morse functions suitably, one can make $\bigoplus \MCA\cdot \pt_i \to \bigoplus \MCA\cdot \mathbf{1}_i$ also vanish. For instance, we may locate $\mathbf{1}_i$ on the boundary of a positive disk for all $i$, and $\pt_i$ on the boundary of a negative disk so that they never lie on the boundary of the same disk.
%

Finally, we restrict ourselves to the subcomplex $\CFbloop$ of $\CFb$ generated by cyclic element, and the decomposition $d =\delta_+ + \delta_-$ descends to $\CFbloop$:
\begin{equation}\label{eqn:cfbloopcyc}
	\begin{tikzcd}
	 	\oplus_i \pi_i \MCA \pi_i\cdot\mathbf{1}_i \ar[r, bend left=10, "\deltap" description] 	 	& \oplus_{e \in \Q_1} (\MCA\cdot\Xe)_{cyc} \ar[l, bend left=10, "\deltam" description] \ar[r, bend left=10, "\deltap" description]
	 	& \oplus_{e\in \Q_1} (\MCA\cdot\bXe)_{cyc} \ar[l, bend left=10, "\deltam" description] \ar[r, bend left=10, "\deltap" description]
	 	& \oplus_i \pi_i \MCA \pi_i \cdot \pt_i.  \ar[l, bend left=10, "\deltam" description]
	 \end{tikzcd}
\end{equation}
where $(\MCA\cdot\Xe)_{cyc}=\pi_{t(\Xe)}\MCA\cdot\Xe$ and $(\MCA\cdot\bXe)_{cyc}=\pi_{h(\Xe)}\MCA\cdot\bXe$. This makes $\CFbloop$ into a double complex. In the main application, we will compare this double complex with the one in \cite{Wong21} computing $\cHH{\mf}$.

\subsection{$\delta_+$-cohomology and $HH^\ast (\Jac,\Jac)$}
Recall that $HH_c^\ast (\mf)$, identified as the Hochschild cohomology of the curved algebra $(\Jac, W)$, can be (additively) computed as the cohomology of $HH^\ast(\Jac,\Jac)$ with respect to the differential $d_W = \{W,-\} : HH^\ast(\Jac,\Jac) \to HH^\ast(\Jac,\Jac)$. We first claim that the $\delta_+$-cohomology of $\CFbloop$ (which is the first page of the associated spectral sequence) agrees with $HH^\ast(\Jac,\Jac)$. In fact, one can easily identify their underlying complexes.

\begin{lemma}\label{lem:idplus}
There is a chain-level isomorphism 
\begin{equation}\label{eqn:hochadp}
(\CFbloop,\delta_+) \cong (\hom_{J\textsf{-Bimod}} (J^\bullet, \Jac ), d_A),
\end{equation}
where $d_A$ denotes the Hochschild differential with respect to the Koszul resolution $J^\bullet$ of $\Jac$ (see \ref{subsec:koszulresolJ}).
\end{lemma}

\begin{proof}
We write $J=\Jac$ for simplicity.
Recall from \ref{subsec:hhfromk} that the Hochschild cochain complex $\hom_{J\textsf{-Bimod}} (J^\bullet, M )$ simplifies to
\begin{equation}\label{eqn:koszulcyhh^*MJ}
J_{cyc} \stackrel{d_0}{\longrightarrow}  (J \otimes_\Bbbk E^\ast)_{cyc}  \stackrel{d_1}{\longrightarrow}  (J \otimes_\Bbbk E)_{cyc} \stackrel{d_2}{\longrightarrow} J_{cyc}.
\end{equation}
where $E$ is the vector space formally generated by $\Q_1$, and $E^\ast$ is its dual.
Given that $\MCA =\Jac$ by Proposition \ref{lem:spacetime superpotential}, there is an obvious graded vector space isomorphism between \eqref{eqn:koszulcyhh^*MJ} and \eqref{eqn:cfbloopcyc}. The identification for degree $0$ and $3$ components are simply the trivial identify $J_{cyc} = \oplus_i \pi_i J \pi_i$. 

Now the image generators in \eqref{eqn:extdeggen} under $\delta_+$ can be explicitly calculated as follows. Firstly, the even degree generators are mapped to
\begin{eqnarray}
\delta_+ (f \cdot \mathbf{1}_i) &=& \displaystyle\sum_{\substack{x \in Q_1 \\ h(X)=i}} xf X- \sum_{\substack{x \in Q_1\\ t(X) =i} } fxX  \label{eqn:imagedeltap0} \\
\delta_+ (f \cdot \bXe) &=& \xe f \pt_{t(e)} - f \xe \pt_{h(e)}. \label{eqn:imagedeltap2}
\end{eqnarray}
The former \eqref{eqn:imagedeltap0} is determined by the unital property $m_2 ( \mathbf{1}_i, X ) = (-1)^{|X|} m_2 ( X, \mathbf{1}_j ) = X$  for $X \in CF(L_i,L_j)$, while the latter \eqref{eqn:imagedeltap2} is due to  $m_2(X,\bar{X}) = - \pt_i,  m_2 (\bar{X},X) = \pt_j$ when $X \in CF(L_i,L_j)$ and $\bar{X} \in CF(L_j,L_i)$.
This is essentially the same calculation as \cite[(10.2)]{Sei11}. 

For a degree $1$ generator $f \cdot X_e$, 
$\delta_+ (f \cdot X_e)$ counts all polygons \( P \) having \( X_e \) as a corner. Let   $w_p = x_k \cdots x_1 x_e$ be the (clockwise) labeling of corners appearing in $\partial P$ up to cyclic permutations. See Figure \ref{fig:zigzagdisks}. Note that these are precisely two summands of the cyclic potential $\Phi$ involving $x_e$.

\begin{figure}[h]
\includegraphics[scale=0.175]{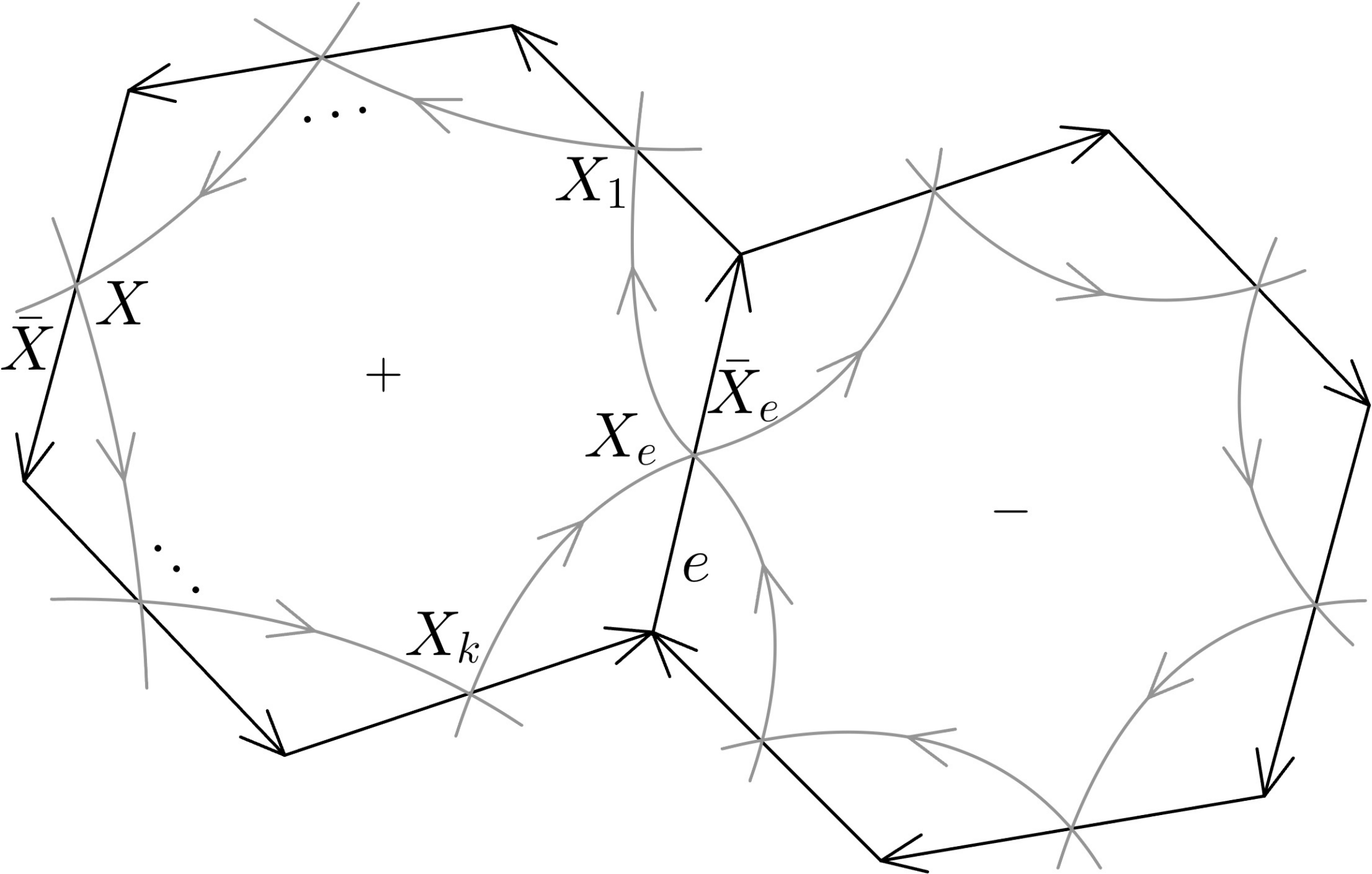}
\caption{Holomorphic disks bounded by $\lL$ (grey) with $\Xe$ on their corner}\label{fig:zigzagdisks}
\end{figure}

Since \( \delta_+ \) produces an output of exterior degree 2, each corner \( X_i \) of \( P \) other than \( X_e \) contributes a term of the form \( x_{i-1} \cdots x_1 f x_k \cdots x_{i+1} \bar{X}_i \). Therefore
\begin{equation}\label{eqn:imagedeltap1}
\begin{array}{rcl}
 \delta_+ (f \cdot X_e) &=& \displaystyle\sum_{\substack{P \,\, \textnormal{with}\\ w_P = x_k \cdots x_1 x_e}} x_{i-1} \cdots x_1 f x_k \cdots x_{i+1} \bar{X}_i\\
 &=& \displaystyle\sum_{e \in \Q_1} \left( \frac{\partial^2 \Phi}{\partial x \partial x_e} \right)' f \left( \frac{\partial^2 \Phi}{\partial x \partial x_e} \right)'' \otimes_\Bbbk \bar{X} 
\end{array}
\end{equation}
\eqref{eqn:imagedeltap0}, \eqref{eqn:imagedeltap2} and \eqref{eqn:imagedeltap1} agree with the calculation of $d_A$ given in \eqref{eqn:disforhhk}, and the lemma follows.
\end{proof}

 It is straightforward to match generators of $HH^\ast(J,J)$ appearing in Proposition \ref{prop:hhjacw} and Floer generators of $H^\ast (\CFbloop,\delta_+)$. Let us write 
\begin{equation}\label{eqn:Fidhhcf}
F: HH^\ast(J,J) \stackrel{\cong}{\to} H^\ast (\CFbloop,\delta_+)
\end{equation}
for the isomorphism in the proof of Lemma \ref{lem:idplus} (or its chain level map from the identification of the underlying complexes). Then we have
\begin{equation}\label{eqn:ajidn1}
F(x_\alpha) = x_\alpha \unit, \quad F(\partial_{\cP_i})=\sum_{e \in \mathcal{P}_i} x_e X_e,
\end{equation}
and $F(\partial_\alpha) =  \sum_{e \in \mathcal{P}_i} x_\alpha W^{-1} x_e X_e$ for $\alpha \in \mathrm{int}\, \sigma_i$. Here, $x_\alpha W^{-1} x_e$ belongs to $ \Jac$ since it has nonnegative degrees with respect to all perfect matchings (see \cite[Lemma 4.5]{Wong21}). Also,
\begin{equation}\label{eqn:ajidn2}
 F(\theta_v ) = \pt_{i}
 \end{equation}
where $L_i=L_{i_v}$ is the component of $\lL$ corresponding to the vertex $v \in \Q_0$. Finally, from \eqref{eqn:psiij}, we have
\begin{equation}\label{eqn:ajidn3}
 F(\psi_{i,j}) =   \displaystyle\sum_i (x_{i+1} \cdots x_k )(x_1 \cdots x_{i-1} )\otimes_\Bbbk \bar{X_i},
 \end{equation}
where $x_1 \cdots x_k$ represents the antizigzag $\mathcal{O}^+ (Z_{i,j})$.

\begin{remark}\label{rmk:idjandcf}
The resolution $J^\bullet$ itself can also be expressed in terms of the Floer complex of $\lL$. Let us write $V:= \CF$. As modules, we have
$$ J^\bullet \cong J \otimes_\Bbbk V^\ast \otimes_\Bbbk J,$$
which identifies $E^\ast$ with the dual of $\mathbb{C} \langle \bXe : e \in \Q_1 \rangle$, and $E$ with the dual of $\mathbb{C} \langle \Xe : e \in \Q_1 \rangle$.
Under this identification, the differential on the right hand side is a $J$-bimodule map obtained by taking dual of the $A_\infty$-operations (the ``deformed $m_1$") on $V$. Namely, it is given by
$$  a \in V^\ast   \mapsto \sum \lambda \,x_1 \cdots x_i  \otimes a' \otimes y_j \cdots y_1 $$
where the sum in the right hand side runs over all $A_\infty$-operations satsifying
$$ m_k ( Y_1, \cdots, Y_j ,A' ,X_i,\cdots, X_1) = \cdots + \lambda^{-1} A + \cdots$$
for degree 1 immersed generators $X$'s and $Y$'s (i.e., $A$ appears as a summand of the output with a nontrivial coefficient). Here $A$, $A'$ are Floer generators in the basis chosen in \ref{subsec:CFb for zigzag Lagrangians} with $|A| = |A'| +1$, and $a$, $a'$ are their corresponding elements in the dual basis. With this perspective, the identification between $d_A$ and $\delta_+$ becomes manifest.
\end{remark}

In order to finish the proof of Proposition \ref{prop:flequalshh}, it remains to compare $d_W=\{W,-\}$ and $\delta_-$ on the cohomologies of \eqref{eqn:hochadp}. 
This is automatic for (exterior) degree 0 ($d_A$- or $\delta_+$-)cocycles.
To simplify the comparison for positive degree cocycles, we need some preparations.

\subsection{The cup product on $HH^\ast (\Jac,\Jac)$}
Recall that $HH^\ast (\Jac,\Jac)$ carries the cup product $\cup$. We show that this can be obtained from the ``leading order" Floer product on $\CFbloop$.
The Floer product $m_2^{b,b,b}$ on $\CFbloop$ also decomposes according to the exterior grading as
$$m_2^{b,b,b} =  m_{2,0}^{b,b,b}  +  m_{2,<0}^{b,b,b}  $$
($ m_{2, \geq 2}^{b,b,b}$ vanishes due to obvious degree reason and unitality.) 

By breaking down the equation $\left[m_1^{b,b},m_2^{b,b,b} \right]=0$ degreewise, we see that
$\delta_+=m_{1,+1}^{b,b}$ is a derivation with respect to $ m_{2,0}^{b,b,b}$, and hence $m_{2,0}^{b,b,b}$ induces a product on 
$$H^\ast (\CFbloop, \delta_+) = HH^\ast(J,J)$$ 
We claim that this product coincides with the cup product on $HH^\ast(J,J)$.

\begin{lemma}
The product on $HH^\ast (J,J)$ induced from $m_{2,0}^{b,b,b}$ agrees with the cup (Yoneda) product on the Hochschild cohomology.
\end{lemma}

\begin{proof}
Let us first recall the following well-known fact about the cup product on the Hochschild cohomology. Consider a resolution $J^\bullet$ of $J$, and suppose $\tilde{\Delta}: J^\bullet  \to J^\bullet \otimes_J J^\bullet$ lifts the diagonal map, i.e,
\begin{equation}\label{eqn:liftdiag}
\xymatrix{ J^\bullet \ar[r] \ar[d]_{\tilde{\Delta}} & J  \ar[d]^{\Delta}\\
J^\bullet \otimes_J J^\bullet \ar[r] & J \otimes_J J \cong J,
 }
 \end{equation}
Given such $\tilde{\Delta}$, the cup product on $HH^\ast (J) =H^\ast \hom_{J\textsf{-Bimod}}(J^\bullet, J)$ can be computed by
 \begin{equation}\label{eqn:cuponHH}
 \hom_{J\textsf{-Bimod}}(J^\bullet, J)^{\otimes 2} \to \hom_{J\textsf{-Bimod}}(J^\bullet \otimes J^\bullet, J) \to  \hom_{J\textsf{-Bimod}}(J^\bullet, J),
 \end{equation}
 where the second map is the dual of $\tilde{\Delta}$. 
 
 In our situation, we have $J^\bullet = J \otimes_\Bbbk V^\ast \otimes_\Bbbk J$ for $V=\CF$ as explained in Remark \ref{rmk:idjandcf}, and hence, we need to find a $J$-bimodule map
 $$ \tilde{\Delta} : J \otimes_\Bbbk V^\ast \otimes_\Bbbk J \to (J \otimes_\Bbbk V^\ast \otimes_\Bbbk J) \otimes_J (J \otimes_\Bbbk V^\ast \otimes_\Bbbk J)= J \otimes_\Bbbk V^\ast \otimes_\Bbbk J \otimes_\Bbbk V^\ast \otimes_\Bbbk J.$$
 Analogously to Remark \ref{rmk:idjandcf}, it can be obtained by daulising $A_\infty$-operations (the ``deformed $m_2$") in the sense that 
 $$ \tilde{\Delta}: a \in V^\ast \mapsto \sum \lambda \, x_1 \cdots x_i \otimes  a_1 \otimes y_1 \cdots y_j \otimes a_2  \otimes  z_1 \cdots z_k  $$
where the sum is taken for all $A_\infty$-operations satisfying
\begin{equation}\label{eqn:hatm2}
 m_k (Z_k,\cdots, Z_1,A_2, Y_j,\cdots, Y_1, A_1,X_i,\cdots, X_1) = \cdots + \lambda^{-1} A + \cdots 
 \end{equation}
for any degree 1 immersed generators $X$, $Y$ and $Z$. It is easy to check that $\tilde{\Delta}$ satisfies \eqref{eqn:liftdiag} (the only $A_\infty$-operation involved here is $m_2 (\unit, \unit) = \unit$).
%
%
Thus the product in \eqref{eqn:cuponHH} is essentially the double dual of \eqref{eqn:hatm2} and can be easily identified with $m_{2,0}^{b,b,b}$.
\end{proof}

The following remark explains the origin of the discrepancy between the product structure on $SH^\ast (X)$ and that on $HH^\ast(\mf)$, which is used in the comparison carried out in \cite{Wong21}.
\begin{remark}\label{rmk:discrem2}
Observe that 
$$\left[m_{1,-1}^{b,b},m_{2,0}^{b,b,b} \right] = -  \left[m_{1,+1}^{b,b},m_{2,<0}^{b,b,b} \right]$$
where the right hand side vanishes on the $\delta_+(=m_{1,+1}^{b,b})$-cohomology. Therefore $m_{2,0}^{b,b,b}$ also induces a product on the $E_2$-page of the spectral sequence of the double complex $d_A + d_W$. This is precisely the product appearing in \cite{Wong21}.
Although the $E_2$-page agrees with $\cHH{\mf}$ as vector spaces, $m_{2,0}^{b,b,b}$ does not seem to be an intrinsic product on $\cHH{\mf}$, nor is it compatible with the product on $SH^\ast(X)$. 
\end{remark}
%

\subsection{Proof of $\HFbloop \cong HH_c^\ast (\mf)$}
We now prove Propostion \ref{prop:flequalshh}. By Lemma \ref{lem:idplus}, it suffices to show that $\delta_-$ equals $d_W = \{W,-\}$, or more precisely $F$ in \eqref{eqn:Fidhhcf} intertwines $d_W = \{W,-\}$
\begin{equation*}
\xymatrix{ HH^\ast (J,J) \ar[d]_{F}^{\cong} \ar[r]^{d_W} & HH^{\ast-1} (J,J) \ar[d]^F_{\cong}\\
H^\ast (\CFbloop,\delta_+) \ar[r]^{\delta_-} & H^{\ast-1} (\CFbloop,\delta_+).
}
\end{equation*}
This is trivial for exterior degree $0$ generators in $HH^\ast (J)$. We begin with the following set of  generators of the degree $1$ component
$$\partial_{\cP_i} ,\; \partial_\alpha \quad \textnormal{for} \,\, 1 \leq i \leq N \,\,\textnormal{and} \,\, \alpha \in H_1 (T^2,\mathbb{Z})$$
(as a module over the ring spanned by degree $0$ elements, see \ref{subsec:hhjacw}).
The image of $d_W$ evaluated at these degree $1$ generators are explicitly computed in \cite{Wong21}. Lemma \ref{lem:calcinwong} applying to $f=W$ tells us that
\begin{equation}\label{eqn:dlpartiali}
d_W (\partial_{\cP_i}) = -\{\partial_{\cP_i}, W \} = - W,
\end{equation}
and using the definition $ \partial_\alpha = x_\alpha W^{-1} \partial_{\cP_i}$, we have
\begin{equation}\label{eqn:dlpartialalpha}
d_W (\partial_\alpha)= -\{\partial_\alpha, W \} = -x_\alpha.
\end{equation}
for $\alpha \in \mathrm{int}\, \sigma_i$. 


\begin{lemma}
$d_W$ and $\delta_-$ agree on the degree $1$ component $ HH^1 (J,J) \cong H^1 (\CFbloop,\delta_+)$.
\end{lemma}

\begin{proof}
Since both $\delta_-$ and $d_W$ are derivations with respect to $\cup=m_{2,0}^{b,b,b}$, it suffices to check if \eqref{eqn:dlpartiali} and \eqref{eqn:dlpartialalpha} hold for their corresponding Floer generators, i.e., if the following hold true:
$$\delta_- \left(\sum_{e \in \mathcal{P}_i} x_e X_e \right) = F(-W)=-\WL \unit, \qquad  \delta_- \left( \sum_{e \in \mathcal{P}_i} x_\alpha W^{-1} x_e X_e \right) = F(-x_\alpha)=- x_\alpha \unit.$$
The former counts a polygon with odd-degree corners containing \( X_e \), assigning its output to one of the \( \mathbf{1}_{L_i} \)'s. By the definition of a perfect matching, exactly one such polygon exists for each \( e \), and the conclusion follows. For the latter, we first need to invert $W$ (by taking localization), and use the linearly of operators (since they are trivial on degree 0 components). This is possible since $HH^1 (J,J) \to HH^1 (J[W^{-1}],J[W^{-1}])$ is known to be injective, or equivalently $HH^1 (J,J)$ is torsion free. (See the proof of \cite[Theorem 4.6]{Wong21}.)
\end{proof}

It follows that \( \delta_- \) and \( d_W \) agree on the subspace multiplicatively generated by degree 0 and degree 1 elements. By Proposition \ref{prop:hhjacw}, the complement of this subspace is generated over the degree 0 component by \( \theta_v \) and \( \theta_{i,j} \). Therefore, the following lemma completes the proof of Proposition \ref{prop:flequalshh}.

\begin{lemma}
$d_W$ and $\delta_-$ agree on the generators $\psi_{i,j}$ and $\theta_v$ in $HH^\ast(J,J)$ and their corresponding generators in $H^\ast (\CFbloop,\delta_+)$.
\end{lemma}

\begin{proof}
Making use of the algebraic compatibility between the BV-operator $\Delta$ and the bracket \eqref{eqn:BVbracket}, one has
$$d_W (\theta_v) = \{W, \theta_v\} = \Delta (W \theta_v) - W \Delta(\theta_v)  = \Delta (W \theta_v).$$
By Lemma \ref{lem:bvopex}, we have $\Delta(\theta_v) =0$ since through Calabi-Yau duality it can be represented by a constant path in $HH_\ast$.  

Note that \( W \theta_v \) can be represented by any boundary cycle in \( \mathcal{Q} \) (a summand of \( W \)) starting at the vertex \( v \). In particular, one may choose \( W \theta_v = x_{i_1} \cdots x_{i_k} \), which is one such cycle obtained from the labeling of the holomorphic polygon \( P \) (with boundary on \( \mathbb{L} \)) passing through \( \pt_{i_v} \). Here, \( i_v \) is chosen such that the component \( L_{i_v} \) of \( \mathbb{L} \) is dual to the vertex \( v \in \mathcal{Q}_0 \). Then, by Lemma~\ref{lem:bvopex},
\[
d_W (\theta_v) = \Delta (W \theta_v) = \sum_i \left( x_{i_a +1} \cdots x_{i_k} \right) \left( x_{i_1} \cdots x_{i_a -1} \right) \bar{X}_{i_a}.
\]
This is precisely \( \delta_{-} (F(\theta_v)) = \delta_{-} (\pt_{i_v}) \), contributed solely by the same polygon \( P \). 

For $\psi_{i,j}$, we use
\[
\partial_{\cP_k} \cup \psi_{i,j}
= \deg_{P_k} (x_{\eta_i})\, x_{\eta_i}\,\theta_v
\]
for all $1 \leq  k \leq N$ ((ii) of Lemma \ref{lem:calcinwong}), where $v \in \mathcal{E}_{i,j}$. Apart from $\psi_{i,j}$, we know that $d_W$ and $\delta_-$ agree on all the other factors. Thus we see that the difference of $d_W (\psi_{i,j}$ and $\delta_- (F(\psi_{i,j}))$ is anahilated by by $\partial_{\cP_k} \cup$. More precisely, we have
$$ \partial_{\cP_k} \cup (d_W (\psi_{i,j}) - F^{-1} \delta_- (F(\psi_{i,j})) ) = 0$$
Hence, if we write $d_W (\psi_{i,j}) - F^{-1} \delta_- (F(\psi_{i,j})) ) := \sum w_i \partial_{\cP_i} + \sum_\alpha c_\alpha \partial_\alpha$, then the above implise
$$ \sum_{i \neq k} w_i \partial_{\cP_i} \cup \partial_{\cP_k} + \sum_{\alpha} c_\alpha \partial_\alpha \cup \partial_{\cP_k} =0$$
for all $k$. On the other hand, one knows explicitly an additive basis of $HH^2 (J)$ from Proposition \ref{prop:hhjacw}, and this forces  $w_i  = c_\alpha =0$ for all $i$ and $\alpha$. Therefore, we have
$ F(d_W (\psi_{i,j})) = \delta_- (F(\psi_{i,j}))$.
%
%
%
%
\end{proof}

\section{The Kodaira-Spencer map $\mathsf{KS}:SH^\ast(X_\Qv) \to \cHH{\mf}$}

The closed string mirror symmetry in our geometric setting is expressed as an equivalence between the symplectic cohomology \( SH^\ast(X_\Qv) \) and the invariant \( \cHH{\mf} \), introduced in Section~\ref{subsec:csupphh}, which is isomorphic to \( \HFbloop \) by Proposition~\ref{prop:flequalshh}. To establish this equivalence, we use a geometrically constructed map
\[
\mathsf{KS}: SH^\ast(X_\Qv) \to \HFbloop,
\]
called the \emph{Kodaira--Spencer map}. This map is analogous to the one introduced in~\cite{FOOO16}, but adapted to the noncompact setting. Its construction is essentially provided in~\cite[Section~4.1]{HJL25}, though now incorporating noncommutative mirror variables. While this does not affect the definition of the relevant moduli spaces, special care must be taken with the ordering of \( b \)-inputs, as the associated variables are noncommutative.

We begin with a brief review of the symplectic cohomology of $X_\Qv$, the domain of $\mathsf{KS}$.

\subsection{Symplectic cohomology of a punctured Riemann surface}\label{subsec:shrs}
Let \( (E, \omega) \) be an exact convex symplectic manifold. There exists a compact domain \( E^{\mathrm{in}} \subset E \) such that \( \partial E^{\mathrm{in}} \) is a contact boundary with contact 1-form \( \alpha \) (defined by \( \alpha = \lambda|_{\partial E^{\mathrm{in}}} \), where \( \omega = d\lambda \)). The complement of \( E^{\mathrm{in}} \) is the conical end of \( E \), given by the symplectomorphism
\begin{equation*}
(E \setminus E^{\mathrm{in}}, \omega) \cong \big( \partial E^{\mathrm{in}} \times (1, \infty),\, d(r\alpha) \big).
\end{equation*}

Let \( H : E \to \mathbb{R} \) be a generic autonomous Hamiltonian, which is \( C^2 \)-small and Morse on \( E^{\mathrm{in}} \), and becomes linear with some slope \( \epsilon > 0 \) for sufficiently large \( r \) on the conical end:
\begin{equation}
H(y, r) = \epsilon r \quad \text{for } r \gg 1,
\end{equation}
where \( (y, r) \in \partial E^{\mathrm{in}} \times (1, \infty) \) are the cylindrical coordinates. The slope \( \epsilon \) is chosen generically so that (i) there are no Reeb orbits on \( \partial E^{\mathrm{in}} \) with period less than or equal to \( \epsilon \), and (ii) no positive multiple of \( \epsilon \) is the period of a Reeb orbit.

For each positive integer \( w \), consider the Hamiltonian Floer complex \( CF^*(wH) \), generated over a coefficient ring \( \mathbb{C} \) by the 1-periodic orbits of the Hamiltonian vector field \( X_{wH} \) on \( E \). To achieve transversality and break the \( S^1 \)-symmetry of the periodic orbits, one replaces \( H \) with a small time-dependent perturbation, still denoted by \( H \).
We use a \( \mathbb{Z}/2 \)-grading on \( CF^*(wH) \), induced by the Conley–Zehnder index. Introducing a formal variable \( \qq \) of degree \( -1 \), with \( \qq^2 = 0 \), the symplectic cochain complex of \( E \) is defined as
\begin{equation*}
SC^*(E) = \bigoplus_{w=1}^\infty CF^*(wH)[\qq].
\end{equation*}
The differential \( d \) on \( SC^*(E) \) is given by
\begin{equation}\label{eqn:SCdifferential}
d(x + \qq y) = (-1)^{|x|} \dd x + (-1)^{|y|} (\qq \dd y + \KK y - y),
\end{equation}
where \( \dd: CF^*(wH) \to CF^{*+1}(wH) \) is the Floer differential, and \( \KK: CF^*(wH) \to CF^*((w+1)H) \) is the continuation map.

Finally, the symplectic cohomology is then defined as the direct limit
\begin{equation}
SH^*(E) = \varinjlim HF^*(wH),
\end{equation}
where the limit is taken over the continuation maps, which are canonical at the level of cohomology.
Symplectic cohomology is equipped with a graded-commutative product, known as the \emph{pair-of-pants product}. Moreover, one can define an \( A_\infty \)-structure \( \{ \mu^k \} \) on \( SC^*(E) \), with \( \mu^1 = d \). We refer the reader to \cite[§4.6]{RS17} for further details.

Returning to our geometric setup, let \( E = X_\Qv \) denote the punctured Riemann surface obtained by removing the vertices of the dual dimer \( \Q^\vee \) from the surface \( \Sigma^\vee \) in which it is embedded. We equip \( E \) with a canonical conical end structure near each puncture. Then $\mathbb{Z}_{>0}$ noncontractible Hamiltonian orbits appear near each puncture, where $\mathbb{Z}_{>0}$ encodes their winding numbers. 

On the other hand, a puncture in \( \Sigma^\vee \) corresponds to a vertex $v_Z$ of the dimer \( \Q^\vee \), and hence to a zigzag cycle $Z$ in $\Q$ by dimer duality. Therefore noncontractible generators of \( SH^\ast(E) \) can be parametrized in terms of zigzag cycles in $\Q$ together with winding numbers. In fact, it can be explicitly computed as the \emph{logarithmic cohomology} of the pair \( (\Sigma^\vee, \Q_0^\vee) \), as shown in \cite{GP20}:
\begin{equation}\label{eqn:logarithmic cohomology presentation for SH}
SH^\ast(E) \cong H^*_{\mathrm{log}}(\Sigma^\vee, \Q_0^\vee) = H^*(\psurf; \mathbb{C}) \oplus \bigoplus_Z t_Z H^*(\mathring{S}_Z; \mathbb{C})[t_Z],
\end{equation}
where the sum is taken over all zigzag cycles $Z$ in $\Q$ (or equivalently for all $v_Z \in \Q^\vee_0$). \( t_Z \) are distinct formal variables, each recording the winding number around the puncture \( v_Z \in \Q_0^\vee \) (dual to $Z$), and each \( \mathring{S}_Z \) is diffeomorphic to \( S^1 \), arising as the boundary of the real-oriented blow-up at the point \( Z \subset \Sigma^\vee \). Consequently,
\begin{equation*}
H^0(\mathring{S}_Z; \mathbb{Z}) = \langle e_Z \rangle, \quad H^1(\mathring{S}_Z; \mathbb{Z}) = \langle f_Z \rangle.
\end{equation*}
The summand \( H^*(\psurf; \mathbb{C})=H^*(X_\Qv;\C) \) is computed using the Morse cohomology of the Hamiltonian \( H \). A specific choice of Morse function will be made in the proof of Lemma~\ref{lem:pquvks}.

The ring structure on \( SH^\ast(E) \) is readily described. The nontrivial products among the noncontractible generators are given by:
\[
e_Z t_Z^n \cdot e_{Z'} t_{Z'}^m = \delta_{Z Z'}\, e_Z t_Z^{n+m}, \quad 
e_Z t_Z^n \cdot f_{Z'} t_{Z'}^m = \delta_{Z Z'}\, f_Z t_Z^{n+m}, \quad 
f_Z t_Z^n \cdot f_{Z'} t_{Z'}^m = 0.
\]
In addition, for an odd-degree critical point \( p \) of \( H \), we have the product:
\[
p \cdot e_Z t_Z^n = f_Z t_Z^n,
\]
where \( Z \) is the puncture toward which the gradient trajectory of \( p \) asymptotically flows. This, together with the unital property of the unit \( \mathbf{1}_E \in H^*(X_\Qv; \mathbb{C}) \) (corresponding to the maximum of \( H \)), completely determines the ring structure on \( SH^\ast(E) \).

\subsection{Construction of $\mathsf{KS}$}\label{subsec:construction of Kodaira-Spencer map}

The \emph{Kodaira–Spencer map}
\begin{equation}\label{eqn:ksksk}
\mathsf{KS}: SH^\ast(E) \to \HFbloop,
\end{equation}
is a variant of the closed-open map constructed in \cite{RS17}\footnote{Their construction also applies in the monotone setting.}:
\begin{equation}\label{eqn:coco}
\mathcal{CO} : SH^\ast(E) \to HH^\ast(\mathcal{W}(E)),
\end{equation}
but incorporates an arbitrary number of boundary insertions taken from a \emph{noncommutative} weak bounding cochain \( b \). The map \eqref{eqn:ksksk} is essentially the same as the Kodaira–Spencer map constructed in \cite[Sec 4.1]{HJL25}, except that here we record the order of boundary inputs more precisely due to their noncommutativity.
In \eqref{eqn:coco}, \( \mathcal{W}(E) \) denotes the wrapped Fukaya category of the exact convex symplectic manifold \( E \), whose objects are exact Lagrangians (equipped with suitable brane structures), and whose morphism spaces are given by the wrapped Floer complex:
\begin{equation*}
\mathrm{hom}_{\mathcal{W}(E)}(L, L') = CW^*(L, L') := \bigoplus_{w \geq 1} CF^*(L, L'; wH)[\qq],
\end{equation*}
where \( \qq \) is a formal variable of degree \(-1\) satisfying \( \qq^2 = 0 \). The differential involves continuation maps between Hamiltonians \( wH \to (w+1)H \), analogously to \eqref{eqn:SCdifferential}. We refer to \cite{AS10} for further details.

In our case, the target of \( \mathsf{KS} \) is the Floer complex of a compact Lagrangian \( \mathbb{L} \), which uses a \emph{slope zero} Hamiltonian. (This is modeled using a Morse–Bott setup; see \S\ref{subsec:CFb for zigzag Lagrangians}.) To accommodate this, we introduce a slight modification of the wrapped Fukaya category.
Define the extended wrapped Fukaya category \( \mathcal{W}_\Diamond(E) \) to have the same objects as \( \mathcal{W}(E) \), but with morphism spaces between compact Lagrangians \( L, L' \) given by:
\begin{equation*}
\mathrm{hom}_{\mathcal{W}_\Diamond(E)}(L, L') := CF^*(L, L')[\qq] \oplus CW^*(L, L').
\end{equation*}
Here, \( CF^*(L, L') \) is computed using a \( C^2 \)-small Hamiltonian (of slope zero). Alternatively, a Morse–Bott model may be used if \( L \) and \( L' \) intersect cleanly. Since continuation maps are quasi-isomorphisms for compact Lagrangians, it does not alter the equivalence class of the category: \( \mathcal{W}_\Diamond(E) \simeq \mathcal{W}(E) \). Also, this construction defines a natural functor from the compact Fukaya category \( \mathcal{F}(E) \) to \( \mathcal{W}(E) \). 

The object \( \mathbb{L} \in \mathcal{W}_\Diamond(E) \) thus has endomorphism algebra
\begin{equation*}
\mathrm{hom}_{\mathcal{W}_\Diamond(E)}(\mathbb{L}, \mathbb{L}) = CF^*(\mathbb{L}, \mathbb{L})[\qq] \oplus CW^*(\mathbb{L}, \mathbb{L}).
\end{equation*}
The natural inclusion
\[
f : CF^*(\mathbb{L}, \mathbb{L}) \hookrightarrow \mathrm{hom}_{\mathcal{W}_\Diamond(E)}(\mathbb{L}, \mathbb{L})
\]
into the first summand is an \( A_\infty \) quasi-isomorphism (with trivial higher-order maps), and hence allows one to transfer the deformation theory by weak bounding cochains, as in \eqref{eqn:boundary deformed Maurer-Cartan base change}. The resulting deformation of \( \mathrm{hom}_{\mathcal{W}_\Diamond(E)}(\mathbb{L}, \mathbb{L}) \) by the noncommutative weak bounding cochain \( b = \sum x_e X_e \) is denoted by
\begin{equation}\label{eqn:wdiall}
\mathrm{hom}_{\mathcal{W}_\Diamond(E)}((\mathbb{L}, b), (\mathbb{L}, b)).
\end{equation}
The quasi-isomorphism from \( \CFb \) to \eqref{eqn:wdiall} follows directly from \( f \).

Accordingly, we extend the symplectic cochain complex to include a slope-zero component
\begin{equation*}
SC^*_\Diamond(E) := H^*(E)[\qq] \oplus \bigoplus_{w=1}^\infty CF^*(wH)[\qq],
\end{equation*}
where the second summand is the usual symplectic cochain complex \( SC^*(E) \), and the first term accounts for the slope-zero (compact) sector. The closed-open map can be adapted to this modification, and we have 
\[
\mathcal{CO}_\Diamond : SC^*_\Diamond(E) \to CC^*(\mathcal{W}_\Diamond(E)).
\]

Within this technical framework, we first define
\begin{equation*}
\tilde{\ksmap} : SC^*_\Diamond(E) \to \mathrm{Hom}_{\EW{E}}((\mathbb{L}, b), (\mathbb{L}, b))
\end{equation*}
as the count of punctured pseudo-holomorphic disks with boundary and interior data specified as in Figure~\ref{fig:Kodaira-Spencer moduli a}. More precisely, consider such a punctured disk with boundary inputs at odd-degree corners labeled by \( X_1, \ldots, X_k \), an output labeled by some \( Y \), and a puncture asymptotic to a Hamiltonian orbit \( \alpha \) (see Figure~\ref{fig:Kodaira-Spencer moduli b}). This disk contributes to \( \tilde{\ksmap} (\alpha) \) a term of the form:
\[
x_k \cdots x_1 \cdot Y \in \mathrm{Hom}_{\EW{E}}((\mathbb{L}, b), (\mathbb{L}, b)).
 \] 
Here, each input \( x_i X_i \) represents a deformation input corresponding to the chosen bounding cochain, and the scalar coefficients \( x_i \) are factored out according to the sign and ordering conventions analogous to \eqref{eqn:nc ainf structure}. Thus, the coefficient for \( Y \) is \( x_k \cdots x_1 \).

Note that \( \mathcal{CO}_\Diamond \) (the closed-open map in this setting) uses the same moduli spaces of punctured disks, although the way contributions are extracted differs. The detailed description of the data carried by pseudo-holomorphic disks in the moduli space can be found in \cite[Section 4.1]{HJL25}, and will not be repeated here for the sake of conciseness.

\begin{figure}[h]
\centering
\hfill
\begin{minipage}{0.475\textwidth}
\centering
{\includegraphics[scale=0.125]{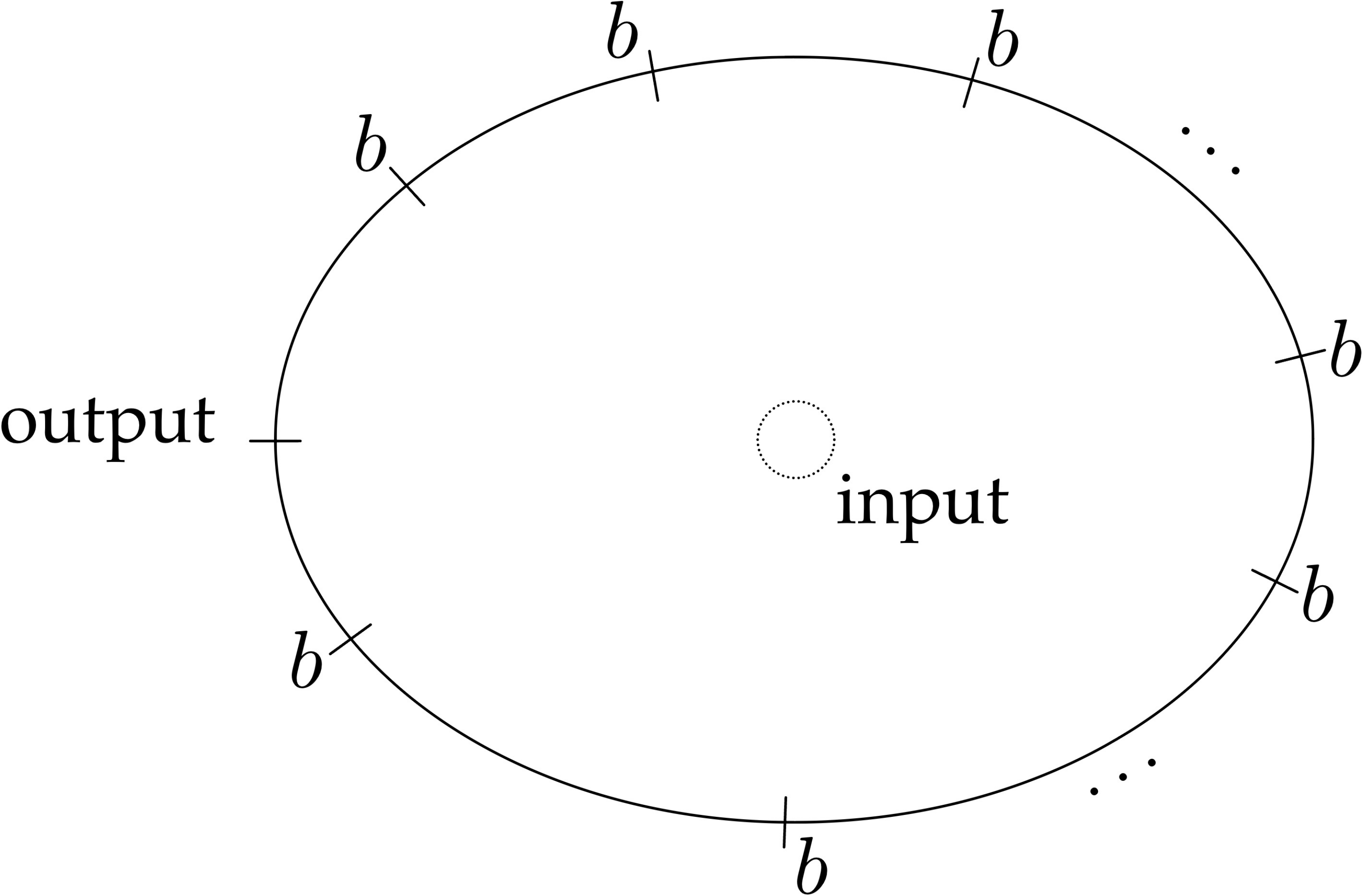}}\caption{The general configuration of a disk in the Kodaira-Spencer moduli}\label{fig:Kodaira-Spencer moduli a}
\end{minipage}
\hfill
\begin{minipage}{0.45\textwidth}
\centering
{\includegraphics[scale=0.12]{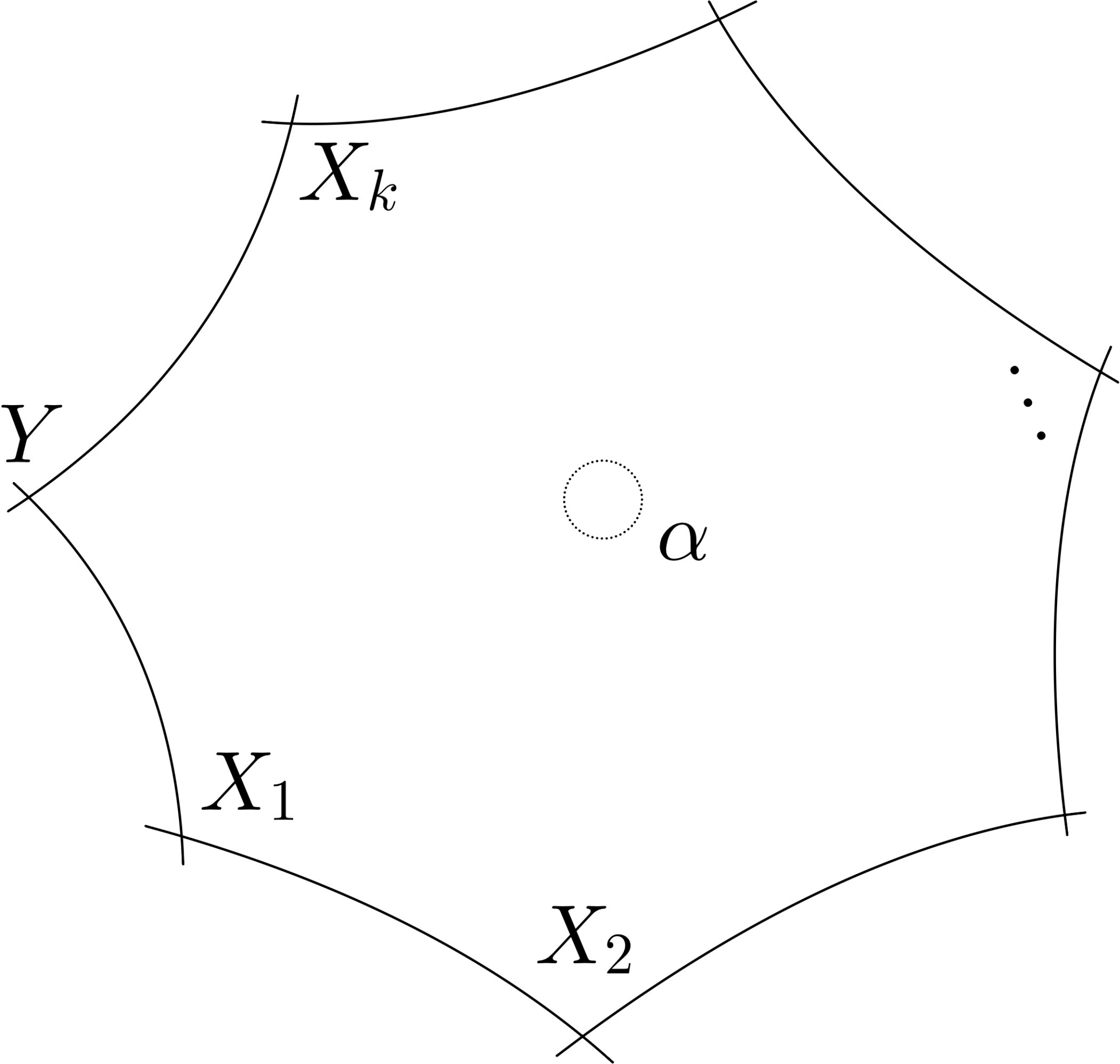}}\caption{A contributing disk for $\ksmap(\alpha)\label{fig:Kodaira-Spencer moduli b}$ with an output $Y$}
\end{minipage}
\hfill
\end{figure}

Finally we define $\mathsf{KS}: SC^\ast(E) \to \CFb$ by the following commutative diagram
\begin{equation*}
\begin{tikzcd}[row sep = 3.5em, column sep = 2.2em]
SC^*(E) \ar[r, hook, "\sim", "qis"'] \ar[d, "\ksmap"]& SC_\Diamond^*(E) \ar[d, "\tilde{\ksmap}"] \\
\CFb \ar[r, hook, "\sim", "f"'] & \mathrm{hom}_{\EW{E}}((\lL,b),(\lL,b)).
\end{tikzcd}
\end{equation*}
The target of $\mathsf{KS}$ is a priori $\CFb$, but it lands on $\CFbloop$.

\begin{prop}\label{prop:KS loops}
The map $\ksmap: SC^\ast(E) \to \CFb$ is a ring homomorphism, and it factors through $\CFbloop$:
\begin{equation*}
\ksmap:SC^*(E) \to \CFbloop ( \subset \CFb).
\end{equation*}
\end{prop}

\begin{proof}
The fact that \( \ksmap: SC^*(E) \to \CFb \) is a ring homomorphism follows from a standard cobordism argument. The proof is essentially the same as that of \cite[Lemma 4.4]{HJL25}, as the relevant disk degenerations preserve the cyclic order of inputs, and hence the ordering of noncommutative variables.

Now, let \( \alpha \) be any generator in the \( \qq^0 \)-component of \( SC^*(E) \). Suppose \( u \) is a pseudo-holomorphic punctured disk contributing to \( \ksmap(\alpha) \), with input \( \alpha \), output \( Y \), and boundary insertions \( x_1 X_1, \ldots, x_k X_k \) from the weak bounding cochain $b$. 
By examining the boundary conditions of the punctured disk, it is straightforward to verify that  
\[h(X_1) =h(Y), \quad t(X_i) = h(X_{i+1}) \text{ for } 1 \leq i \leq k-1, \quad t(X_k) = t(Y).
\]
Otherwise, the corresponding moduli space is empty. The resulting summand \( x_k \cdots x_1 Y \) in \( \ksmap_u(\alpha) \) thus forms a cyclic path in \( \CFbloop \).

\end{proof}

\section{$\mathsf{KS}$ is an isomorphism}
We now establish the closed-string mirror symmetry for punctured Riemann surfaces $X_\Qv:=\psurf$, namely, we prove that the map  
\begin{equation}\label{eqn:KSisom}
\mathsf{KS} : \SH \to \HFbloop(\cong\cHH{\mf}),
\end{equation}
constructed in the previous section, is a ring isomorphism.  
The proof naturally divides into two parts.
First, we restrict $\mathsf{KS}$ to the even-degree component and compute the image of \emph{primitive orbits} (that is, noncontractible Hamiltonian orbits of winding number~$1$).  
An interesting feature is that these orbits do not map to multiples of the unit $\unit$ when there exist parallel zigzag cycles (Definition~\ref{defn:parallels}) within the same homology class.  
This phenomenon explains the appearance of the elements~$\psi_{i,j}$ on the mirror side.  
The odd-degree component forms a finite-rank module over the even-degree component.  
After a careful analysis of its module structure, the argument reduces to comparing the corresponding ranks.

In \ref{subsec:KSmapandsing}, we also give a detailed account of the behavior of the Kodaira--Spencer map in relation to the Landau--Ginzburg models
\[
(Y, \underline{W}) \quad \text{and} \quad (\tilde{Y}_{\mathrm{crep}}, \tilde{W}),
\]
where $Y$ denotes the toric Gorenstein singularity resolved by $J$, and $\tilde{Y}_{\mathrm{crep}}$ its toric crepant resolution.  
Here $\tilde{W} : \tilde{Y}_{\mathrm{crep}} \to \mathbb{C}$ is simply given as the product of the three toric affine coordinates on each local chart~$\mathbb{C}^3$.  
\[
\xymatrix{
(\text{``Spec'' } J, W) \ar[rd]_{\mathrm{NCCR}} && (\tilde{Y}_{\mathrm{crep}}, \tilde{W}) \ar[ld]^{\mathrm{CCR}} \\
& (Y, \underline{W}) &
}
\]

We begin with a more detailed calculation on $\HFbloop$ using the spectral sequence associated with the double complex $\delta(=m_1^{b,b})=\delta_+ + \delta_-$.

\subsection{The $E_2$ page of the spectral sequence and $\HFbloop$}
\label{subsec:Kodaira-Spencer map is an isomorphism}

Recall that we have an additive isomorphism 
\begin{equation}\label{eqn:hfisome2}
\HFbloop \cong E_2
\end{equation}
where $E_2$ denotes the $E_2$-page of the spectral sequence associated with the double complex $\delta_{+} + \delta_{-} = d_A + d_W$.\footnote{The actual $E_2$-page is a periodic repetition of what we refer to as the $E_2$-page here, where ours corresponds precisely to a single period (one component) of this repeated structure. See \cite[5.1]{Wong21}.} 
We will frequently identify generators of $\CFbloop$ and $CC^\ast (J,J)$ \eqref{eqn:CCJJ} via the chain-level identification $F$ \eqref{eqn:ajidn1}, \eqref{eqn:ajidn2} and \eqref{eqn:ajidn3}.

Following \cite{Wong21}, $E_2$ (originally for ($CC^\ast (J,J), d_A + d_W)$) can be additively described as follows: for a fixed $v_0 \in \Q_0$, and $1 \leq r<s<t \leq N$, 
\begin{equation*}
E_2^{\mathrm{even}}= \mathbb{C} \langle x_{\eta_i}^{n},\; x_{\eta_i}^{n} \Psi_{i,j}  : 1 \leq i \leq N, 1<j\leq m_i, n \geq 0 \rangle,
\end{equation*}
\begin{equation}\label{eqn:E2odd1}
E_2^{\mathrm{odd}} = \mathbb{C} \langle x_{\eta_i}^n U,\; x_{\eta_i}^n V,\; x_{\eta_i}^n \Theta_v : 1 \leq i \leq N, v(\neq v_0) \in \mathcal{E}_{i, j>1}) \rangle / V (\eta_i) x_{\eta_i} U = U (\eta_i) x_{\eta_i} V
\end{equation}
where
\begin{equation*}
U= \partial_{\cP_r} - \partial_{\cP_t}, \quad V= \partial_{\cP_s} - \partial_{\cP_t}, \quad \Psi_{i,j} = \psi_{i,j} - \psi_{i,1}, \quad \Theta_v = \theta_v - \theta_{v_0}.
\end{equation*}
By cyclically permuting the index $j$ in $\mathcal{E}_{i,j}$, we may assume that $v_0 \in \mathcal{E}_{i,1}$ for each $i$.  

A few remarks are in order.

\begin{remark}
Continuing from Remark~\ref{rmk:discrem2}, we emphasize that the product in the above expression is induced by the cup product on $HH^\ast(J,J)$, which corresponds to the operation $m_{2,0}^{b,b,b}$ on the chain level, that is, on the $E_0$-page. By contrast, the map $\mathsf{KS}$ in \eqref{eqn:KSisom} is a ring homomorphism with respect to the full product $m_2^{b,b,b}$ on the target.
\end{remark}

\begin{remark}\label{rmk:higerdegodd}
$U$ and $V$, originally in $(J \otimes_\Bbbk E^\ast)_{cyc} \subset CC^\ast (J,J)$, can be seen as elements of $H^1 (T^2;\mathbb{Z})$ where $T^2$ is the torus that $\Q$ embeds in. $U(\eta_i)$ and $V(\eta_i)$ are simply their evaluations at the $1$-cycle $\eta_i$. 
Choose $W:=aU + bV$ for $a,b \in \mathbb{Z}$ such that $\ker W$ does not contain any $\eta_i$'s. When $x_{\eta_i} U \neq 0$, then $U(\eta_i) \neq 0$ from the relation in $E_2^{\mathrm{odd}}$ \eqref{eqn:E2odd1}, and
 $$ x_{\eta_i} U =  c x_{\eta_i} W
 $$
 with $c= \frac{ U(\eta_i) }{a U(\eta_i) + b V(\eta_i) }(\neq 0)$. Similarly any nonzero $x_{\eta_i} V$ can be expressed as a scalar multiple of $x_{\eta_i} W$. Therefore
 \begin{equation}\label{eqn:basisredundant}
 \{ U,V, x_{\eta_i}^n W, x_{\eta_i}^m \Theta_v :  v(\neq v_0) \in \mathcal{E}_{i, j}, 1 \leq i \leq N, 1< j \leq m_i, n \geq 1, m \geq 0  \}
 \end{equation}
forms a basis of $E_2^{\mathrm{odd}}$. Readers are cautioned that the expression in \eqref{eqn:basisredundant} contains a redundancy, in the sense that  
\[
x_{\eta_i}^m \Theta_v = x_{\eta_i}^m \Theta_{v'}
\]
for distinct vertices $v \neq v'$ whenever (and only when) $m \ge 1$ and $v, v' \in \mathcal{E}_{i,j}$. 
To avoid a potential confusion, we write $y_{i,j}^m=x_{\eta_i}^m \Theta_v$ for any $v (\neq v_0) \in  \mathcal{E}_{i,j}$.
\end{remark}

%
%
%
%
%

Note that $E_2$-page is $\mathbb{Z}$-graded, and each can be identified with the component of the associated graded of the homology. The precise identification goes as follows:
$$(E_2)_{\deg 0} = \{ [a_0] \in HF^{\mathrm{even}}_{cyc} ((\lL,b),(\lL,b)) : \mbox{(exterior)}\, \deg a_0 = 0 \},$$
$$(E_2)_{\deg 1} = \{ [a_1] \in HF^{\mathrm{odd}}_{cyc} ((\lL,b),(\lL,b)) : \mbox{(exterior)}\, \deg a_1 = 1 \},$$
$$(E_2)_{\deg 2} =  \dfrac{HF^{\mathrm{even}}_{cyc} ((\lL,b),(\lL,b))}{  \{ [a_0] \in HF^{\mathrm{even}}_{cyc} ((\lL,b),(\lL,b)) : \mbox{(exterior)}\, \deg a_0 = 0 \} },$$
$$(E_2)_{\deg 3} = \dfrac{ HF^{\mathrm{odd}}_{cyc} ((\lL,b),(\lL,b)) }{ \{ [a_1] \in HF^{\mathrm{odd}}_{cyc} ((\lL,b),(\lL,b)) : \mbox{(exterior)}\, \deg a_1 = 1 \}},$$
and hence, we have short exact sequences
\begin{equation}\label{eqn:sesforeven}
0 \to (E_2)_{\deg 0} \to  \Loop{HF}^{\mathrm{even}} \to (E_2)_{\deg 2}  \to 0,
\end{equation}
\begin{equation}\label{eqn:sesforodd}
0 \to (E_2)_{\deg 1} \to  \Loop{HF}^{\mathrm{odd}} \to (E_2)_{\deg 3}  \to 0.
\end{equation}
The isomorphism \eqref{eqn:hfisome2} can be obtained by choosing (noncanonical) splitting of these short exact sequences.

\subsection{$\mathsf{KS}$ image of the sum of ``parallel" orbits}
\label{subsec:Kodaira-Spencer map hits the function ring}
One distinguishing feature of our $\mathsf{KS}$-map, compared to \cite{FOOO16}, is that its image (in even degrees) is not necessarily a scalar multiple of the unit class $\mathbf{1}_{\lL}$. While this introduces some additional complexity in later analysis, there is nevertheless a particular combination of Hamiltonian orbits that always maps to the unit component in $\HFbloop$.

Let $\eta \in H_1(T^2;\mathbb{Z})$ be a homology class such that the zigzag cycles in $\Q$ are supported on $-\eta$. In general, there may be several zigzag cycles in this class, say $Z_1, \dots, Z_m$ (see Definition \ref{defn:parallels}). Consider the degree $0$  Hamiltonian orbits $e_{Z_i} t_{Z_i}$ which wind once around the punctures $v_{Z_i} \in \Q_0^\vee \subset \Sigma^\vee$, the vertex dual to the cycle $Z_i$. Via this duality, we may refer to these as \emph{parallel orbits}.
We will show that the sum of these parallel orbits is always mapped to a multiple of $\mathbf{1}_{\lL}$ under the Kodaira--Spencer map.

\begin{prop}\label{lem:KS isomorphism for even degree}
Let $\{Z_1, Z_2,\ldots,Z_m\}$ be the set of all zigzag cycles in the homology class $- \eta $ for some $\eta\in H_1 (T^2;\mathbb{Z})$. Then we have, modulo $m_1^{b,b}$-coboundary,
\begin{equation}\label{eqn:KS deg0 formula}
\sum_{i=1}^{m} \ks{e_{Z_i}t_{Z_i}}  = x_\eta \unit \mod \mathrm{Im} (m_1^{b,b}).
\end{equation}
\end{prop}

We first need some preparation.
\begin{lemma}\label{prop:computing Kodaira-Spencer map in even degree}
Let $Z$ be a zigzag cycle in $\Q$. Let $e_Z t_Z^n$ denote the symplectic cohomology generator around $v_Z \in \Qvv$ in the presentation \eqref{eqn:logarithmic cohomology presentation for SH}. Let $\eta = -[Z] \in H_1(T^2;\Z)$, the homology class of an anti-zigzag. 
We have
\begin{equation}\label{eqn:formula of Kodaira-Spencer map in even degree}
\WL\ks{e_{Z}t_Z} = \sum_{e\in\zig{Z}}x_\eta \mbb{\xe\Xe} = \sum_{e\in\zag{Z}} x_\eta \mbb{\xe\Xe}.
\end{equation}
\end{lemma}

\begin{proof}
We will only prove the first equality, and the second follows by a symmetric argument. 
Choose a representation $Z=b_L a_L\cdots b_1 a_1$ (unique up to cyclic permutations) such that $a_i$'s are zigs, and
and $\antipm{Z}$ be anti-zigzags of $Z$ as presented in 
\begin{equation*}
\antipos{Z} = \vec{p}_1 \cdots \vec{p}_L \;\, \text{and} \;\; \antineg{Z} = \vec{n}_1 \cdots \vec{n}_L
\end{equation*}
arranged as in \eqref{eqn:anti-zigzags}
(see Figure~\ref{fig:anti-zigzags}). 
We want to show 
\begin{equation}\label{eqn:eztzai}
\WL\ks{e_{Z}t_Z} =  \sum_{i=1}^L x_\eta \mbb{ x_{a_i} X_{a_i}}.
\end{equation}
Let us first compare the degree 2 components (with respect to the exterior grading) of both sides.  
The possible degree 2 outputs of the terms $\mbb{x_{a_i} X_{a_i}}$ are of the following types:
\[
\textcircled{1} \frac{\partial \Phi}{\partial x_{e_p}} \cdot \bar{X}_{e_p}, \quad 
\textcircled{2} \frac{\partial \Phi}{\partial x_{e_n}} \cdot \bar{X}_{e_n}, \quad \text{or} \quad 
\textcircled{3} \frac{\partial \Phi}{\partial x_{b_{i-1}}} \cdot \bar{X}_{b_{i-1}}
\]
for an edge $e_p$ lying in $\vec{p}_i$ and $e_n$ in $\vec{n}_{i-1}$. See Figure \ref{fig:ksimage1}, which is related via dimer duality to the face arrangement near the zigzag cycle $Z$ drawn in Figure \ref{fig:anti-zigzags}. For example, the coefficient $\frac{\partial \Phi}{\partial x_{e_n}}$ reads the edge labelling of the polygon $u_0$ (contributing to $W_\lL$ also) that omits $x_{e_n}$.
\begin{figure}[h]
\includegraphics[scale=0.35]{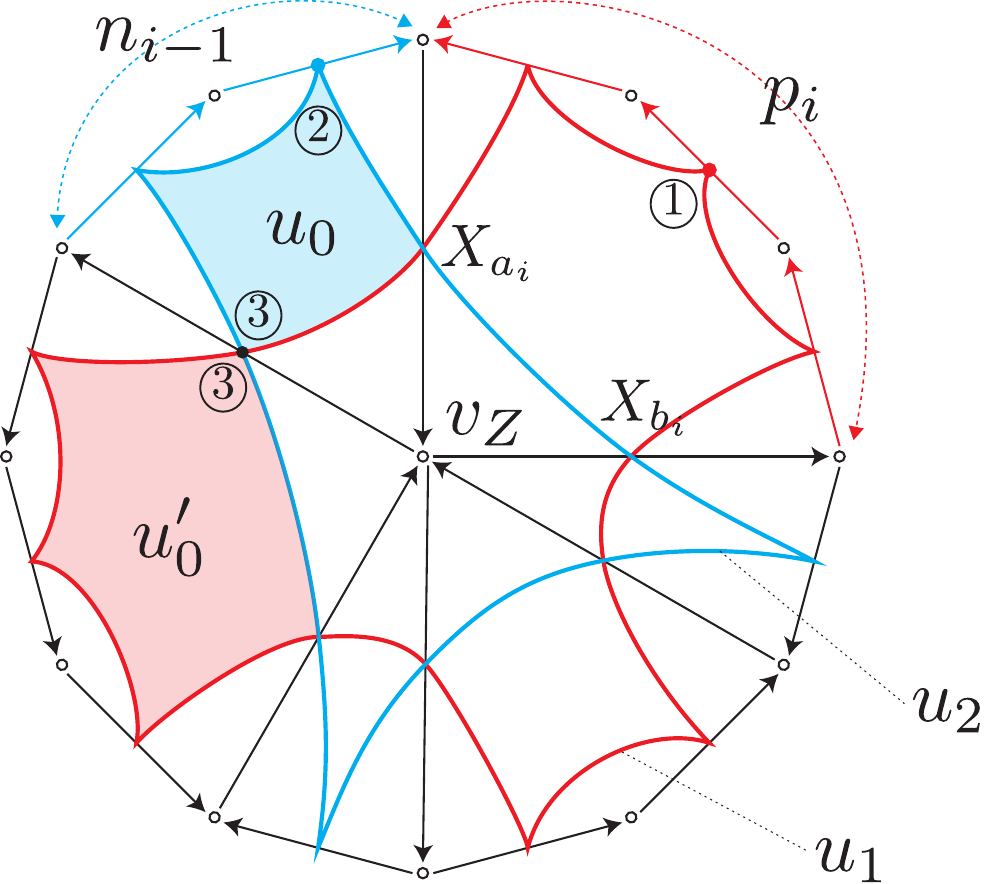}
\caption{Holomorphic polygon counting for $\sum_{i=1}^L x_\eta \mbb{ x_{a_i} X_{a_i}}$}
\label{fig:ksimage1}
\end{figure}

Observe that the type~\textcircled{3} terms cancel among themselves in the sum $\sum_{i=1}^L \mbb{x_{a_i} X_{a_i}}$. In Figure~\ref{fig:ksimage1}, the polygons $u_0$ and $u_0'$ form a cancelling pair for the output $\bar{X}_{b_{i-1}}$.  
As a result, we are left with the following degree 2 terms in the output of the sum $ \sum_{i=1}^L x_\eta \mbb{x_{a_i} X_{a_i}}$:
\[
x_\eta \frac{\partial \Phi}{\partial x_{e_p}} \cdot \bar{X}_{e_p}, \quad 
x_\eta \frac{\partial \Phi}{\partial x_{e_n}} \cdot \bar{X}_{e_n}.
\]
Note that $x_\eta \frac{\partial \Phi}{\partial x_{e_p}}$ corresponds to the path obtained by removing $x_{e_p}$ from the cyclic path
\[
h(\bar{X}_{e_p}) \cdot W_\lL x_\eta \cdot h(\bar{X}_{e_p}),
\]
that is, the component of $W_\lL x_\eta$ based at the vertex $h(\bar{X}_{e_p})$. This precisely matches the summand of $\WL\ks{e_Z t_Z}$ with output $\bar{X}_{e_p}$, which is contributed by the polygon $u_1$ depicted in Figure~\ref{fig:ksimage1}.\footnote{Here, one requires a more careful analysis of the contributing \emph{pseudo-holomorphic polygon} to confirm that the count matches the topological picture. This can be justified by applying the same domain-stretching argument used in the proof of Lemma~4.6, which addresses an essentially identical situation.}
Similarly, the term $x_\eta \frac{\partial \Phi}{\partial x_{e_n}} \cdot \bar{X}_{e_n}$ appears in $\WL\ks{e_Z t_Z}$ as the contribution from the polygon $u_2$ in Figure~\ref{fig:ksimage1}. 

This establishes the equality \eqref{eqn:eztzai} when restricted to the degree 2 components.  
To compare the degree 0 components, a parallel argument applies, with the only difference being that the outputs (maximum points of Morse functions on $L_i$) now lie along the edges of the polygons. We omit the details.
\end{proof}

Now we are ready to prove Proposition \ref{lem:KS isomorphism for even degree}.

\begin{proof}[Proof of Proposition \ref{lem:KS isomorphism for even degree}]
Observe that, for every perfect matching $\cP$, we have
\begin{equation}\label{eqn:m1bbpm}
\sum_{e\in\cP}\mbb{\xe\Xe} = \WL\unit,
\end{equation}
since each holomorphic polygon contributing to \( \WL \) contains exactly one \( X_e \) in $\{ X_e : e \in \mathcal{P}\}$, and every \( \bar{X} \)-term appears twice in the computation of \( m_1^{b,b} \), with canceling coefficients due to the Jacobi relations.
By Theorem \ref{thm:combinatorics of matching polytope for dimers in tori} (iv), there exists a subset $\cJ_\eta \subset \Q_1$ intersecting none of $Z_i$'s such that for any subset $I\subseteq \{1,\ldots,m\}$, 
\begin{equation}\label{eqn:perfect matchings on a boundary edge}
\PM_I:=\cJ_\eta\cup \bigcup_{i \in I} \zig{Z_i} \cup \bigcup_{i \notin I} \zag{Z_i} \subset \Q_1
\end{equation}
defines a boundary matching that sits on the edge of $PM(\Q)$ perpendicular to $\eta$. 

Taking $\mathcal{P}$ in \eqref{eqn:m1bbpm} to be any such $\mathcal{P}_I$ \eqref{eqn:perfect matchings on a boundary edge}, we have
\begin{equation*}\label{eqn:perfect matching decomposition1}
\WL\unit=\sum_{e\in\cP_I}\mbb{\xe\Xe}=\sum_{\substack{i \in I\\ e\in\zig{Z_i}}}\mbb{\xe\Xe} + \sum_{\substack{i \notin I\\ e\in\zag{Z_i}}}\mbb{\xe\Xe} + \sum_{e\in\cJ_\eta}\mbb{\xe\Xe}
\end{equation*}
holds. Multiplying the central elelement $x_\eta$  in $ \MCA$ yields
\begin{equation}\label{eqn:perfect matching decomposition2}
\WL x_\eta \unit=\sum_{i\in I}\sum_{e\in\zig{Z_i}} x_\eta \mbb{\xe\Xe} + \sum_{i \notin I} \sum_{ e\in\zag{Z_i}}x_\eta \mbb{\xe\Xe} + \sum_{e\in\cJ_\eta}x_\eta \mbb{\xe\Xe}.
\end{equation}

On the other hand, Proposition \ref{prop:computing Kodaira-Spencer map in even degree} tells us that for all $1\leq i \leq m$,
\begin{equation*}
\sum_{e\in\zig{Z_i}}x_\eta \mbb{\xe\Xe} = \WL \ks{e_{Z_i}t_{Z_i}},\; \sum_{ e\in\zag{Z_i}}x_\eta \mbb{\xe\Xe} = \WL \ks{e_{Z_i}t_{Z_i}},
\end{equation*}
and, combining with \eqref{eqn:perfect matching decomposition2}, we have
\begin{equation}\label{eqn:wxeta1rem}
\WL x_\eta \unit=\WL\sum_{i=1}^m \ks{e_{Z_i}t_{Z_i}} + \sum_{e\in\cJ_\eta}\mbb{ x_\eta\xe\Xe}.
\end{equation}
Here, we used the fact that $m_1^{b,b}$ is linear over the center of $A_\lL$ (Remark \ref{rmk:linearmkovercenter}).
We claim that the second term in \eqref{eqn:wxeta1rem} is divisible by $W_\lL$. By \cite[Lemma 3.19]{BockABC}, it is enough to show that $x_\eta\xe$ has a strictly positive degree with respect to any corner matching in $PM(\Q)$. First of all, this is true for the two corner matchings whose join contains boundary matchings $\mathcal{P}_I$ \eqref{eqn:perfect matchings on a boundary edge}.\footnote{These are the cases when $I=\phi$ or $I=\{1,\cdots,m\}$ in \eqref{eqn:perfect matchings on a boundary edge}.} This is because these two should also contain $\cJ_\eta$ by Theorem \ref{thm:combinatorics of matching polytope for dimers in tori} (iii), but $e \in \cJ_\eta$ in the second summand of \eqref{eqn:wxeta1rem}. For any corner matching $\mathcal{P}$ other than these two, we have $\deg_\mathcal{P} (x_\eta) \geq 1$ due to Remark \ref{rmk:pmdegandvan}, and hence the claim follows.

Therefore there exists an element $x_b \in J$ such that $W_\lL \cdot x_b = x_\eta x_e$, and in particular $x_b X_e$ is a well defined element in $\CFbloop$. We have
\begin{equation}\label{eqn:perfect matching decomposition3}
\WL x_\eta \unit=\WL\sum_{i=1}^m \ks{e_{Z_i}t_{Z_i}} + \WL\sum_{e\in\cJ_\eta}\mbb{x_b \Xe}
\end{equation}
Because $\Q$ is zigzag consistent, both $J$ is cancellative, and hence, multiplying $W_\lL$ is injective. Therefore we obtain
\begin{equation*}
x_\eta \unit=\sum_{i=1}^m\ks{e_{Z_i}t_{Z_i}} + \sum_{e\in\cJ}\mbb{x_b \Xe}.
\end{equation*}
\end{proof}
Passing to cohomology and using the fact that $\mathsf{KS}$ is a ring homomorphism, we have:
\begin{cor} In the setting of Lemma \ref{lem:KS isomorphism for even degree},
\begin{equation*}
\mathsf{KS} \left(\sum_{i=1}^m e_{Z_i} t_{Z_i}^n \right)  = x_\eta^n \unit.
\end{equation*}
\end{cor}

%
%
%
%
%
%

\subsection{$\mathsf{KS}$ on the even degree cohomology}
From the above discussion, we see that $x_\eta \unit$ solely cannot distinguish parallel orbits, and it will turn out that the additional even cycles $\Psi_{i,j}$ remedy the situation. 

Let $Z_{i,j}$ be parallel zigzag cycles in $\Q$ in class $-\eta_i \in H_1(T^2,\mathbb{Z})$ for $1\leq j\leq m_i$.
For simplicity, we denote their corresponding degree $0$ generators in $SH^\ast(X)$  by
$$ \alpha_{i,j}:=e_{Z_{i,j}} t_{Z_{i,j}}  \qquad \mbox{for} \,\, 1\leq j \leq m_i.$$
We set
\begin{equation}\label{eqn:alphaiii}
\alpha_i = \alpha_{\eta_i} := \sum_{l=1}^{m_i} \alpha_{i,l},\qquad \tau_{i,j}:= \sum_{l=1}^{j-1} \alpha_{i,l}  \quad  (\mbox{for}  \,\, 2 \leq j \leq m_i),
\end{equation}
and write $\left(SH^{\mathrm{even}} (X) \right)_0$ for the subspace of $SH^{\mathrm{even}} (X)$ (multiplicatively) generated by $\alpha_i$'s. Thus
$$ 0 \to \left(SH^{\mathrm{even}} (X) \right)_0 \to SH^{\mathrm{even}} (X) \to \overline{SH^{\mathrm{even}} (X)} \to 0.$$
We have shown in Proposition \ref{lem:KS isomorphism for even degree} that 
$$\mathsf{KS}: \alpha_i^n \mapsto x_{\eta_i}^n \unit$$
which gives the isomorphism $\left(SH^{\mathrm{even}} (X) \right)_0 \stackrel{\cong}{\longrightarrow} (E_2)_{\deg 0}$.
Comparing with \eqref{eqn:sesforodd}, we have
\begin{equation}\label{eqn:evendecompses}
\xymatrix{ 0 \ar[r] &\left(SH^{\mathrm{even}} (X) \right)_0  \ar[r]\ar[d]^{\cong}& SH^{\mathrm{even}} (X) \ar[r]\ar[d]^{\mathsf{KS}|_{\mathrm{even}}} & \overline{SH^{\mathrm{even}} (X)} \ar[r]\ar[d]^{h} & 0     \\
0 \ar[r]  &(E_2)_{\deg 0} \ar[r] &HF^{\mathrm{even}}_{cyc} ((\lL,b),(\lL,b)) \ar[r] &   (E_2)_{\deg 2}\ar[r] &      0 }.
 \end{equation}

On the other hand, using the explicit description of the ring structure on $SH^{\mathrm{even}} (X) $ given in \ref{subsec:shrs}, we see that $\left(SH^{\mathrm{even}} (X) \right)_0$ is a subring of $SH^{\mathrm{even}}  (X)$, and the set 
$$\left\{\alpha_i^n \cdot \tau_{i,j} =\sum_{l=1}^{j-1}  \alpha_{i,l}^{n+1}  = \tau_{i,j}^{n+1} : 1\leq i \leq N, 1< j \leq m_i, n \in \mathbb{Z}_{\geq 0} \right\} (\subset SH^{\mathrm{even}} (X) ) $$ 
projects to a basis of the quotient $\overline{SH^{\mathrm{even}} (X)}$.
It follows from convergence of the spectral sequence that the map $h$ on the last column of \eqref{eqn:evendecompses} is simply extracting the degree 2 component from the actual image of $\mathsf{KS}$. Namely, if $\mathsf{KS}(\alpha) = [a_0+a_2] \in \Loop{HF}^{\mathrm{even}}$ then, it maps $\overline{\alpha} \in \overline{SH^{\mathrm{even}} (X)}$ to $[a_2] \in (E_2)_{\deg 2}$. 
We claim that this also gives an isomorphism.

\begin{lemma} $[\tau_{i,j}^{n+1}] \in \overline{SH^{\mathrm{even}} (X)}$ maps to $x_{\eta_i}^n \Psi_{i,j} \in (E_2)_{\deg 2}$ under the induced map $h$ in \eqref{eqn:evendecompses}.  
\end{lemma}
\begin{proof}
We first prove $h( [\tau_{i,j}] )= \Psi_{i,j}$.
Let us find the image of $\alpha_{i,j}$ in $(E_2)_{\deg 2}$ under the induced map $h$ from $\mathsf{KS}$. If we write the antizigzags associated with $\alpha_{i,j}$ appearing in the boundary of $\mathcal{E}_{i,j}$ and $\mathcal{E}_{i.j-1}$, respectively as $x_1 \cdots x_k$ and $y_1 \cdots y_l$ (up to cyclic permutation), then,
$$\overline{\alpha_{i,j}} \stackrel{h}{\mapsto} \sum x_{a+1} \cdots x_k x_1 \cdots x_{a-1} \bar{X}_a - \sum y_{b+1} \cdots y_l x_1 \cdots x_{b-1} \bar{X}_b.$$
(By definition of $h$, degree 0 part of the actual image of $\mathsf{KS}$ is ignored.)
As shown in Figure \ref{fig:psiij}, take a vertex $v$ and $v'$ on these antizigzags. In particular, we have $v \in \mathcal{E}_{i,j}$ and $v' \in \mathcal{E}_{i,j-1}$.
\begin{figure}[h]
\includegraphics[scale=0.4]{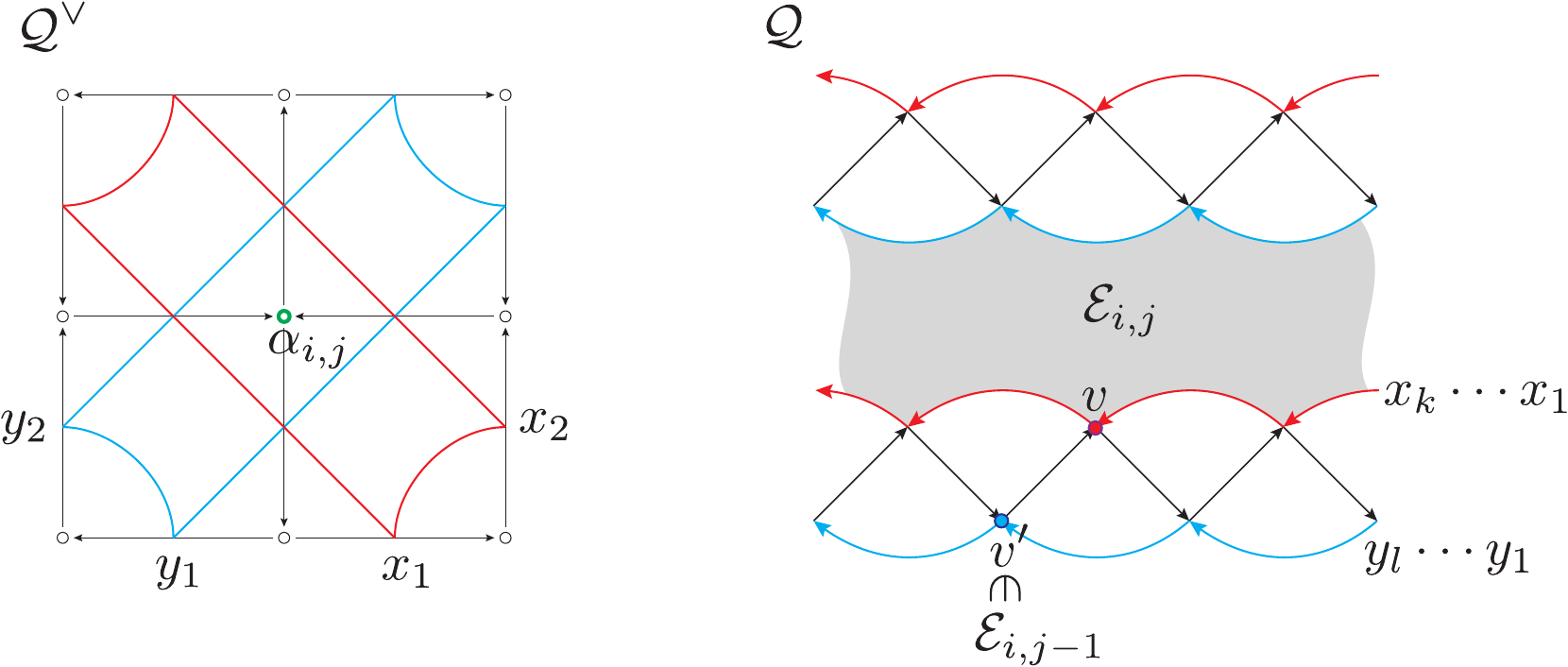}
\caption{}
\label{fig:psiij}
\end{figure}
Therefore, 
$$ \psi_{i,j} = \Delta (x_{\eta_i}\theta_v) = \sum x_{a+1} \cdots x_k x_1 \cdots x_{a-1} \bar{X}_a,$$
$$ \psi_{i,j-1} = \Delta (x_{\eta_i}\theta_{v'}) = \sum y_{b+1} \cdots y_l x_1 \cdots x_{b-1} \bar{X}_b,$$
and hence $\overline{\alpha_{i,j}}$ maps to $\psi_{i,j} - \psi_{i, j-1}$ under $h$,
and
$$h:\overline{\tau_{i,j}} \mapsto \psi_{i,2} - \psi_{i, 1} + \psi_{i,3} - \psi_{i,2} + \cdots + \psi_{i,j} - \psi_{i, j-1} = \psi_{i,j} - \psi_{i,1} =\Psi_{i,j}.$$

The actual image of $\tau_{i,j}$ in $\Loop{HF}^{\mathrm{even}}$ is some lifting, say  $\tilde{\Psi}_{i,j}$, of $\Psi_{i,j}$ obtained by adding a suitable degree $0$ element to it. Since $\mathsf{KS}$ is a ring homomorphism and the image of $\alpha_i$ is a multiple of unit,
$$ \alpha_i^n \cdot \tau_{i,j} \left(=\sum_{l=1}^{j-1}  \alpha_{i,l}^{n+1}  = \tau_{i,j}^{n+1} \right) \mapsto x_{\eta_i}^n  \tilde{\Psi}_{i,j}.$$
We see that the map $\overline{SH^{\mathrm{even}} (X)} \to (E_2)_{\deg 2}$ sends $[\tau_{i,j}^{n+1}]$ to $x_{\eta_i}^n \Psi_{i,j}$. (Here, $\alpha_i^n \cdot \tau_{i,j}  = \tau_{i,j}^{n+1}$ uses the knowledge of the ring structure on $SH^{\mathrm{even}} (X)$.)
\end{proof}

Therefore, by $5$-lemma, we have
\begin{prop}
$\mathsf{KS}|_{\mathrm{even}} : SH^{\mathrm{even}} (X) \to HF^{\mathrm{even}}_{cyc} ((\lL,b),(\lL,b))$ is a ring isomorphism.
\end{prop}

\subsection{$\mathsf{KS}$ on the odd degree cohomology}

Observe that the short exact sequence \eqref{eqn:sesforodd} is indeed a short exact sequence of $R:=(SH^{\mathrm{even}} (X))_0 = \Loop{HF}^{0}$-modules:
$$0 \to (E_2)_{\deg 1} \to  HF^{\mathrm{odd}}_{cyc} ((\lL,b),(\lL,b)) \to (E_2)_{\deg 3}  \to 0.$$
Recall from \ref{subsec:Kodaira-Spencer map is an isomorphism} that $(E_2)_{\deg 1}$ is generated over $R$ by the two elements $U$ and $V$, where we choose them to be
$$ U:= \partial_{\cP_i} - \partial_{\cP_{i-1}}, \qquad V:= \partial_{\cP_{i+1}} - \partial_{\cP_{i-1}}$$
for some fixed corner matching $\mathcal{P}_i$.
We denote by $\underline{(E_2)_{\deg 1}}$ the $\mathbb{C}$-vector space generated by $U$ and $V$ so that
\begin{equation}\label{eqn:undere21}
(E_2)_{\deg 1} = R \cdot \underline{(E_2)_{\deg 1}}.
\end{equation}
Likewise, we may write
\begin{equation}\label{eqn:undere23}
 (E_2)_{\deg 3} = R \cdot \underline{(E_2)_{\deg 3}}
 \end{equation}
where $\underline{(E_2)_{\deg 3}}$ is the  $\mathbb{C}$-vector space generated by $\Theta_v$'s.

Analogously, $SH^{\mathrm{odd}} (X)$ admits a decomposition
\begin{equation}\label{eqn:shcontnoncont}
SH^{\mathrm{odd}} (X) = H^1 (X;\mathbb{C}) \oplus SH^{\mathrm{odd}}_+ (X),
\end{equation}
and $SH^{\mathrm{odd}} (X)$ is generated by $H^1 (X;\mathbb{C}) $ over $R$.\footnote{This is from the explicit computation in \ref{subsec:shrs} for the punctured Riemann surface $X_\Qv$. In general, the three terms in \eqref{eqn:shcontnoncont} form a short exact sequence only, which may not be canonically split.} Recall that in our setup, \( H^1(X; \mathbb{C}) \) is modeled on the Morse homology of a proper Morse function on \( X \), with a choice of Morse function to be specified shortly.

\begin{lemma}\label{lem:pquvks}
There exists two independent elements $p$ and $q$ in $H^1 (X;\mathbb{C}) $ that are mapped to $U$ and $V$ under $\mathsf{KS}$.
\end{lemma}
\begin{proof}
We choose a Morse function $h=h_{X_\Qv}$ on $X_\Qv$ that fits into the dimer structure of $\Q^\vee$, as depicted on the left of Figure \ref{fig:choicemorse}. It has an odd degree critical point $p_e$ for each edge $e \in \Q^\vee_1$, and an even degree one $p_f$ for each face $f \in \Q^\vee_2$. Note that this is not a perfect Morse function, but we are interested  only in odd degree critical points for which the differential vanishes. To avoid (self-)intersection points $\mathbb{L}$, we slightly perturb each $p_e$ away from the edge $e$ so that it lies in the positive face adjacent to $e$.
\begin{figure}[h]
\includegraphics[scale=0.45]{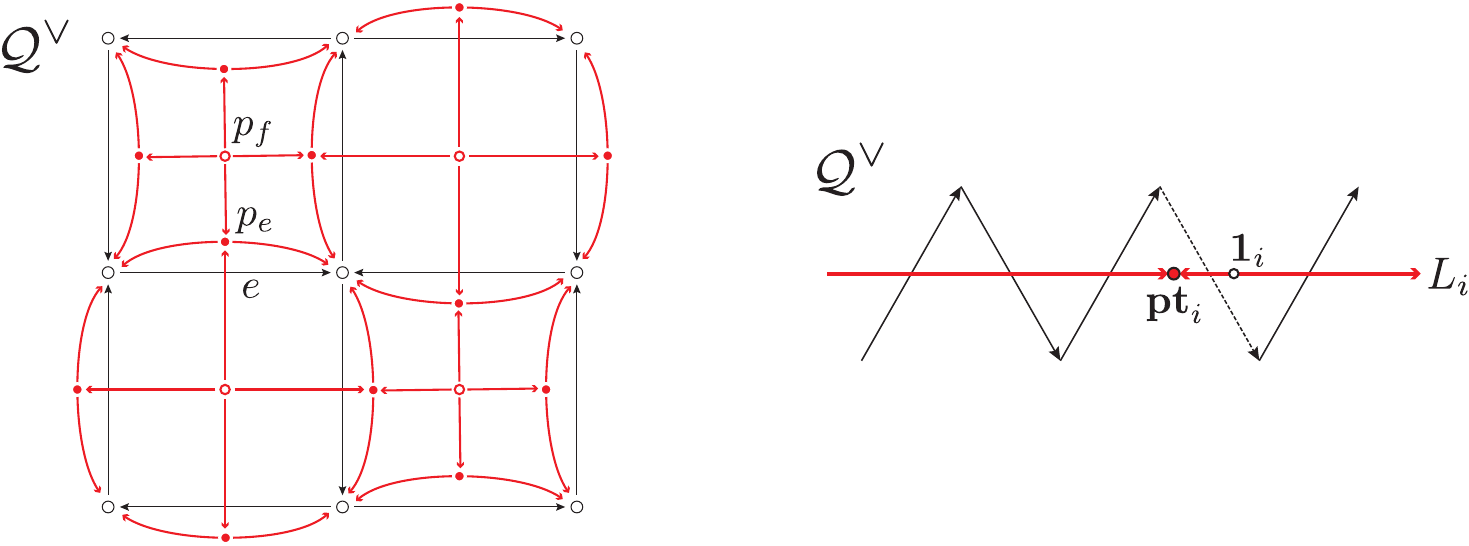}
\caption{Critical points and gradient flows of the Morse functions $h$ and $f_{L_i}$}
\label{fig:choicemorse}
\end{figure}

We next specify a perfect Morse function $f_{L_{i_v}}$ on each Lagrangian (irreducible) component $L_{i_v}$ for each $v\in \Q_0$. Recall that $L_{i_v}$ is obtained by smoothing the zigzag path in $\Q^\vee$ dual to $v$. As shown in the right of Figure \ref{fig:choicemorse}, we fix a zig edge in this zigzag path, and locate the maximum and the minimum, respectively, on the positive and the negative face adjacent to this zig. With this choice, the gradient trajectory flows along the orientation $L_v$ except on a short segment that intersects this fixed zig edge.

For a fixed $i_0$, we consider all zigzag paths $Z_{i_0,j}$ $(1 \leq j \leq m_{i_0})$ in $\Q$ in class $-\eta_{i_0,j} \in H^1 (T^2,\mathbb{Z})$.
By dimer duality, they correspond to vertices $v_{Z_{i_0,j}} \in \Q^\vee_0$.
We denote by $a_1, b_1, \cdots, a_L, b_L$ edges of $\Q^\vee$ incident to $v_{Z_{i_0,j}}$ arranged in counterclockwise order where $a_i$ and $b_i$ indicate  incoming and outgoing, respectively (this is consistent with Figure \ref{fig:dualtozigzag}).
Let
\[p_{i_0,j} :=  \sum_{l=1}^k p_{a_l}  + \sum_{l=1}^k p_{b_l}.\]

Let \( L_1 \) and \( L_2 \) be the two Lagrangian branches in \( \lL \) that intersect at $X_{a_l}$ (lying on the edge \( a_l \)).\footnote{It is possible that $L_1=L_2$, in which case $X_{a_l}$ is a self-intersection point of a single Lagrangian circle component.} Then \( \mathsf{KS}(p_{a_l}) \) is computed by counting pearl trajectories, each consisting of an ambient Morse trajectory of $h$ from \( p_{a_l} \), followed by a Morse trajectory of $f_{L_i}$, which then flows to one of the following:
\[
\textcircled{1} \,\text{a self-intersection point } X, \quad \text{or} \quad \textcircled{2}\, \text{a minimum point } \pt_i \text{ on } L_i,
\]
See Figure~\ref{fig:morsemorse} for type \textcircled{1}. Type  \textcircled{2} is obtained by letting the gradient trajectory of $f_{L_i}$ flow to the minimum. Case \textcircled{1} contributes an output of the form \( xX \) to \( \mathsf{KS}(p_{a_l}) \), while case \textcircled{2} contributes simply \( \pt_i \). The image \( \mathsf{KS}(p_{b_l}) \) can be computed similarly.

\begin{figure}[h]
\includegraphics[scale=0.6]{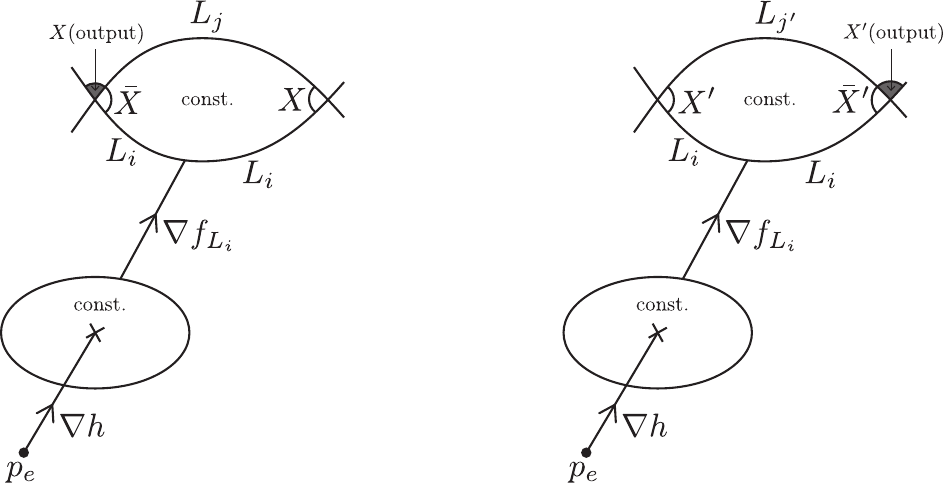}
\caption{Two pearl trajectories consisting of constant disks and resulting in $xX$ (or $x' X'$) as outputs}
\label{fig:morsemorse}
\end{figure}

Let \( Q \) be the punctured polygon formed by connecting the midpoints of the edges \( a_1, b_1, \ldots, a_L, b_L \) in order, as illustrated in Figure~\ref{fig:noptone} (see the shaded rectangle in the center). Since the critical points 
\begin{equation}\label{eqn:papbarr}
 p_{a_1}, p_{b_1}, \ldots, p_{a_L}, p_{b_L} 
 \end{equation}
arise as small perturbations of the vertices of \( Q \), the terms in \( \mathsf{KS}(p_{i_0,j}) \)  of type \( \pt_i \) cancel cyclically. As a result, \( \mathsf{KS}(p_{i_0,j}) \) lies in \( (E_2)_{\deg 1} \subset SH^{\text{odd}}(X) \) (i.e., It receives contributions only from case \textcircled{1} above.). In what follows, we compute \( \mathsf{KS}(p_{i_0,j}) \) by direct counting.

\begin{figure}[h]
\includegraphics[scale=0.45]{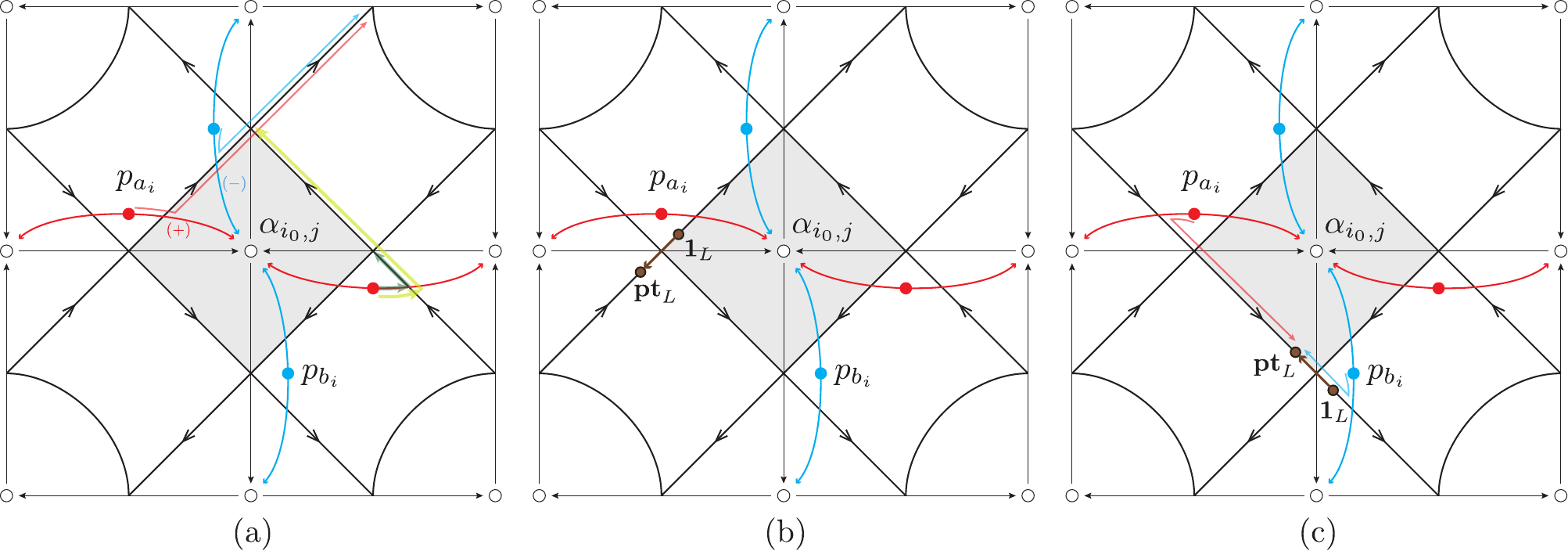}
\caption{Neighboring faces of $\alpha_{i_0,j}$ in $\Q^\vee$ and ambient Morse flows. The red and blue trajectories in (a) are cancel pairs whereas the green and yellow ones contribute $x_{a_j} X_{a_j}$ and $x_{b_j} X_{b_j}$ (with opposite signs). (b) and (c) show the effect of presence of $\mathbf{1}_{L_v}$ and $ \pt_{L_v}$ on $\partial Q$, respectively.}
\label{fig:noptone}
\end{figure}
If the boundary edges of \( Q \) do not carry any insertions of \( \mathbf{1}_{L_v} \) or \( \pt_{L_v} \), then Figure~\ref{fig:noptone}  (a) shows that the image is given by
\begin{equation}\label{eqn:kspi0}
\mathsf{KS}(p_{i_0,j}) =  \sum x_{a_l} X_{a_l} - \sum x_{b_l} X_{b_l},
\end{equation}
since all other \( xX \)-terms cancel out due to consecutive pearl trajectories of type~\textcircled{1}, when arranged according to \eqref{eqn:papbarr}.
The sign difference for this cancellation arises from the distinct orientations of the gradient flows of \( h \) in their contributing pearl trajectories. 
On the other hand, the terms \( x_{a_l} X_{a_l} \) and \( \sum x_{b_l} X_{b_l} \) in \eqref{eqn:kspi0} appear with opposite signs due to the differing configurations of their contributing pearl trajectories, as illustrated in the two diagrams in Figure~\ref{fig:morsemorse}. 
(This is essentially the same reason why \( m_2(X, \bar{X}) \) and \( m_2(\bar{X}, X) \) have opposite signs, \cite[(10.2)]{Sei11}.)\footnote{Adopting a different sign convention results only in an overall change of signs in \eqref{eqn:kspi0}, which does not affect our argument (since we may replace $p_{i_0,j}$ by $-p_{i_0,j}$ if necessary). What truly matters is the relative sign difference on the right-hand side 
of \eqref{eqn:kspi0}.}

Even if there is $\mathbf{1}_{L_v}$ or $\pt_{L_v}$ on the boundary of $Q$, Figure~\ref{fig:noptone} (b) and (c) show that it does not affect $\mathsf{KS}(p_{i_0,j})$ making use of our choice of (relative) positions of critical points of $h$ and $f_{L_i}$. 

Now, let $p:= \sum_j p_{i_0,j}$. Observe that  $\sum x_{a_l} X_{a_l}$ in $\mathsf{KS}(p_{i_0,j})$ is the sum of all zigs in the zigzag cycle $Z_{i_0,j}$ whereas $\sum x_{b_l} X_{b_l}$ is the sum of all zags in $Z_{i_0,j}$.
Therefore
$$\mathsf{KS}(p) =\partial_{\cP_{i_0-1}}- \partial_{\cP_{i_0}}  $$
by properties of the corner matchings $\cP_{i_0}$ and $\cP_{i_0+1}$ in Theorem \ref{thm:combinatorics of matching polytope for dimers in tori}. 
Lastly, we take $q:= \sum_j p_{i_0 -1,j}$, run the same argument to have
$$\mathsf{KS}(q)=\partial_{\cP_{i_0-2}} - \partial_{\cP_{i_0 -1}}.
$$
Therefore, $q$ and $p$ map to $U$ and $V$, respectively, for appropriate choices of $r,s,t$.
\end{proof}

Let $K$ be the subspace of $H^1 (X;\mathbb{C}) $ generated by $p$ and $q$.
Then, by Lemma \ref{lem:pquvks}, $K$ maps isomorphically onto $\underline{(E_2)_{\deg 1}} \subset (E_2)_{\deg 1}$. Thus the restriction
$$\mathsf{KS}|_{H^1 (X;\mathbb{C})} : H^1 (X;\mathbb{C}) \to SH^{\mathrm{odd}} (X)$$ 
descends to $\tilde{r}: H^1 (X;\mathbb{C}) / K \to (E_2)_{\deg 3}$, i.e., 
\begin{equation*}
\xymatrix{ 0 \ar[r] &K \ar[r]\ar@{^{(}->}[d]& H^1 (X;\mathbb{C})  \ar[r]\ar[d] &  H^1 (X;\mathbb{C}) /K  \ar[r]\ar[d] & 0   \\
0 \ar[r]  &(E_2)_{\deg 1} \ar[r] & HF^{\mathrm{odd}}_{cyc} ((\lL,b),(\lL,b)) \ar[r] &   (E_2)_{\deg 3}\ar[r] &     0 }.
 \end{equation*}
We claim that the image actually lies in $\underline{(E_2)_{\deg 3}}= \mathbb \langle \Theta_v : v (\neq v_0) \in \Q_0 \rangle$. Moreover,

\begin{lemma}\label{lem:tilderisom}
$\tilde{r}: H^1 (X;\mathbb{C}) / K \to (E_2)_{\deg 3}$ factors through an isomorphism
\begin{equation}\label{eqn:rh1xk}
r:  H^1 (X;\mathbb{C}) / K \stackrel{\cong}{\longrightarrow} \underline{(E_2)_{\deg 3}}.
\end{equation}
\end{lemma}

\begin{proof}
Consider an odd degree critical point $p_e$ of the ambient Morse function $h$ constructed in the proof of Lemma \ref{lem:pquvks} (for $e \in \Q^\vee_1$). If we denote by $L_1$ and $L_2$ two Lagrangian branches in $\lL$ meeting at $X_e$, then we see that the image $\bar{p_e} \in H^1 (X;\mathbb{C}) / K$ maps to 
\begin{equation}\label{eqn:imagepebar}
\bar{p_e} \mapsto \theta_{t(e)} - \theta_{h(e)}(=\Theta_{t(e)} - \Theta_{h(e)} ) \in \underline{(E_2)_{\deg 3}},
\end{equation}
where we now view $e$ as an edge in $\Q$ joining $v_{L_1}$ and $v_{L_2}$. Notice that we do not have $U,V$ terms in the image, due to the projection $\HFbloop^{\mathrm{odd}} \to (E_2)_{\deg 3}$. Thus $r$ in \eqref{eqn:rh1xk} is well-defined.

Since $\Q$ is a connected graph, the map $r$ must be surjective. Indeed, if  
$e_l e_{l-1} \cdots e_1$ is a path in $\Q$ from $v_0$ to $v$, then by \eqref{eqn:imagepebar} we have
\begin{equation}\label{eqn:eieiei}
\sum_{i=1}^l p_{e_i} = \theta_v - \theta_{v_0} = \Theta_v.
\end{equation}
Furthermore, the combinatorial correspondence from dimer duality implies that the number 
$|\Q_0|$ of vertices of $\Q$ coincides with the elementary (normalized) area of its matching polygon $MP(\Q)$. 
This area can be expressed as
\[
|\Q_0| = 2I + B -2 = 2g+N -2,
\]
where $I$ and $B$ denote the numbers of interior and boundary lattice points of $MP(\Q)$, respectively, 
and $g$ and $N$ are the genus and the number of punctures of 
$X = \Sigma^\vee \setminus \Q^\vee_0$. The first equality follows from Pick's formula, and the second from \cite[Theorem 1.50]{BockABC}. Therefore
\begin{equation*}
\begin{array}{rcl}
\dim H^1 (X;\mathbb{C}) / K  &=& \dim H^1 (X;\mathbb{C}) -2 = (2g + N-1 ) -2 \\
&=& 2I +B -3 = |\Q_0| -1 \\
&=& \dim \underline{(E_2)_{\deg 3}}.
\end{array}
\end{equation*}
We conclude that $r$ is an isomorphism of $\mathbb{C}$-vector spaces.
\end{proof}

In summary, we have
\begin{equation*}
\xymatrix{ 0 \ar[r] &K \ar[r]\ar[d]^{\cong}& H^1 (X;\mathbb{C})  \ar[r]\ar@{^{(}->}[dd] &  H^1 (X;\mathbb{C}) / K \ar[r]\ar[d]^{\cong} & 0   \\
&\underline{(E_2)_{\deg 1}}\ar@{^{(}->}[d]&&\underline{(E_2)_{\deg 3}}\ar@{^{(}->}[d]&
  \\
0 \ar[r]  &(E_2)_{\deg 1} \ar[r] & HF^{\mathrm{odd}}_{cyc} ((\lL,b),(\lL,b)) \ar[r] &   (E_2)_{\deg 3}\ar[r] &      0 }.
 \end{equation*}
where the injectivity of the middle map follows from the $5$-lemma.
Therefore, $H^1 (X;\mathbb{C})$ isomorphically maps onto some subspace $\underline{HF}$ of $HF^{\mathrm{odd}}_{cyc} ((\lL,b),(\lL,b))$.
From \eqref{eqn:undere21} and \eqref{eqn:undere21}, we see that $HF^{\mathrm{odd}}_{cyc} ((\lL,b),(\lL,b))$ is generated by $\underline{HF}$ over $R$. In particular, we have:
\begin{lemma}\label{lem:kssurodd}
$\mathsf{KS}_{\mathrm{odd}}: SH^{\mathrm{odd}} (X) \to \Loop{HF}^{\mathrm{odd}} $ is surjective.
\end{lemma}

In order to prove $\mathsf{KS}_{\mathrm{odd}}$ is an isomorphism, it suffices to show that the induced map
\begin{equation}\label{eqn:ksinducedoddinj}
SH_+^{\mathrm{odd}} (X) (\cong SH^{\mathrm{odd}} (X) / H^1 (X;\mathbb{C})) \to \overline{HF}: =  HF^{\mathrm{odd}}_{cyc} ((\lL,b),(\lL,b)) /\underline{HF}
\end{equation}
is an isomorphism. (In fact, it suffices to show that \eqref{eqn:ksinducedoddinj} is injective due to Lemma \ref{lem:kssurodd}.) For this, we put $\overline{\Loop{HF}^{\mathrm{odd}}}$ in the analogous short exact sequence from the spectral sequence, and look at
\begin{equation*}
\xymatrix{ 
&&SH_+^{\mathrm{odd}} (X) \ar[d]&&  \\
0 \ar[r]  &\overline{(E_2)_{\deg 1}} \ar[r]  & \overline{HF} \ar[r]  &   \overline{(E_2)_{\deg 3}} \ar[r]  &      0}.
 \end{equation*}
The above short exact sequence is obtained by taking quotient of the first row by the second in the diagram below:
\begin{equation*}
\xymatrix{ 
&0 \ar[d] &0 \ar[d] &0 \ar[d]& \\
0 \ar[r] &\underline{(E_2)_{\deg 1}} \ar[r]\ar[d] & \underline{HF}  \ar[r] \ar[d] & \underline{(E_2)_{\deg 3}} \ar[r]\ar[d]^{\cong} & 0   \\
0 \ar[r]& (E_2)_{\deg 1}\ar[r]  \ar[d]& HF^{\mathrm{odd}}_{cyc} ((\lL,b),(\lL,b)) \ar[r]\ar[d]& (E_2)_{\deg 3} \ar[d]\ar[r]& 0
  \\
0 \ar[r]  &\overline{(E_2)_{\deg 1}} \ar[r] \ar[d]& \overline{HF} \ar[r] \ar[d]&   \overline{(E_2)_{\deg 3}} \ar[r] \ar[d]&      0 \\
&0&0&0& }.
 \end{equation*} 
 From Remark \ref{rmk:higerdegodd}, we know that $\overline{(E_2)_{\deg 1}} $ and $ \overline{(E_2)_{\deg 3}}$ admit free (additive) generators
 $$\{x_{\eta_i}^n W : n \geq 1   \} \quad \mbox{and} \quad \{ x_{\eta_i}^m \Theta_v :  v(\neq v_0) \in \mathcal{E}_{i, j>1}, m >1 \}, $$ 
respectively. 
Now we are ready to prove
\begin{prop}
$\mathsf{KS}_{\mathrm{odd}}: SH^{\mathrm{odd}} (X) \to HF^{\mathrm{odd}}_{cyc} ((\lL,b),(\lL,b)) $ is an ($R$-module) isomorphism.
\end{prop}

\begin{proof}
From our discussion above, it only remains to show that \eqref{eqn:ksinducedoddinj} is injective. 
Recall from Remark \ref{rmk:higerdegodd} that
$$ \{ x_{\eta_i}^n W, x_{\eta_i}^m \Theta_v :  v(\neq v_0) \in \mathcal{E}_{i, j>1}, n \geq 1, m \geq 0  \}$$
forms a basis of $\overline{(E_2)_{\deg 1}} \oplus\overline{(E_2)_{\deg 3}}$, where $y_{i,j}^m = x_{\eta_i}^m \Theta_v$ for any $v \in \mathcal{E}_{i,j}$. We choose a particular $v=v_{i,j} \in \mathcal{E}_{i,j}$ as follows, which will help us fixing an explicit representative for $y_{i,j}^m$.
As shown in Figure \ref{fig:tthetav}, take any zigzag path $Z$ in $\Q$ starting from $v_0$ that is transversal to $\pm \eta_i$. This path will pass through regions $\mathcal{E}_{i,2}, \mathcal{E}_{i,3}, \cdots, \mathcal{E}_{i,j}$. For each $j$, we first choose $v_{i,j}$ to be the point at which the zigzag path $Z$ entering the region $\mathcal{E}_{i,j}$. Thus we have $y_{i,j}^m = x_{\eta_i}^m \Theta_{v_{i,j}}$ in this case.

Suppose the zigzag path $Z$ from $v_0$ to $v_{i,j}$ is written as $x_{l_j} x_{l_j-1} \cdots x_2 x_1$. Each $x_i$ has the corresponding edge denoted by the same letter $x_i$ in the dual $\Q^\vee$. We then take
$$  \xi_{v_{i,j}} := \sum_{a=1}^{l_j} p_{x_a}.$$
Observe that since $x_1,\cdots, x_{l_j}$ are a composable sequence of arrows, 
the argument as in the proof of Lemma \ref{lem:tilderisom} (see \eqref{eqn:eieiei}) tells us that the degree 3 part of $\mathsf{KS}( \xi_{v_{i,j}})$ is given as
$$ \theta_{v_{i,j}} -\theta_{v_0}=   \Theta_{v_{i,j}}.$$
Therefore $\mathsf{KS}( \xi_{v_{i,j}})$ is a lift of $ \Theta_{v_{i,j}} \in (E_2)_{\deg 3}$ to $HF^{\mathrm{odd}}_{cyc} ((\lL,b),(\lL,b))$,
$$HF^{\mathrm{odd}}_{cyc} ((\lL,b),(\lL,b)) \to (E_2)_{\deg 3},\qquad  \mathsf{KS}( \xi_{v_{i,j}}) \mapsto \Theta_{v_{i,j}}.$$

\begin{figure}[h]
\includegraphics[scale=0.55]{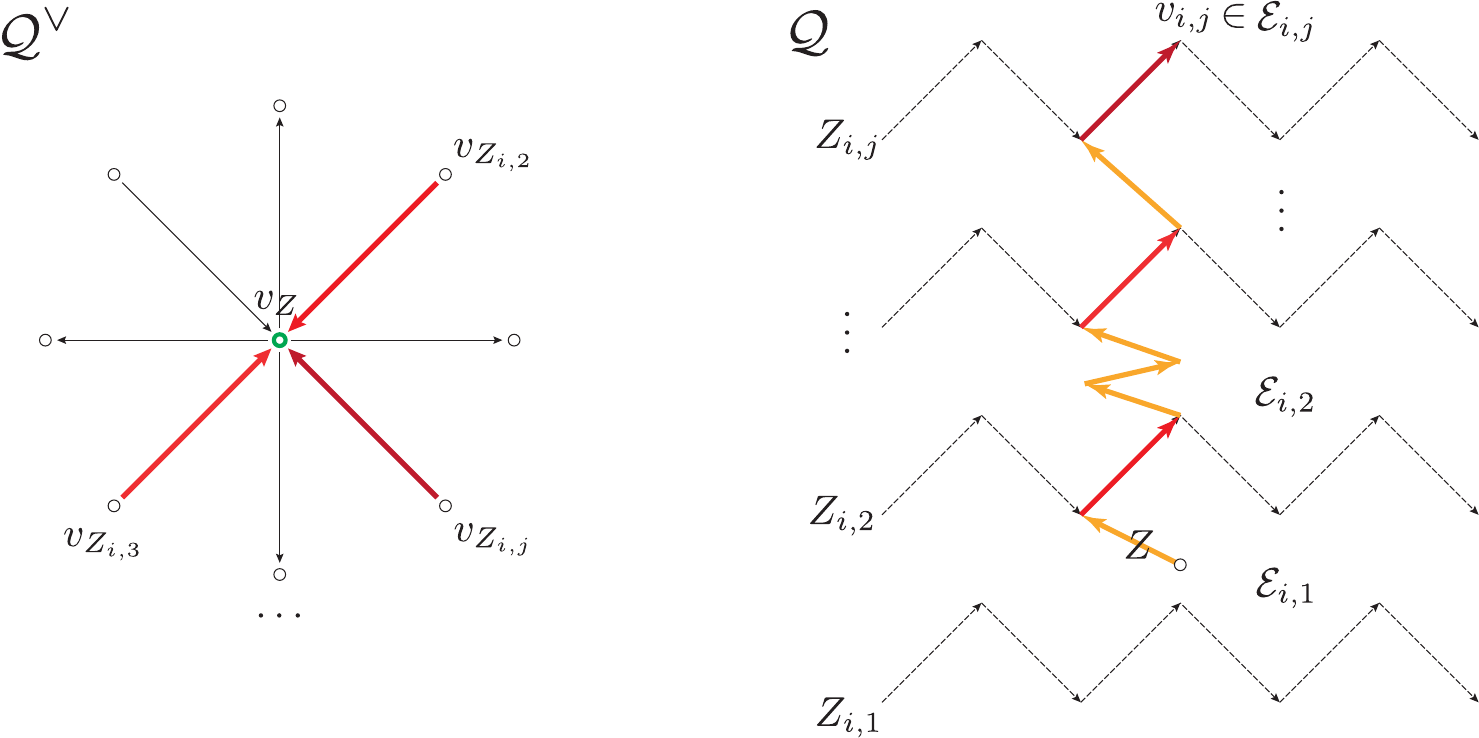}
\caption{}
\label{fig:tthetav}
\end{figure}

\begin{claim*}
We claim that 
\begin{equation}\label{eqn:newbasissh}
\{ \alpha_{i}^n (aq+bp),  \alpha_{i}^n \xi_{v_{i,j}} : 1 \leq i \leq N, 1< j \leq m_i, n \geq 1 \}
\end{equation}
is a basis of $SH_+^{\mathrm{odd}} (X)$ where $a,b$ are coefficients appearing in Remark \ref{rmk:higerdegodd}, and $\alpha_i$ is the sum of degree 0 primitive orbits winding once around $v_{Z_{i,j}}$ for all $j$ as in \eqref{eqn:alphaiii}. 
\end{claim*}

\begin{proof}[Proof of Claim]
Let $\tilde{\beta}_{i,j}$ be the odd degree Hamiltonian orbit that is in the same homology class with $\alpha_{i,j}$ for $j>2$. Then, using the explicit product structure described in  \ref{subsec:shrs}, we have
\begin{equation}\label{eqn:betatildeij}
\alpha_{i,j'} \cdot \xi_{v_{i,j}} = \left\{
\begin{array}{cl}
 \tilde{\beta}_{i,j'} & j' \leq j\\
 0 & \textnormal{otherwise}.
\end{array}\right.
\end{equation}
We denote
$ \beta_{i,j} := \tilde{\beta}_{1,1} + \cdots + \tilde{\beta}_{i,j}$
for $1 \leq j \leq m_i -1$, and
$ \beta_i :=  \beta_{1,1} + \cdots + \beta_{i,m_i}.$
It is easy to see from the explicit ring structure of $SH^\ast(X)$  that
$$ \{\alpha_i^n \beta_i, \alpha_i^n \beta_{i,j} : 1 \leq j \leq m_i -1, n \geq 0 \}$$
forms a basis of $SH^{\mathrm{odd}}_+ (X)$. It follows from \eqref{eqn:betatildeij} that 
$$ \alpha_i^n \cdot \xi_{v_{i,j}} = \alpha_i^{n-1} \cdot \beta_{i,j}$$
for $n \geq 1$.

We now show that \( \alpha_i^n \beta_i \) is a nonzero scalar multiple of \( \alpha_i^n (a p + b q) \), which completes the proof of the lemma. It suffices to consider the case $n=0$. Let us first compute $\alpha_i \cdot p$. This is a linear combination of $\tilde{\beta}_{i,j}$'s whose coefficient is the signed count of Morse trajectories from one of critical points in $p$ asymptotic to the puncture $\alpha_{i,j}$. We denote this number by $r_{i,j;p}$ so that
$\alpha_i \cdot p = \sum_j r_{i,j;p} \tilde{\beta}_{i,j}$.

Recall $p = \sum_j p_{i_0,j}$, and each $p_{i_0,j}$ is the sum of $p_e$ for $e$ incident to $v_{Z_{i_0,j}}$.
Hence, a Morse trajectory from $p$ to $\alpha_{i,j}$ exists whenever there is an edge $e$ joining $v_{Z_{i,j}}$ and $v_{Z_{i_0,j'}}$  for some $1 \leq j' \leq m_{i_0}$. On the dual dimer $\Q$, the corresponding edge $e$ is nothing but an intersection of zigzag cycles $Z_{i,j}$ and $Z_{i_0,j'}$. By the construction of $h$ (see the proof of Lemma \ref{lem:pquvks}, especially Figure \ref{fig:noptone}), the sign of the  Morse trajectory is positive (resp. negative) if $e$ is a zig (resp. a zag) of $Z_{i_0,j'}$.
\footnote{In fact, the zigzag consistency condition implies that if $Z_{i,j}$ intersects 
$Z_{i_0,j'}$ at a zig (resp. a zag), then all other intersections 
between them must also occur at a zig (resp. a zag).  
Moreover, whether the intersection is a zig or a zag is determined solely 
by the sign of the intersection of their homology classes.}


By Theorem \ref{thm:combinatorics of matching polytope for dimers in tori},  $V = \partial_{\cP_{i_0-1}} - \partial_{\cP_{i_0}}$ (after cancellation of the common edges in $\partial_{\cP_{i_0}}$ and $\partial_{\cP{i_0 -1}}$) consist of zigs of $Z_{i_0,j'}$ with positive signs and zags of $Z_{i_0,j'}$ with negative signs (for all $j'$). Thus we have 
$$r_{i,j;p} 
= m_{i_0} V(\eta_i).$$ 
where the factor \( m_{i_0} \) appears since the above occurs for each \( Z_{i_0, j'} \) with \( 1 \leq j' \leq m_{i_0} \).
Therefore
$$ \alpha_{i,j} \cdot p = m_{i_0} V(\eta_i) \tilde{\beta}_{i,j}$$
Proceeding with similar argument, we have
$$\alpha_{i,j} \cdot (aq+ bp) = m_{i_0} (aU(\eta_i) + b V(\eta_i)) \tilde{ \beta}_{i,j}$$
which implies 
$$\alpha_i  \cdot (ap+bq) = \left(\sum_{j=1}^{m_i} \alpha_{i,j} \right)\cdot (aq +bp) =  m_{i_0} (aU(\eta_i) + b V(\eta_i)) \cdot \sum_{j=1}^{m_i} \tilde{\beta}_{i,j} = c  \beta_i,$$
with $c = m_{i_0} (aU(\eta_i) + b V(\eta_i))$ which is nonzero by our choice of $a$ and $b$ in Remark \ref{rmk:higerdegodd}. Therefore \eqref{eqn:newbasissh} is a basis of $SH^{\mathrm{odd}}_+ (X)$.
\end{proof}

Now, suppose 
$$\beta:= \sum_i c_i \alpha_{i}^{n_i} (aq+bp) + \sum_{i,j} c_{i,j}\alpha_{i}^{n_{i,j}} \xi_{v_{i,j}}$$
for some $c_i, c_{i,j} \in \mathbb{C}$ and $n_i, n_{i,j} \in \mathbb{Z}_{\geq 1}$ belongs to the kernel of $\overline{KS}: SH_+^{\mathrm{odd}} (X) \to \overline{HF}$ \eqref{eqn:ksinducedoddinj}. Then $\varphi (\beta) = \sum_{i,j} c_{i,j} x_{\eta_i}^{n_{i,j}} \Theta_{v_{i,j}} =0$ must vanish, where $\varphi$ is defined as the composition
\begin{equation*}
\xymatrix{ 
&&SH_+^{\mathrm{odd}} (X) \ar[d]_{\overline{KS}} \ar[dr]^{\varphi}&&  \\
0 \ar[r]  &\overline{(E_2)_{\deg 1}} \ar[r]  & \overline{HF} \ar[r]  &   \overline{(E_2)_{\deg 3}} \ar[r]  &      0}.
 \end{equation*}
By their linear independence, we have $c_{i,j}=0$ for all $i,j$. We see that the image of $\beta$ lies in $\overline{(E_2)_{\deg 1}} $, and is indeed given by $\sum_i c_i x_{\eta_i} W$, which should vanish as well. Again, by linear independence, we have $c_i=0$, and hence $\beta=0$.
\end{proof}

Continuing from Remark~\ref{rmk:discrem2}, several multiplicative relations on 
$$HH^\ast_c(\mf) = \HFbloop$$ 
stated in \cite[Theorem~5.3]{Wong21} no longer hold when the more natural product $m_2^{b,b,b}$ is used---that is, the product for which the Kodaira--Spencer map $\mathsf{KS}$ in \eqref{eqn:KSisom} becomes an isomorphism.  For instance, $\Psi_{i,j} \cdot \Psi_{i,j'}$ is no longer zero with respect to $m_2^{b,b,b}$.  
(The precise value of this product can be determined once one chooses liftings of these elements (in $E_2$) to $\HFbloop$, or equivalently, a splitting of \eqref{eqn:sesforeven}.)

\subsection{Singularities of $W$ occurring along discriminant loci of $Y_\Q$}\label{subsec:KSmapandsing}
For a given \emph{consistent dimer model} $\mathcal{Q}$ embedded in the two-torus $T^2$, let $Y_{\mathcal{Q}}$
denote the associated toric Gorenstein singularity, so that the \emph{Jacobian algebra} $\Jac(\mathcal{Q})$ provides a noncommutative crepant resolution (NCCR) of $Y_{\mathcal{Q}}$. 
The \emph{singular locus} of $Y_{\mathcal{Q}}$ lies along its toric strata of codimension $\ge 2$, which can be read directly from its \emph{toric fan}. 
Note that the fan of $Y_{\mathcal{Q}}$ is simply the cone over the matching polytope $MP(\mathcal{Q})$.

We are interested in the Landau--Ginzburg model on $Y_{\mathcal{Q}}$ with the superpotential
\[
W \in \mathcal{Z} (\Jac) \cong \mathbb{C}[Y_{\mathcal{Q}}],
\]
viewed as a regular function on $Y_{\mathcal{Q}}$. 
Locally on a toric crepant resolution
$\widetilde{Y_{\mathcal{Q}}} \longrightarrow Y_{\mathcal{Q}}$, 
the function $W$ lifts to the product of three affine coordinates. 
Hence its  critical locus coincides precisely with the toric strata of codimension $\ge 2$ in $\widetilde{Y_{\mathcal{Q}}}$.
The map $\widetilde{Y_{\mathcal{Q}}} \to Y_{\mathcal{Q}}$ identifies certain \emph{noncompact codimension-2 strata} (each isomorphic to $\mathbb{C}$), and \emph{contracts} all the \emph{compact codimension-2 strata} (each isomorphic to $\mathbb{P}^1$).

If an edge of $MP(\mathcal{Q})$ contains interior lattice points, then the corresponding codimension-2 stratum acquires an \emph{orbifold singularity}, whose order equals the number of such lattice points. 
In our setup, the edge of $MP(\mathcal{Q})$ normal to $\eta_i$ contains $(m_i - 1)$ interior lattice points. 
This implies that $m_i$ irreducible codimension-2 components (each $\cong \mathbb{C}$) are identified under the map $\widetilde{Y_{\mathcal{Q}}} \to Y_{\mathcal{Q}}$. 
This identification is precisely captured by the element $\Psi_{i,j}$ (for $1 < j \leq m_i$). 
For example, if there were no singularities along the codimension-2 strata, then all even cocycles in $\HFbloop$ would be scalar multiples of the unit.

Similarly, if the elementary (normalized) area of $MP(\mathcal{Q})$ exceeds $1$, then the vertex (the unique torus fixed point) becomes singular. 
This can be detected by the presence of nontrivial elements $\Theta_v$. 
The elementary area of $MP(\mathcal{Q})$ equals the absolute value of the Euler characteristic of the mirror curve $X_\Qv$. 
Hence, it is natural that the number of irreducible components in $\mathbb{L}$ increases as the singularity at the torus fixed point of $Y_{\mathcal{Q}}$ becomes deeper.

\appendix

\section{Comparison between two resolutions of $J$}\label{sec:comparison}

Our goal is to find an explicit formula of the $BV$ operator $\Delta$ on the Hochschild cohomology of $J=\Jac$
$$HH^\ast(J,J) = \hom_{J\textnormal{-bimod}} (J^\bullet, J)$$ 
when applying to the degree $3$ elements, where $J^\bullet$ is the length $3$ resolution introduced in \ref{subsec:koszulresolJ}. The original definition of the $BV$ operator uses the Connes operator on $HH_\ast(J,J)$ which is described in terms of the bar resolution.
We use the following comparison theorem from elementary homological algebra to compare the standard bar resolution of $J = \Jac$ and $J^\bullet$. 

\begin{thm}[Comparison Theorem]\label{eqn:compthm}
For $R-S$ bimodules $M$ and $N$, let $ \cdots P_2 \to P_1 \to P_0 \to M$ be a complex with all $P_i$ projective, and $\cdots A_2  \to A_1 \to A_0 \to N$ an acyclic complex. Then for any $f \in \hom_{R-S} (M,N)$, there exists a chain map $f_\bullet : P_\bullet \to A_\bullet$ that extends $f$, i.e., $f_{-1} = f$. Moreover, any two such liftings are homotopic.
\end{thm}

Thus we need to find a chain map $f_\bullet$ that lifts the identify map $\underline{J} \to \underline{J}$. We present one concrete lifting here.
\begin{equation*}
\xymatrix{0  \ar[r] \ar[d]^{f_{\geq 4} \equiv 0}&(J \otimes_{\mathbb{C}} J)^\Bbbk \ar[r]^j \ar[d]^{f_3}& J \otimes_\Bbbk E^\ast \otimes_\Bbbk J \ar[r]^{c} \ar[d]^{f_2} &J \otimes_\Bbbk E \otimes_\Bbbk J \ar[r]^{\quad j^\vee} \ar[d]^{f_1} &J \otimes_\Bbbk J \ar[r] \ar[d]^{f_0} &\underline{J} \ar[d]^{f_{-1}=id} \\
  \cdots \ar[r]& J \otimes_{\Bbbk} \underline{J}^{\otimes 3} \otimes_{\Bbbk} J \ar[r]_{b_3}& J \otimes_{\Bbbk} \underline{J}^{\otimes 2} \otimes_{\Bbbk} J \ar[r]_{b_2}& J  \otimes_{\Bbbk} \underline{J} \otimes_{\Bbbk} J \ar[r]_{\quad b_1}&   J \otimes_{\Bbbk} J \ar[r]&\underline{J} }
\end{equation*}
First few maps are obvious:
$$f_0 = id : J \otimes_\Bbbk J \to J \otimes_\Bbbk J,$$
$$ f_1 =incl. : J \otimes_\Bbbk E \otimes_\Bbbk J \to J \otimes_\Bbbk \underline{J} \otimes_\Bbbk J,$$
and we set $f_{\geq 4} \equiv 0$.

%

We next define $f_2$ as follows:
\begin{equation*}
\begin{array}{rcl}
 f_2 : J \otimes_\Bbbk E^\ast \otimes_\Bbbk J &\to& J \otimes_\Bbbk \underline{J}^{\otimes 2} \otimes_\Bbbk J \\
a \otimes_\Bbbk \bar{x} \otimes_\Bbbk b &\mapsto& \sum_y a \left( \frac{\partial^2 \Phi}{\partial x \partial y} \right)' \otimes_\Bbbk y \otimes_\Bbbk   \left( \frac{\partial^2 \Phi}{\partial x \partial y} \right)''  \otimes_\Bbbk  b .
\end{array}
\end{equation*}
In other words, $f_2   (a \otimes_\Bbbk \bar{x} \otimes_\Bbbk b) = c(a \otimes_\Bbbk \bar{x} \otimes_\Bbbk 1) \otimes_\Bbbk b$.
%
%
To see this is a chain map, 
\begin{equation*}
\begin{array}{rcl}
b_2 ( f_2 (a \otimes_\Bbbk \bar{x} \otimes_\Bbbk b)) &=&  b_2 \left( c(a \otimes_\Bbbk \bar{x} \otimes_\Bbbk 1)) \otimes_\Bbbk b \right)  \\
&=& \overbrace{b_1 (c(a \otimes_\Bbbk \bar{x} \otimes_\Bbbk 1))}^{=j^{\vee} (c(a \otimes_\Bbbk \bar{x} \otimes_\Bbbk 1))=0} \otimes_\Bbbk b + c(a \otimes_\Bbbk \bar{x} \otimes_\Bbbk 1) b \\
&=& c(a \otimes_\Bbbk \bar{x} \otimes_\Bbbk b).
\end{array}
\end{equation*}
Finally, we take $f_3 : (J \otimes_\mathbb{C} J)^\Bbbk \to J \otimes_\Bbbk \underline{J}^{\otimes 3} \otimes_\Bbbk J$ to be
\begin{equation*}
\begin{array}{rcl}
f_3(a\otimes_\mathbb{C} b) &:=& \displaystyle\sum_x \sum_y  b \left(\frac{\partial^2 \Phi}{\partial x \partial y} \right)' \otimes_\Bbbk y \otimes_\Bbbk  \left( \frac{\partial^2 \Phi}{\partial x \partial y} \right)'' \otimes_\Bbbk x    \otimes_\Bbbk a \\
&=&  \displaystyle\sum_x c(b \otimes_\Bbbk \bar{x} \otimes_\Bbbk 1) \otimes_\Bbbk x \otimes_\Bbbk a.
\end{array}
\end{equation*}
Note that using $j^\vee \circ c (=b_1 \circ c)=0$, we have
\begin{equation*}
\begin{array}{rcl}
 b_3 ( f_3(a\otimes_\mathbb{C} b) ) &=& \displaystyle\sum_x\left( b \cdot c(1 \otimes_\Bbbk \bar{x} \otimes_\Bbbk 1) \cdot x \otimes_\Bbbk a - c(b \otimes_\Bbbk \bar{x} \otimes_\Bbbk 1) \otimes_\Bbbk xa \right) \\
 &=& \displaystyle\sum_x\left( b \cdot c( 1 \otimes_\Bbbk \bar{x} \otimes_\Bbbk x)   \otimes_\Bbbk a - c(b \otimes_\Bbbk \bar{x} \otimes_\Bbbk 1) \otimes_\Bbbk xa \right) \\
  &=& \displaystyle\sum_x\left( b \cdot c( x \otimes_\Bbbk \bar{x} \otimes_\Bbbk 1)   \otimes_\Bbbk a - c(b \otimes_\Bbbk \bar{x} \otimes_\Bbbk 1) \otimes_\Bbbk xa \right) \\
 &=& \displaystyle\sum_x\left( b \cdot f_2 (x \otimes_\Bbbk \bar{x} \otimes_\Bbbk a) -  f_2 (b \otimes_\Bbbk \bar{x} \otimes_\Bbbk xa) \right)\\
 &=& f_2 \left( \sum_x bx \otimes_\Bbbk \bar{x} \otimes_\Bbbk a -  b \otimes_\Bbbk \bar{x} \otimes_\Bbbk xa \right) = f_2 (j (a\otimes_\mathbb{C} b)).
 \end{array}
\end{equation*}
where we used in the third line
$x  \cdot  c(1_\Bbbk \otimes \bar{x} \otimes_\Bbbk 1) =  c(1_\Bbbk \otimes \bar{x} \otimes_\Bbbk 1)  \cdot  x$
which is equivalent to $c \circ j =0$.

\subsection{Homotopy inverse of $f$}\label{subsec:g1g1}

By comparison theorem, there also exists a homotopy inverse $g$ of $f$ that fits into
\begin{equation*}
\xymatrix{0  \ar[r] \ar[d]^{f_{\geq 4} \equiv 0}&(J \otimes_{\mathbb{C}} J)^\Bbbk \ar[r]^j \ar[d]^{f_3}& J \otimes_\Bbbk E^\ast \otimes_\Bbbk J \ar[r]^{c} \ar[d]^{f_2} &J \otimes_\Bbbk E \otimes_\Bbbk J \ar[r]^{\quad j^\vee} \ar[d]^{f_1} &J \otimes_\Bbbk J \ar[r] \ar@{=}[d]^{f_0=id} &\underline{J} \ar@{=}[d]^{f_{-1}=id} \\
  \cdots \ar[r] \ar[d]^{ 0}& J \otimes_{\Bbbk} \underline{J}^{\otimes 3} \otimes_{\Bbbk} J \ar[r]_{b_3} \ar[d]^{g_3}& J \otimes_{\Bbbk} \underline{J}^{\otimes 2} \otimes_{\Bbbk} J \ar[r]_{b_2}  \ar[d]^{g_2}& J  \otimes_{\Bbbk} \underline{J} \otimes_{\Bbbk} J \ar[r]_{\quad b_1} \ar[d]^{g_1}&   J \otimes_{\Bbbk} J \ar[r] \ar@{=}[d] &\underline{J} \ar@{=}[d]   \\
  0  \ar[r] &(J \otimes_{\mathbb{C}} J)^\Bbbk \ar[r]^j  & J \otimes_\Bbbk E^\ast \otimes_\Bbbk J \ar[r]^{c}  &J \otimes_\Bbbk E \otimes_\Bbbk J \ar[r]^{\quad j^\vee}   &J \otimes_\Bbbk J \ar[r]  &\underline{J} }.
\end{equation*}
For our purpose, it is enough to choose an explicit $g_1$.
It depends on a choice of a lifting $\sigma: J \to \mathbb{C}Q$, and we fix one here. 
Then we define $g_1$ to be a composition of $\sigma$ and the following map
$$ \mathfrak{c}: \mathbb{C}Q \to J \otimes_\Bbbk E \otimes_\Bbbk J \qquad x_1 x_2 \cdots x_k \mapsto \sum_i x_1 \cdots x_{i-1} \otimes_\Bbbk x_i \otimes_\Bbbk x_{i+1} \cdots x_k.$$
More precisely, $g_1$ is given as the composition of sequence of maps
\begin{equation}\label{eqn:g1g1g1}
g_1 : J \otimes_\Bbbk \underline{J} \otimes_\Bbbk J \to J \otimes_\Bbbk \mathbb{C} Q \otimes_\Bbbk J \to (J \otimes_\Bbbk \mathbb{C}Q) \otimes_\Bbbk E  \otimes_\Bbbk (\mathbb{C}Q \otimes_\Bbbk J) \to J \otimes_\Bbbk E \otimes_\Bbbk J
\end{equation}
where the first map and second map are $1\otimes_\Bbbk \sigma \otimes_\Bbbk 1$ and $1\otimes_\Bbbk \mathfrak{c} \otimes_\Bbbk 1$, and the last map is the projection $\mathbb{C} Q \to J$ (composed with the multiplication $J \otimes_\Bbbk J \to J$). 

\subsection{Proof of Lemma \ref{lem:bvopex}}\label{subsec:appbvformula}

We start by the Connes differential $B$ on the Hochschild homology $HH_\ast (J,J)$ which dualizes to the BV operator $\Delta$ via \eqref{eqn:hochdualcy},
\begin{equation*}
\xymatrix{  HH_{3-\ast} (J,J) \ar[r]^{B} \ar[d]_{\cong} &HH_{4-\ast} (J,J) \ar[d]^{\cong} \\
HH^\ast (J,J) \ar[r]^{\Delta} & HH^{\ast-1} (J,J) }
\end{equation*}
$B$ is originally described in terms of  the bar resolution $P^\bullet:= J \otimes_\Bbbk\underline{J}^{\otimes \bullet} \otimes_{\Bbbk} J $, and hence, is an operator on the standard bar complex (another model for the Hochschild chain complex)
$$ J \otimes_{J\textsf{-Bimod}} P^\bullet = \oplus_{n \geq 0} \left( J \otimes_\Bbbk \underline{J}^{\otimes n} \right)_{cyc}.
$$ 
Specifically, the Connes differential $B$ is defined by
\begin{equation}\label{eqn:Connes}
\begin{array}{rcl}
 B: \left( J \otimes_\Bbbk \underline{J}^{\otimes n} \right)_{cyc} &\longrightarrow & \left( J \otimes_\Bbbk \underline{J}^{\otimes n+1} \right)_{cyc} \\
 a_0 \otimes_\Bbbk a_1 \otimes_\Bbbk \cdots \otimes_\Bbbk a_n &\mapsto&  \sum_{i=0}^{n} (-1)^{in} \, 1 \otimes a_i \otimes \cdots \otimes a_n \otimes a_0 \otimes \cdots \otimes a_{i-1}.
\end{array}
\end{equation}

\begin{proof}[Proof of Lemma \ref{lem:bvopex}]
The case of our interest is the degree $3$ Hochschild cocycle in $HH^\ast(J,J)$, and hence, dually, we need to look at the Connes operator applying to the Hochschild cycle of degree $0$ (i.e., $n=0$ in \eqref{eqn:Connes}):
$$ B:\underline{J}_{cyc} \to (J \otimes_\Bbbk \underline{J})_{cyc} \qquad a  \mapsto 1 \otimes_\Bbbk a$$ 
for a cyclic path $J_{cyc}$. We want to read off in terms of the Koszul resolution (or the associated Hochschild chain complex with coefficient $J$). 
That is, we  find $\tilde{B}$ that commutes the diagram
\begin{equation*}
\xymatrix{   \underline{J} \otimes_{J\textsf{-Bimod}} (J \otimes_\Bbbk J)  \ar@{=}[d] \ar[r]^{\hspace{-0.6cm}\tilde{B}} &  \underline{J} \otimes_{J\textsf{-Bimod}} (J \otimes_\Bbbk E \otimes_\Bbbk J)  \ar@<-0.5ex>[d]_{id \otimes f_1}  \\
\underline{J} \otimes_{J\textsf{-Bimod}}(J \otimes_\Bbbk J) \ar[r]_{\hspace{-0.6cm} B} &
\underline{J} \otimes_{J\textsf{-Bimod}}(J \otimes_\Bbbk \underline{J} \otimes_\Bbbk J)  \ar@<-0.5ex>[u]_{id \otimes g_1} 
} \qquad \cong \qquad
\xymatrix{ \underline{J}_{cyc}  \ar@{=}[d] \ar[r]^{\hspace{-0.7cm}\tilde{B}}&  (\underline{J} \otimes_\Bbbk E)_{cyc} \ar@<-0.5ex>[d]_{id \otimes f_1}   \\
  \underline{J}_{cyc} \ar[r]_{\hspace{-0.7cm}B} & (\underline{J} \otimes_\Bbbk \underline{J})_{cyc}   \ar@<-0.5ex>[u]_{id \otimes g_1} 
}
\end{equation*}
where vertical maps are obtained by the comparison map between $J^\bullet$ and $P^\bullet$\begin{equation*}
\xymatrix{ J^\bullet:  \quad J \otimes_\Bbbk E \otimes_\Bbbk J \ar[r] \ar@<3.5ex>[d]_{f_1} & J \otimes_\Bbbk J \ar@{=}[d] \\
P^\bullet : \quad J \otimes_\Bbbk \underline{J} \otimes_\Bbbk J \ar[r] \ar@<-4.5ex>[u]_{g_1}  & J \otimes_\Bbbk J
} 
\end{equation*}
given in \eqref{eqn:g1g1g1}.
For this purpose, we only need to compute the image of $B(a)=1 \otimes_\Bbbk a$ under the comparison map $id \otimes g_1$, i.e.,
$$ \tilde{B} (a) = id \otimes g_1 (1 \otimes_\Bbbk a).$$
By definition, this can be computed from (any of) its lifting $\sigma(a)=x_1 x_2 \cdots x_k \in \mathbb{C} \Q$ as in \ref{subsec:g1g1}. Since $a$ is a cyclic path, we have $h(x_k) = t(x_1)$. 

Identifying $1 \otimes_\Bbbk a = 1 \otimes_{J\textsf{-Bimod}} (1 \otimes_\Bbbk a \otimes_{\Bbbk} 1)$ under $(\underline{J} \otimes_\Bbbk \underline{J})_{cyc}  \cong  \underline{J} \otimes_{J\textsf{-Bimod}} (J \otimes_\Bbbk \underline{J} \otimes_\Bbbk J)$, the image $id \otimes g_1 (1 \otimes_\Bbbk a)$ is given by
\begin{equation*}
\xymatrix{ \displaystyle\sum_i (x_{i+1} \cdots x_k )(x_1 \cdots x_{i-1} )\otimes_\Bbbk x_i   \ar@{=}[r]&  1 \otimes_{J\textsf{-Bimod}}
 \left(\displaystyle\sum_{i} x_1 \cdots  x_{i-1} \otimes_\Bbbk x_i \otimes_{\Bbbk} x_{i+1} \cdots x_k \right) \\
B(a)=1 \otimes_\Bbbk a   \ar@{=}[r] \ar[u]& 1 \otimes_{J\textsf{-Bimod}} (1 \otimes_\Bbbk a \otimes_{\Bbbk} 1)  \ar[u]^{id \otimes g_1} 
 }
 \end{equation*}
where $x_i$ denotes its projection in $J$ by abuse of notation, and the equality at the bottom is from the identification $(\underline{J} \otimes_\Bbbk E)_{cyc} \cong  \underline{J} \otimes_{J\textsf{-Bimod}} (J \otimes_\Bbbk E \otimes_\Bbbk J) $.

By duality \ref{eqn:hochdualcy}, this immediately implies that $$\Delta:HH^3 (J,J) \to HH^2 (J,J)$$ is given by
$$\Delta ( a ) =    \displaystyle\sum_i (x_{i+1} \cdots x_k )(x_1 \cdots x_{i-1} )\otimes_\Bbbk \bar{X_i}$$ 
where we view $a \in \underline{J}$ as a degree 3 Hochschild cochain (cocycle) this time. 

\end{proof}
%
%

\bibliographystyle{amsalpha}
\bibliography{geometry}

\end{document}